\documentclass[a4paper,10pt]{amsart}
\usepackage[utf8]{inputenc}
\usepackage{amsthm}
\usepackage{amsmath}
\usepackage{bm}
\usepackage{enumitem}
\usepackage{amsfonts}
\usepackage{amssymb}
\usepackage{mathtools}
\usepackage{appendix}
\usepackage{mathrsfs}
\usepackage{setspace}
\usepackage{comment}
\usepackage{xcolor}
\usepackage{todonotes}
\usepackage{aliascnt}
\usepackage[foot]{amsaddr}
\usepackage[pdfdisplaydoctitle,colorlinks,breaklinks,urlcolor=blue,linkcolor=blue,citecolor=blue]{hyperref} 

\mathtoolsset{showonlyrefs}
\newcommand{\A}{\mathbb{A}}

\newcommand{\EE}{\mathbb{E}}

\renewcommand{\H}{\mathcal{H}}

\newcommand{\I}{\mathcal{I}}

\newcommand{\N}{\mathbb{N}}

\newcommand{\R}{\mathbb{R}}

\newcommand{\T}{\mathbb{T}}

\newcommand{\W}{\mathbf{W}}

\newcommand{\sx}{\left}
\newcommand{\dx}{\right}

\newcommand{\var}{\textup{-var}}
\newcommand{\uno}{\mathbf{1}}

\DeclareMathOperator{\dvg}{\text{div}}

\newcommand{\st}{_{s,t}}
\newcommand{\Y}{\mathbf{Y}}

\newcommand{\one}{\bm{1}}

\newcommand{\norm}[1]{\left\|#1\right\|}

\newcommand{\expt}[2][]{\mathbb{E}_{#1}\left[#2\right]}

\newtheorem{theorem}{Theorem}[section]

\newaliascnt{definition}{theorem}
\newtheorem{definition}[definition]{Definition}
\aliascntresetthe{definition}

\newaliascnt{hypothesis}{theorem}

\aliascntresetthe{hypothesis}

\newaliascnt{corollary}{theorem}
\newtheorem{corollary}[corollary]{Corollary}
\aliascntresetthe{corollary}

\newaliascnt{lemma}{theorem}
\newtheorem{lemma}[lemma]{Lemma}
\aliascntresetthe{lemma}

\newaliascnt{proposition}{theorem}
\newtheorem{proposition}[proposition]{Proposition}
\aliascntresetthe{proposition}

\theoremstyle{remark}
\newaliascnt{remark}{theorem}
\newtheorem{remark}[remark]{Remark}
\aliascntresetthe{remark}

\numberwithin{equation}{section}

\title[Fast-Slow Navier-Stokes System Driven by fBM]{Averaging Dynamics and Wong-Zakai approximations for a Fast-Slow Navier-Stokes System Driven by fractional Brownian Motion}
\author[E. Luongo]{Eliseo Luongo}
\address{Fakultät für Mathematik, Universität Bielefeld, 33501 Bielefeld, Germany\newline
Current Address: Institut für Mathematik, Technische Universität Berlin, 10623 Berlin, Germany}
\email{\href{mailto:eliseo.luongo at tu-berlin.de}{eliseo.luongo at tu-berlin.de}}
\author[F. Triggiano]{Francesco Triggiano}
\address{Scuola Normale Superiore, Piazza dei Cavalieri, 7, 56126 Pisa, Italia \newline Current Address: Institute of Mathematics, EPFL, Lausanne, Switzerland}
\email{\href{mailto:francesco.triggiano at epfl.ch}{francesco.triggiano at epfl.ch}}
\thanks{\textbf{Corresponding Author:} Francesco Triggiano, \href{mailto:francesco.triggiano at epfl.ch}{francesco.triggiano at epfl.ch}.}
\date\today

\allowdisplaybreaks
\begin{document}
\begin{abstract}
We study a Fast-Slow system of coupled two- and three-dimensional Navier-Stokes equations in which the fast component is perturbed by an additive fractional Brownian noise with Hurst parameter $\H>\frac{1}{3}$. Using rough path theory, we establish subsequential
convergence in law of the slow component, with limiting dynamics
depending on the Hurst parameter and the noise scaling. For $\H<\frac12$, an averaging regime yields a deterministic
Navier-Stokes equation with an additional It\^o-Stokes drift.
Under a stronger noise scaling, corresponding to a Wong-Zakai approximation, and a centering condition on the
quadratic fast dynamics, we obtain a Navier-Stokes equation
driven by fractional transport noise. For $\H>\frac12$, the transport-noise limit holds in the Wong-Zakai regime without this centering condition, with
integration understood in the Young sense.
\end{abstract}

\maketitle

\section{Introduction}\label{sec:introduction}
In recent years, Fast-Slow systems associated with fluid equations have attracted considerable attention; see, for instance, \cite{flandoli20212d, flandoli2022additive, debussche2024second,debussche2023rough}. The paradigmatic example to have in mind is provided by the stochastic Navier–Stokes equations on the three-dimensional torus, given by
\begin{align}\label{intro_syst_h12}
    \begin{cases}
    du^{\epsilon}=\nu \Delta u^{\epsilon}dt-(u^{\epsilon}+v^{\epsilon})\cdot\nabla u^{\epsilon}dt+\nabla p^{\epsilon}\,dt,\\
    dv^{\epsilon}=\left(\nu\Delta v^{\epsilon}+\epsilon^{-1}Cv^{\epsilon}-(u^{\epsilon}+v^{\epsilon})\cdot\nabla v^{\epsilon}+\nabla q^{\epsilon}\right)\,dt+\epsilon^{-1}dW_t,\\
    \nabla\cdot u^{\epsilon}=0,\quad \nabla\cdot v^{\epsilon}=0,
\end{cases}  
\end{align}
where $W$ is an infinite-dimensional Brownian motion satisfying suitable spatial regularity assumptions, $\epsilon\ll 1$ is a small parameter governing the separation of scales between the \emph{large} scales $u^{\epsilon}$ and the \emph{small} scales $v^{\epsilon}$, and $C$ is a suitable damping operator. For simplicity, we restrict this heuristic discussion to the case $C=-I$. In the aforementioned works, it was shown that, up to extraction of a subsequence, the large-scale component $u^\epsilon$ converges in distribution, as $\epsilon\rightarrow 0$, to a process $u$ which is a weak solution to the following stochastic Navier-Stokes equations:
\begin{align}\label{intro_transport_h12}
    \begin{cases}
    du=\left(\nu \Delta u-u\cdot\nabla u-\overline{r}\cdot\nabla u\right) dt+d\nabla p-\circ dW\cdot\nabla u,\\
    \nabla\cdot u=0,
\end{cases}  
\end{align}
where $\overline{r}$ is a deterministic vector field, referred to in the literature as the It\^o-Stokes drift (cf. \cite{memin2014fluid}, \cite{debussche2023consistent}), and is related to the spatial covariance structure of the noise $W$. Moreover, $p$ is a semimartingale, often called the turbulent pressure, and the stochastic integral is understood in the Stratonovich sense. 

The interest in results of this type is twofold. On the one hand, they are relevant for stochastic model reduction techniques in climate dynamics; see, for example, \cite{majda2001mathematical, franzke2005low, franzke2006low}. On the other hand, they provide a rigorous justification for the emergence of transport-type noise in fluid equations such as \eqref{intro_transport_h12}. Alternative approaches to those developed in \cite{debussche2024second} can be found in \cite{holm2015variational, memin2014fluid}. Over the last decade, starting from \cite{flandoli2010well}, models like \eqref{intro_transport_h12} have received increasing attention due to their regularizing effects (cf. \cite{maurelli_et_al, flandoli2021high, coghi2023existence} and the references therein), as well as their ability to capture key features of turbulent flows, including enhanced and anomalous dissipation phenomena (see for instance \cite{MR1905858, flandoli2024quantitative, luo2024elementary, rowan2024anomalous, drivas2025anomalous}) and dynamo effects (cf. \cite{butori2024mean, butori2024background}).

From a modeling perspective, however, it is important to account for non-Markovian noise with memory effects when describing turbulent fluids; see, for instance, \cite{apolinario2022dynamical, apolinario2023linear, franzke2015stochastic, faranda2014modelling, lilly2017fractional}. This has led to a growing interest in the analysis of equations of the form \eqref{intro_transport_h12}, where the Brownian motion $W$ is replaced by a fractional Brownian motion, as in \cite{hofmanova2019navier, hofmanova2021rough, crisan2022solution, roveri2024well, galeati2026well, flandoli2023reduced, cifani2025diffusion}. Nevertheless, as pointed out in \cite{flandoli20212d, flandoli2022additive, debussche2024second, debussche2023rough}, such limit equations should be regarded as idealized models arising from the Fast-Slow system \eqref{intro_syst_h12}. A rigorous analysis of the original system when the driving noise $W$ is replaced by a fractional Brownian motion $W^{\H}$ with Hurst parameter $\H$ is still missing. Indeed, all the aforementioned works crucially rely on the Markovian nature of $W$ and on the availability of It\^o-type formulas for the system \eqref{intro_syst_h12}.

Furthermore, although a substantial body of literature on fast-slow systems driven by fractional noise has recently emerged (see, for instance, \cite{li2022slow, gehringer2022functional, pei2024almost, pei2023almost, li2026navierstokesfractionaltransportnoise} and the survey \cite{gehringer2019rough}), to the best of our knowledge none of these works addresses the specific issues discussed above. The present article initiates the study of this gap, see \autoref{main_thm} and \autoref{centered_noise} below for details.

Before presenting our main contributions, we briefly discuss the origin of the terms appearing in \eqref{intro_syst_h12}, as their justification plays a central role in our framework. We start from the incompressible Navier–Stokes equations
\begin{align*}
    \begin{cases}
        \partial_t \overline{u}&=\nu\Delta \overline{u}-\overline{u}\cdot\nabla \overline{u}+\nabla \overline{p},\\
        \nabla\cdot \overline{u}&=0.
    \end{cases}
\end{align*}
We decompose the initial condition into large and small scales according to a suitable scale-separation rule,
\begin{align*}
    \overline{u}(0)=u^{\epsilon}(0)+v^{\epsilon}(0),
\end{align*}
where $\epsilon>0$ is a parameter characterizing the separation of scales. Assuming that this separation persists for positive times, we introduce two velocity fields $u^{\epsilon}(t)$ and $v^{\epsilon}(t)$, describing the evolution of the large- and small-scale components, respectively. Their dynamics are then governed by
\begin{align}\label{system_intermediate_small}
    \begin{cases}
        \partial_t u^{\epsilon}+(u^{\epsilon}+v^{\epsilon})\cdot\nabla u^{\epsilon}+\nabla p^{\epsilon}&=\nu \Delta u^{\epsilon},\\
        \partial_t v^{\epsilon}+(u^{\epsilon}+v^{\epsilon})\cdot\nabla v^{\epsilon}+\nabla q^{\epsilon}&=\nu \Delta v^{\epsilon},\\
        \nabla\cdot u^\epsilon=\nabla\cdot v^\epsilon=0.
    \end{cases}
\end{align}
At this stage, system \eqref{system_intermediate_small} does not include either noise or damping at the level of the small scales. The introduction of an additive noise term in the small-scale equation is motivated by instabilities induced by boundary irregularities or internal obstacles, which are typically not explicitly modeled in mathematical formulations based on toroidal geometries or smooth bounded domains. We refer to \cite[Chapter~5]{flandoli2023stochastic} for a detailed discussion. However, the presence of additive noise alone leads, via It\^o’s formula, to an \emph{infinite} mean energy injection into the system as $\epsilon\to 0$. To counterbalance this effect, an additional damping operator must be introduced in the small-scale equation. This modification yields precisely the structure of \eqref{intro_syst_h12}, leaving only the scaling of the noise and damping terms to be justified.

The fact that the noise and damping operators scale in the same way can be motivated in at least two different manners in the literature for systems driven by standard Brownian motion. The first argument, inspired by Wong-Zakai-type approximations, is that this scaling leads the small-scale dynamics to approximate white noise. More precisely, consider the process
\begin{equation}\label{eq_wepsilon_intro}
dw^{\epsilon}_t = -\epsilon^{-1} w^{\epsilon}_t\,dt + \epsilon^{-1}  dW_t
\end{equation}
with null initial condition.
It, formally, satisfies
\begin{align}\label{heuristic_convergence_noise}
    w^{\epsilon}_t = \epsilon^{-1} \int_0^t e^{-\epsilon^{-1}(t-s)}\, dW_s \rightarrow dW_t
\end{align}
because, upon integrating equation \eqref{eq_wepsilon_intro} in time and multiplying by $\epsilon$, we obtain
\begin{align*}
    W_t
    = \epsilon\, w^{\epsilon}_t + \int_0^t w^{\epsilon}_s\, ds
    \approx \int_0^t w^{\epsilon}_s\, ds.
\end{align*}
This heuristic can be made rigorous using, for instance, rough path techniques; see \cite{friz2014convergence}. Moreover, this argument is robust with respect to the choice of the Hurst parameter and therefore applies equally when $W$ is replaced by a fractional Brownian motion with Hurst parameter $\H$.

A second motivation is provided by multiscale analysis; see \cite[Section~2]{flandoli20212d} for the case $\H=\tfrac12$, and \autoref{sec:appendix_multiscale} for its extension to the general case considered here. In contrast to the Wong–Zakai argument, this approach crucially depends on the self-similarity properties of fractional Brownian motion, which vary with the Hurst parameter. As a consequence, from this perspective the scaling $\epsilon^{-1}$ in \eqref{intro_syst_h12} is not universal. For general values of $\H$, the multiscale analysis instead suggests the scaling $\epsilon^{-\frac12-\H}$. In particular, the Wong–Zakai and multiscale scalings coincide only in the classical Brownian case $\H=\tfrac12$.

In view of the discussion above, our primary objective is to rigorously analyze the asymptotic behavior, as $\epsilon \to 0$, of the system
\begin{align}\label{e:equation1}
\begin{cases}
     du^{\epsilon}=\nu \Delta u^{\epsilon}dt-(u^{\epsilon}+v^{\epsilon})\cdot\nabla u^{\epsilon}dt+\nabla p^{\epsilon}\,dt,\\
    dv^{\epsilon}=\left(\nu\Delta v^{\epsilon}+\epsilon^{-1}Cv^{\epsilon}-(u^{\epsilon}+v^{\epsilon})\cdot\nabla v^{\epsilon}+\nabla q^{\epsilon}\right)\,dt+\epsilon^{-\alpha}dW^{\H}_t,\\
    \nabla\cdot u^{\epsilon}=0,\quad \nabla\cdot v^{\epsilon}=0,
\end{cases}
\end{align}
focusing on the interaction between the \emph{large}- and \emph{small}-scale components under stochastic fractional perturbations. Here $\nu>0$, $p^{\epsilon}$ and $q^{\epsilon}$ are pressure fields, $u^{\epsilon},v^{\epsilon}\colon \T^d\times[0,T]\to\R^d$ denote the large- and small-scale velocity fields, respectively, and $W^{\H}$ is a divergence-free, $\H$-fractional Wiener process on $[L^2(\T^d)]^d$. The noise is spatially colored, with covariance operator $Q$, and has Hurst parameter $\H\in\left(\tfrac13,1\right)$. The operator $C$ is linear. Precise assumptions on $C$ and $Q$ are given in \autoref{sec_hp_noise}; here we only mention that $C$ and $Q$ are required to commute. Solutions to \eqref{e:equation1} are interpreted in the sense of unbounded rough drivers; see \cite{bailleul2017unbounded} and the recent review \cite{hocquet2025unbounded}. This notion is recalled in \autoref{sec_rough_formulation}.
As discussed above, there are in principle two relevant choices for the parameter $\alpha\in\{1,\tfrac12+\H\}$, corresponding to different noise intensities. For technical reasons (see the discussion following \autoref{main_thm}), for each fixed value of $\H$ we are only able to treat, in full generality of the noise coefficients, one of these two regimes. Specifically, we work with the choice $\alpha=1\wedge(\tfrac12+\H)$, which yields a Wong–Zakai regime when $\H>\tfrac12$ and a multiscale regime when $\H<\tfrac12$. As already noted, the two regimes coincide in the critical case $\H=\tfrac12$, and this fact has important consequences for the structure of the limiting equation. 

When $\H>\tfrac12$, we prove that, up to extraction of a subsequence, $u^{\epsilon}$ converges in law to a weak solution of the Young-type PDE with transport noise:
\begin{align}\label{eq_Hgreat12}
    \begin{cases}
    du=\left(\nu \Delta u-u\cdot\nabla u\right) dt+d\nabla p- d(-C)^{-1}W^{\H}\cdot\nabla u,\\
    \nabla\cdot u=0.
    \end{cases}
\end{align}
In this Wong–Zakai regime, $v^{\epsilon}$ converges, as expected, to a fractional white noise. However, in contrast to \cite{debussche2024second, debussche2023rough}, no It\^o-Stokes drift appears in the limit. This absence is not due to centering assumptions on the covariance $Q$ and the operator $C$, as in \cite{flandoli20212d, flandoli2022additive}, but rather to the specific scaling of the noise.

On the other hand, when $\H<\tfrac12$, the noise intensity is too weak to yield convergence to fractional white noise. Nevertheless, the quadratic nonlinearity $v^{\epsilon}\cdot\nabla v^{\epsilon}$ still produces a macroscopic effect in the limit dynamics. In this case, we show that $u^{\epsilon}$ converges in law to a weak solution of the deterministic PDE with an additional transport term,
\begin{align}\label{eq_Hsmall12}
    \begin{cases}
    \partial_t u=\nu \Delta u-\overline{r}\cdot\nabla u-u\cdot\nabla u+\nabla p,\\
    \nabla\cdot u=0,
    \end{cases}
\end{align}
where $\overline{r}$ denotes the It\^o-Stokes drift and depends explicitly on $C$, $Q$, and $\H$; see \autoref{lem_ito_stokes} below.

It appears that only in the case $\H=\tfrac12$, where the Wong–Zakai and multiscale regimes coincide, do both effects, transport noise and It\^o-Stokes drift, simultaneously emerge in the limiting dynamics, as observed in \cite{debussche2024second, debussche2023rough}.

We also mention that under a special centering condition, corresponding to those appearing in \cite{flandoli20212d, flandoli2022additive}, we are also able to treat the case of $\alpha=1$ and $\H<\frac{1}{2}$. In particular, we prove that, up to extraction of a subsequence, $u^{\epsilon}$ converges in law to a weak solution of the Rough PDE with transport noise \eqref{eq_Hgreat12}. Therefore, similarly as in the case $\H>1/2$ above, in this Wong–Zakai regime, $v^{\epsilon}$ converges to a fractional white noise and no It\^o-Stokes drift appears in the limit. However, this time the It\^o-Stokes drift in the limit equation disappeared due to the special symmetry assumption considered. We refer to \autoref{centered_noise} for details.

Finally, solutions to \eqref{eq_Hgreat12} (respectively, \eqref{eq_Hsmall12}) are understood in the sense of \autoref{def_rough_lim_sol_Hgreat12} (respectively, \autoref{def_rough_lim_sol_Hsmall12}).
\subsection{Methodologies and Main Results}\label{subsec_results}
Similarly to \cite{debussche2023rough}, we introduce the decomposition
\begin{align*}
    v^{\epsilon}=\epsilon^{-\alpha+\H}w^{\epsilon}+r^{\epsilon}.
\end{align*}
After projecting onto divergence-free vector fields via the Leray projector $\Pi$, denoting by $A$ the Stokes operator $A:=\Pi\Delta$, and by $b(\cdot,\cdot)$ the Navier–Stokes nonlinearity as rigorously defined in \autoref{sec_notation}, the components evolve according to the following system:
\begin{align}\label{system_prelimit}
    \begin{cases}
        \partial_t u^{\epsilon}=\nu A u^{\epsilon}+b(u^{\epsilon}+\epsilon^{-\alpha+\H}w^{\epsilon}+r^{\epsilon},u^{\epsilon}),\\
        dw^{\epsilon}=\epsilon^{-1}Cw^{\epsilon} dt+ \epsilon^{-\H}dW^\H_t, \\
        \begin{aligned}\partial_t r^{\epsilon}&=\epsilon^{-1}Cr^{\epsilon}+\nu A(\epsilon^{-\alpha+\H}w^{\epsilon}+r^{\epsilon})\\&\quad+b(u^{\epsilon}+\epsilon^{-\alpha+\H}w^{\epsilon}+r^{\epsilon},\epsilon^{-\alpha+\H}w^{\epsilon}+r^{\epsilon}),
        \end{aligned}\\
        u^{\epsilon}(0) =u_0^{\epsilon},\quad w^{\epsilon}(0) =0,\quad r^{\epsilon}(0)=v_0^{\epsilon}.
    \end{cases}
\end{align}
Solutions to this system are understood in the sense of \autoref{def_rough_sol}. We do not provide the details of the construction of solutions for fixed $\epsilon\in(0,1)$, as these follow from classical arguments. In particular, the rescaled Ornstein–Uhlenbeck process $w^{\epsilon}$ exists on any probability space supporting the $Q$-fractional Wiener process $W^{\H}$, is uniquely determined, and is measurable with respect to $W^{\H}$. Moreover, for each fixed $\epsilon\in(0,1)$ there exists a probabilistically and analytically weak solution to \eqref{system_prelimit}. That is, for every $\epsilon\in(0,1)$ there exist a probability space $(\Omega_\epsilon,\mathcal F_\epsilon,\mathbb P_\epsilon)$, a $Q$-fractional Wiener process $W^{\epsilon,\H}$, and processes $(u^{\epsilon},r^{\epsilon},w^{\epsilon})$ solving \eqref{system_prelimit} in the analytically weak sense. The proof relies on a Galerkin approximation combined with stochastic compactness arguments. In dimension $d=2$, where uniqueness holds as in the deterministic case, probabilistically weak existence can in fact be strengthened to probabilistically strong existence: as for the Ornstein–Uhlenbeck process alone, the triple $(u^{\epsilon},w^{\epsilon},r^{\epsilon})$ exists on any probability space supporting the $Q$-fractional Wiener process $W^{\H}$.

As already mentioned, it is necessary to reformulate the system within the framework of rough path theory. This is carried out in \autoref{sec_rough_formulation} and \autoref{sec_unbounder_rough}, and allows us to control the singular term $b(\epsilon^{-\alpha+\H}w^{\epsilon},u^{\epsilon})$ by introducing the process
\begin{align*}
    y^{\epsilon}_t=\epsilon^{-\alpha+\H}\int_0^t w^{\epsilon}_s ds 
\end{align*}
together with its canonical lifts. Here, in contrast to \cite{debussche2023rough}, we cannot rely on It\^o-type formulas to characterize the lifts, and instead must appeal to the theory of Gaussian rough paths (see \cite{friz2010multidimensional, friz2014course}), especially in the case $\H<\tfrac12$. As already discussed, when $\H>\tfrac12$ no It\^o-Stokes drift appears in the limit, whereas for $\H<\tfrac12$ this drift is the only surviving contribution. This dichotomy stems from the convergence properties of $r^{\epsilon}$, which ultimately rely on an ergodic-type result for $\int_0^T b(w^{\epsilon}_s,w^{\epsilon}_s)ds$. While the Markovian structure makes this argument relatively
standard when $\H=\tfrac12$
(cf. \cite[Section~5]{debussche2023rough}), in the fractional
case we exploit the quadratic structure of the nonlinearity
to reduce the required averaging estimates to explicit
computations involving the covariance of the fractional
Ornstein-Uhlenbeck process; see \autoref{sec:stoc_conv}. With these preliminaries in place, we can now state our main result.
\begin{theorem}\label{main_thm}
Let $\H\in (\frac{1}{3},1)$, and let the initial data $u_0^{\epsilon},v_0^\epsilon $ be divergence-free, of zero mean, and such that
\begin{align*}
   \sup_{\epsilon\in (0,1)} \|u_0^{\epsilon}\|_{L^2}+\|\sqrt{\epsilon}v_0^{\epsilon}\|_{L^2}<+\infty.
\end{align*}
Assume that the operators $Q$ and $C$ satisfy the conditions stated in \autoref{sec_hp_noise}. Then there exist probabilistically weak rough path solutions $(u^{\epsilon},(-C)^{-1}W^{\H,\epsilon})$ to \eqref{system_prelimit} that, up to extraction of a subsequence, converge in law to a probabilistically weak rough path solution of \eqref{eq_Hgreat12} in the case $\alpha=1$ and $\H>\tfrac{1}{2}$, and to a probabilistically weak rough path solution of \eqref{eq_Hsmall12} in the case $\H<\tfrac{1}{2}$ and $\alpha=\tfrac{1}{2}+\H$. 
\end{theorem}
\begin{remark}\label{remark_d2}
    In dimension $d=2$, thanks to the uniqueness of both the prelimit and the limit equations, the above statement can be strengthened. Assuming that $u_0^\epsilon \to u_0$ in $L^2(\T^2;\R^2)$, and keeping the same hypotheses as in \autoref{main_thm}, there exist probabilistically strong rough path solutions $(u^{\epsilon},(-C)^{-1}W^{\H})$ to \eqref{system_prelimit} that converge in probability to the unique weak rough path solution of \eqref{eq_Hgreat12} in the case $\alpha=1$ and $\H>\tfrac{1}{2}$, and of \eqref{eq_Hsmall12} in the case $\H<\tfrac{1}{2}$ and $\alpha=\tfrac{1}{2}+\H$.
\end{remark}
\begin{remark}\label{remark_uniqueness}
    When $\H<\frac12$ and $\alpha=\frac12+\H$, the limiting law is concentrated on weak solutions of the deterministic equation \eqref{eq_Hsmall12}. In dimension three, owing to the possible non-uniqueness of these solutions, our argument does not identify this law as a Dirac mass on  a deterministic trajectory.
\end{remark}
We now discuss why the complementary regimes appear to be out of reach in the full generality of \autoref{main_thm}. In the case $\H>\tfrac{1}{2}$ and $\alpha=\tfrac{1}{2}+\H>1$, neglecting the additional Navier–Stokes terms in \eqref{e:equation1} yields that the small scales satisfy
\begin{align*}
    v^{\epsilon}_t=\epsilon^{-(\H-1/2)}\epsilon^{-1}\int_0^t e^{\epsilon^{-1}C(t-s)} dW^{\H}_s.
\end{align*}
In view of \eqref{heuristic_convergence_noise}, this leads to
\begin{align*}
    v^{\epsilon}_t\approx \epsilon^{-(\H-1/2)}(-C)^{-1}dW^{\H}_t\rightarrow +\infty,
\end{align*}
showing that the small-scale component diverges.

The case $\H<\tfrac{1}{2}$ and $\alpha=1$ is more delicate.
When $\alpha=1$, the leading fast velocity is
$\epsilon^{-1+\H}w^\epsilon$. Although $w^\epsilon$ has zero mean,
its quadratic self-interaction need not have zero mean. Since the
fast equation is quadratic in the fast variables, this interaction
appears as a forcing term in the equation for $r^\epsilon$, with
a factor $\epsilon^{2\H-2}$. Balancing its average against the
damping, which acts on the time scale $\epsilon$, produces a
contribution to the remainder proportional to
$\epsilon^{2\H-1}\overline r$, where $\overline r$ is a deterministic,
time-independent vector field depending on the covariance structure
of the noise and the operator $C$; see \eqref{eq:OUsize1} below.
Since $\epsilon^{2\H-1}$ diverges when $\H<\frac12$, a nonzero
$\overline r$ obstructs uniform estimates for the remainder and
generates an increasingly strong deterministic transport in the
slow equation.

If $\overline r=0$, namely if \eqref{eq:centering_cond} holds,
this averaged contribution disappears. The quadratic forcing itself,
however, is still present and must be accounted for. We therefore
iterate the decomposition of the fast component by introducing a
second corrector $z^\epsilon$ and a new remainder $R^\epsilon$:
\begin{align*}
    r^\epsilon
    &=\epsilon^{2\H-1}z^\epsilon+R^\epsilon,\\
    v^\epsilon
    &=\epsilon^{\H-1}w^\epsilon
      +\epsilon^{2\H-1}z^\epsilon+R^\epsilon.
\end{align*}
We choose $z^\epsilon$ to absorb the quadratic forcing generated
by $w^\epsilon$. Substituting this decomposition into
\eqref{system_prelimit} gives
\begin{align}\label{system_second_corrector}
    \begin{cases}
        \partial_t u^{\epsilon}=\nu A u^{\epsilon}+b(u^{\epsilon}+\epsilon^{\H-1}w^{\epsilon}+\epsilon^{2\H-1}z^{\epsilon}+R^{\epsilon},u^{\epsilon}),\\
        dw^{\epsilon}=\epsilon^{-1}Cw^{\epsilon} dt+ \epsilon^{-\H}dW^\H_t,\\
        \partial_t z^\epsilon
        =\epsilon^{-1}Cz^\epsilon
         +\epsilon^{-1}b(w^\epsilon,w^\epsilon),
        \\
        \begin{aligned}
        \partial_t R^\epsilon
        &=\epsilon^{-1}CR^\epsilon
          +\nu A\left(\epsilon^{\H-1}w^\epsilon
                     +\epsilon^{2\H-1}z^\epsilon+R^\epsilon\right)\\
        &\quad+b\left(u^\epsilon+\epsilon^{\H-1}w^\epsilon
                     +\epsilon^{2\H-1}z^\epsilon+R^\epsilon,
                     R^\epsilon\right)\\
        &\quad+\epsilon^{\H-1}b(u^\epsilon+R^\epsilon,w^\epsilon)
              +\epsilon^{2\H-1}b(u^\epsilon+R^\epsilon,z^\epsilon)\\
        &\quad+\epsilon^{3\H-2}
              \left(b(w^\epsilon,z^\epsilon)
                   +b(z^\epsilon,w^\epsilon)\right)
              +\epsilon^{4\H-2}b(z^\epsilon,z^\epsilon),
        \end{aligned}\\
        u^{\epsilon}(0) =u_0^{\epsilon},\quad w^{\epsilon}(0) =0,\quad z^\epsilon(0)=0,\quad R^\epsilon(0)=v_0^\epsilon.
    \end{cases}
\end{align}
Thus, the forcing $\epsilon^{2\H-2}b(w^\epsilon,w^\epsilon)$ has
been removed from the equation for the new remainder. The additional
quadratic terms involve the less singular factors
$\epsilon^{3\H-2}$ and $\epsilon^{4\H-2}$, which can be controlled
using the fast damping when $\H>\frac13$. Meanwhile, the centering
condition and the quantitative averaging estimates developed in
\autoref{sec:stoc_conv} ensure that the accumulated contribution of
$\epsilon^{2\H-1}z^\epsilon$ to the slow dynamics vanishes.
Together, these observations allow us to establish the required
estimates for $R^\epsilon$ and continue the convergence argument. In particular, in this setting we can establish the following result:
\begin{theorem}\label{centered_noise}
Let $\alpha=1$, $\H\in\left(\tfrac{1}{3},\tfrac{1}{2}\right)$, and let the initial data $u_0^{\epsilon}$ and $v_0^{\epsilon}$ be divergence-free, of zero mean, and satisfy
\begin{align*}
   \sup_{\epsilon\in (0,1)} \|u_0^{\epsilon}\|_{L^2}+\|\sqrt{\epsilon}v_0^{\epsilon}\|_{L^2}<+\infty.
\end{align*} Assume that $Q$ and $C$ satisfy the conditions stated in \autoref{sec_hp_noise} and \eqref{eq:centering_cond} holds. Then there exist probabilistically weak rough path solutions $(u^{\epsilon},(-C)^{-1}W^{\H,\epsilon})$ to \eqref{system_second_corrector} that, up to extraction of a subsequence, converge in law to a probabilistically weak rough path solution of \eqref{eq_Hgreat12}. 
\end{theorem}
We conclude this subsection by mentioning the recent preprint
\cite{li2026navierstokesfractionaltransportnoise}, which appeared
during the final stages of preparation of this work. There, the
authors study a fast-slow Navier-Stokes system similar to
\eqref{e:equation1}, with $\H\in(\tfrac13,\tfrac12)$ and $\alpha=1$,
focusing on the Wong-Zakai regime. While the slow dynamics coincide
with ours, the fast component is governed by the linear equation
\begin{align}\label{eq_fast_1}
    dv^\epsilon
    =\epsilon^{-1}Cv^\epsilon\,dt+\epsilon^{-1}dW^\H_t.
\end{align}
For this system, \cite[Theorem~B]{li2026navierstokesfractionaltransportnoise}
establishes subsequential convergence to a Navier-Stokes equation
with fractional transport noise.

The quadratic self-interaction omitted from \eqref{eq_fast_1}
plays an important role in our model and is one of the main mechanisms driving turbulence at high Reynolds numbers. In particular, its averaged
contribution generates the It\^o-Stokes drift $\overline r$, which
is absent from the limit obtained with the linear fast dynamics \eqref{eq_fast_1}.
At the same time, as discussed above, this mechanism creates
a difficulty in the Wong-Zakai regime: the contribution
$\epsilon^{2\H-1}\overline r$ becomes singular as $\epsilon\to0$
and appears to obstruct convergence for the full nonlinear system.

If the quadratic self-interaction is removed from the fast equation,
this obstruction disappears and our arguments apply directly.
This holds both for \eqref{eq_fast_1} and for the more general dynamics
\begin{align}\label{eq_fast_2}
    dv^\epsilon
    =\nu Av^\epsilon\,dt+\epsilon^{-1}Cv^\epsilon\,dt
     +b(u^\epsilon,v^\epsilon)\,dt+\epsilon^{-1}dW^\H_t,
\end{align}
yielding the corresponding Wong-Zakai convergence for
$\H\in(\tfrac13,\tfrac12)$ and $\alpha=1$ under our assumptions
on $C$ and $Q$ stated in \autoref{main_thm}. However, neither \eqref{eq_fast_1} nor \eqref{eq_fast_2} allows one to detect the additional scaling
$\alpha=\frac12+\H$ identified in \autoref{sec:appendix_multiscale}.

Our result in \autoref{centered_noise} shows that convergence can
also be established while retaining the quadratic self-interaction,
provided that the centering condition $\overline r=0$ holds.
This condition removes the problematic averaged contribution,
while the remaining quadratic fluctuations are handled through
the second corrector $z^\epsilon$ and the quantitative averaging
estimates developed in \autoref{sec:stoc_conv}.

\subsubsection*{Plan of the paper}In \autoref{sec_prel} we rigorously introduce our assumptions on $Q$ and $C$, and we recall some rough path notation including the notion of solutions for \eqref{eq_Hgreat12}, \eqref{eq_Hsmall12}, \eqref{system_prelimit}. The analysis of the long-time behaviour of $w^{\epsilon}$ and of some of its nonlinear functionals is addressed in \autoref{sec:stoc_conv}. The core of our rough path analysis is \autoref{sec_unbounder_rough}, where we study the convergence properties of the canonical lift of $\int_0^{\cdot}\epsilon^{-\alpha+\H}w^{\epsilon}_s ds$. In the centered regime, these estimates are combined with control of the accumulated contribution of $z^\epsilon$.
The last ingredients for the proofs of \autoref{main_thm} and and \autoref{centered_noise} are some a priori estimates obtained in \autoref{sec_apriori} exploiting the rough path formulation already mentioned. Combining these elements, the proofs of both \autoref{main_thm} and \autoref{centered_noise} follow by standard arguments as discussed in \autoref{sec:end_proof}.
\section{Preliminaries}\label{sec_prel}
\subsection{Notation}\label{sec_notation}
 Let $H$ be the subspace of $L^2(\T^d,\R^d)$ consisting of zero mean, divergence-free vector fields, where $\T^d$ denotes the d-dimensional torus and $d\in\{2,3\}$.\\ We denote by $\Pi$ the Leray-Hopf projector from the zero-mean
subspace of $L^2(\T^d,\R^d)$ onto $H$, i.e. $\Pi=\I+\nabla(-\Delta)^{-1}\dvg$ and by $A=\Pi\Delta$ the Stokes operator on $H$. It is a closed, unbounded, linear operator $A:\mathrm{Dom}(A)=H^2\cap H\subseteq H\rightarrow H$ and it is well-known (see for example \cite{temam2024navier}) that $A$ is self-adjoint and generates an analytic semigroup of negative type on $H$. \\
 We further denote by $H^n=\mathrm{Dom}((-A)^{n/2})$ equipped with the graph norm
\begin{align*}
     \|v\|^2_{H^n}:=\langle(-A)^{n/2}v,(-A)^{n/2}v\rangle,
\end{align*}
 where $\langle\cdot,\cdot\rangle$ denotes the standard scalar product in $L^2$.\\
In the following, we will denote by $b$ the Navier-Stokes non-linearity, namely
\begin{equation*}
    b(u,v)=-\Pi(u\cdot\nabla)v.
\end{equation*}
We conclude this subsection by specifying the formalism of the rough path theory. \\
Any continuous function $\omega:\Delta_T:=\{(s,t):0\le s\le t\le T\}\to [0,+\infty)$ that satisfies $\omega(t,t)=0$ and is super-additive, i.e $\omega(s,u)+\omega(u,t)\le \omega(s,t)$, is called a control.\\
For $p>0$ and a Banach space $V$, a continuous $1$-index map $f:[0,T]\to V$ has finite $p$-variation if 
\begin{equation}\label{eq:pVar}
    \|f\|_{p\var,[0,T]}^p:=\sup_{\mathbf{P}\in \mathbf{\Pi}}\sum_{t_i\in \mathbf{P}}\|f_{t_i,t_{i+1}}\|_V^p<\infty,
\end{equation}
where $\mathbf{\Pi}$ denotes the partitions $\mathbf{P}:=\{t_0:=0,\dots,t_N:=T\}$ of $[0,T]$, while $f_{t_i,t_{i+1}}:=f(t_{i+1})-f(t_i)$. Analogously, a $2$-index function $f:\Delta_T\to V$ has finite $p$-variation if \eqref{eq:pVar} holds with $f_{t_i,t_{i+1}}:=f(t_i,t_{i+1})$.
Let us recall that 
$\omega_f(s,t):=\|f\|_{p\var,[s,t]}^p$  is a control. Whenever we want to highlight the Banach space under consideration, we will denote the $p$-variation seminorm as $\|f\|_{p\var,[0,T],V}.$\\
We consider the spaces $C^{p\var}(V)$
(resp. $C^{p\var}_2(V)$) obtained as the closure of smooth
$1$-index functions (resp. $2$-index functions vanishing on the diagonal) with respect to
the metric $d_{p,1}(f,g)
=\|f(0)-g(0)\|_V+\|f-g\|_{p\var}^{p\wedge1}$
(resp. $d_{p,2}(f,g)=\|f-g\|_{p\var}^{p\wedge1}$).
These spaces are Polish when $V$ is separable.
Moreover, we denote by $C^{p\var}_{2,loc}(V)$ the space of $2$-index maps $f$ such that there exists a finite covering $\{I_k\}_{k}$ of $[0,T]$ where $f$ has finite $p$-variation on any $I_k$.

For any Hilbert space $H$, let $H^{\otimes k}$ be equipped with the standard Hilbert space structure.
\begin{definition} Let $p\in [N,N+1)$ for $N\in\{1,2\}$,
     a $p$-rough path is a function $\Y=(\mathbb{Y}^1,\dots,\mathbb{Y}^N): \Delta_T\to H\oplus \dots \oplus H^{\otimes N}$ such that
     \begin{enumerate}
         \item for $k\le N$, $\mathbb{Y}^k\in C^{\frac{p}{k}\var}_{2}(H^{\otimes k});$
        \item Chen's relation holds, i.e. for $s\le u\le t$
        \begin{equation*}
            \delta\mathbb{Y}^1_{s,u,t}=0,\quad \delta\mathbb{Y}^2_{s,u,t}=\mathbb{Y}^1_{s,u}\otimes\mathbb{Y}^1_{u,t}
        \end{equation*}
     \end{enumerate}
     where $\delta v_{s,u,t}=v\st-v_{s,u}-v_{u,t}$.
\end{definition}
To properly define unbounded rough drivers (URDs in the following), we introduce the concept of scale:
a family of Banach spaces $E=(E_n,\|\cdot\|_n)_{n\in\N}$ will be called a scale if $E_n$ is continuously embedded in $E_m$ for $n>m$. We denote $E_n^{*}$ by $E^{-n}$.

\begin{definition}\label{def:URD}
    Let $p\in [N,N+1)$ for $N\in\{1,2\},$ and $E$ a scale. Then, an $N$-tuple $\mathbf{A}=(\mathbb{A}^1,\dots,\mathbb{A}^N)$ of two index maps is a $p$-unbounded rough driver w.r.t. the scale $E$ if the following conditions hold:
    \begin{enumerate}
        \item there exists a control $\omega$ s.t. 
        \begin{equation*}
            \|\mathbb{A}^k_{s,t}\|_{\mathcal{L}(E^{-m},E^{-m-k})}^{p/k}\lesssim \omega(s,t)\quad \text{for }m\in [0,N+1-k];
        \end{equation*}
        \item a functional version of Chen's relation is satisfied, namely
        \begin{equation}
            \begin{aligned}
                \delta\mathbb{A}^1_{s,u,t}=0,\quad \delta \mathbb{A}^2_{s,u,t}=\mathbb{A}^1_{u,t}\mathbb{A}^1_{s,u}.
            \end{aligned}
        \end{equation}
    \end{enumerate}
\end{definition}
At last, since we will analyze an $H$-valued Gaussian rough path $X$, we denote its covariance by $R_X$, namely
\begin{equation}\label{eq:CovGauss}
    R_X(s,t):=\EE[X_s\otimes X_t].
\end{equation}
\subsection{Small Scales Assumptions}\label{sec_hp_noise}
Regarding the noise structure, let $\xi>\frac{d}{2}+2$ and $\{e_i\}_{i\in \mathbb{N}}$ be a complete orthonormal system in $H$ consisting of elements of $H^{\xi+1}$ that are mutually orthogonal in $H^\xi$ and $\sigma_i\geq 0$ such that
\begin{align}\label{HP_noise_coeff}
    \sum_{i\geq 1}\sigma_i^2 \norm{e_i}_{H^{\xi+1}}^2<+\infty.
\end{align}
The former implies in particular that $\sup_{i}\sigma_i^2<+\infty$ and the operator 
 \begin{align*}
    Q:=\sum_{i\geq 1}\sigma_i^2 e_i\otimes e_i,
\end{align*} is trace class on $H^{\xi+1}.$\\
For a family of i.i.d. 1-dimensional fractional Brownian motions $\{W^{\mathcal{H},i}\}_{i\in\N}$ with Hurst exponent $\H$, we consider the following fractional noise
 \begin{equation*}
     W^\H_t=\sum_{i\in\N}\sigma_ie_iW^{\H,i}_t.
 \end{equation*}
 Moreover, for $1/3<\H <1/2$ we assume that there exists $p\in (\frac{1}{\H},3)$ s.t.
 \begin{equation}\label{eq:p_summability}
     \sum_k|\sigma_k\|e_k\|_{H^{\xi}}|^{p/2}<\infty.
 \end{equation}
Lastly, we introduce the friction operator
$C:D(C)\subseteq L^2\to L^2$ as a linear, densely defined,
closed operator with $0\in\rho(C)$, diagonal with respect
to the system $\{e_i\}_{i\in\mathbb N}$ introduced above:
\begin{align*}
    Ce_k=\lambda_k e_k,\qquad k\in\mathbb N.
\end{align*}
We assume furthermore that 
\begin{align}\label{Ass1onC}
\sup_k\lambda_k=\lambda_1<0\end{align}
and there exists $\gamma\geq 0$ s.t.  
\begin{align}
    \label{Ass2onC} -\langle w,Cw\rangle  \gtrsim \|w\|^2_{H^\gamma}.
\end{align}
We finally assume that the semigroup generated by $C$ acts on $H^\xi$
and satisfies for each $t\geq 0$ and some $\mathfrak{h}>0$
\begin{align}\label{hp_semigroup}
    \|e^{tC}\|_{\mathcal L(H^\xi,H^\xi)}
    \lesssim  e^{-\mathfrak{h}t}.
\end{align}

\subsection{Rough path formulation}\label{sec_rough_formulation} 
As noted earlier in \autoref{subsec_results}, we will decompose $v^\epsilon$ in two different ways depending on whether we are proving \autoref{main_thm} or \autoref{centered_noise}. Accordingly, we present a distinct rough path formulation for each result being considered.

\subsubsection{\autoref{main_thm} solutions' formulation}\label{subsubsec_formulation1}
In order to simplify the notation, from now on we set $\nu=1.$ Following \cite{debussche2023rough,debussche2024second}, we decompose the small scale velocity as 
\begin{equation*}
    v^{\epsilon}=\epsilon^{-\alpha+\H}w^{\epsilon}+r^{\epsilon},
\end{equation*}
where $r^{\epsilon}$ is a remainder term and $w^{\epsilon}$ is a fractional Ornstein-Uhlenbeck process (fOU). This decomposition will allow us to study the effect of the small scales $v^{\epsilon}$ on the large ones $u^{\epsilon}$ by investigating convergence due to the fOU and to the non-linear term on $v^{\epsilon}$ separately. In particular, we treat the system
\begin{equation*}
    \begin{cases}
        du^{\epsilon}= A u^{\epsilon} \,dt+b(\epsilon^{-\alpha+\H}w^{\epsilon},u^{\epsilon})\,dt+b(u^{\epsilon}+r^{\epsilon},u^{\epsilon})\,dt,\\
        dw^{\epsilon}=\epsilon^{-1}Cw^{\epsilon}\,dt+\epsilon^{- \mathcal{H}}dW^\mathcal{H}_t,\\
        \begin{aligned}dr^{\epsilon}&=\epsilon^{-1}Cr^{\epsilon}\,dt+ A(\epsilon^{-\alpha+\H}w^{\epsilon}+r^{\epsilon})\,dt+b(u^{\epsilon}+\epsilon^{-\alpha+\H}w^{\epsilon}+r^{\epsilon},\epsilon^{-\alpha+\H}w^{\epsilon}+r^{\epsilon})\,dt.
        \end{aligned}
    \end{cases}
\end{equation*}
Consider $y_t^{\epsilon}=\int_0^t\epsilon^{-\alpha+\H}w^{\epsilon}_r\,dr$. Then, 
\begin{equation*}
    \int_s^tb(\epsilon^{-\alpha+\H}w^{\epsilon}_r,u^{\epsilon}_r)\,dr=\int_s^tb(dy^{\epsilon}_r,u^{\epsilon}_r).
\end{equation*}

Intending to provide a description that can unify different levels of irregularity of FBMs, while also accounting for the fact that $b$ may be unbounded, we interpret the integral in the sense of Davie’s formulation \cite{davie2008differential}. Building on the higher-order unbounded rough driver framework
developed in \cite{nilssen2026higher}, we expect this formulation
to extend to $\H\in(1/4,1/3)$.

For the range considered here, let $\frac1\H\in[N,N+1)\cap(1,3)$ and $p\in (\frac{1}{\H},N+1)$, 
and consider the canonical  $p$-rough path $\Y^\epsilon$ associated with $y^\epsilon$. We define
\begin{equation*}
\int_s^tb(dy^\epsilon_r,u^\epsilon_r)=\sum_{k=1}^N\mathbb{A}^{k,\epsilon}\st u_s^\epsilon+u^{\epsilon,\natural}\st,
\end{equation*}
where $u^{\natural}\in C^{p/(N+1)-var}_{2,loc}(H^{-3}),$ while 
\begin{align}
    &\mathbb{A}^{1,\epsilon}\st\phi=b(\mathbb{Y}^{1,\epsilon}\st,\phi),\\
    &\mathbb{A}^{2,\epsilon}\st\phi=\int_s^tb(dy^\epsilon_r,b(y^\epsilon_{s,r},\phi))=b^{(2)}\left(\mathbb{Y}^{2,\epsilon}\st,\phi\right)
\end{align}
where $b^{(k)}$ is a bilinear operator on $(H^\xi)^{\otimes k}\times C^{\infty}(\T^d)$ uniquely defined by
\begin{equation*}
    b^{(k)}(f_1\otimes\dots\otimes f_k,\phi)=b(f_k,b(f_{k-1},\dots,b(f_1,\phi))).
\end{equation*}
Notice that $\mathbf{A}^\epsilon=(\mathbb{A}^{1,\epsilon},\dots,\mathbb{A}^{N,\epsilon})$ is a $p$ unbounded rough driver, in the sense of \autoref{def:URD}, with respect to the scale $E_n=H^n$ (see \cite{debussche2023rough}).

We can finally define a solution to equation \eqref{system_prelimit}.
\begin{definition}\label{def_rough_sol}
    Let $\epsilon\in (0,1)$, $\frac1\H\in[N,N+1)\cap(1,3)$, $p\in (\frac{1}{\H},N+1)$ and assume $C,Q$ satisfy the condition of \autoref{sec_hp_noise}.\\
     A tuple $((\Omega^\epsilon, \mathcal{F}^\epsilon, (\mathcal{F}_t^\epsilon)_{t\geq 0},\mathbb{P}^{\epsilon}),W^{\epsilon}_t,u^\epsilon_t,r^\epsilon_t)$ is a probabilistically weak solution to equation \eqref{system_prelimit} if 
    \begin{enumerate}
        \item $(\Omega^\epsilon, \mathcal{F}^\epsilon, (\mathcal{F}_t^\epsilon)_{t\geq 0},\mathbb{P}^{\epsilon})$ is a stochastic basis with complete and right-continuous filtration,
        \item $W^{\epsilon,\H}_t$ is a fractional Wiener process of Hurst parameter $\H$ and space covariance $Q$ and it is adapted to $\mathcal{F}_t^\epsilon$,
        \item $r^\epsilon\in L^2([0,T],H^{\gamma})$ $\mathbb{P}^{\epsilon}$-a.s., it is progressively measurable and satisfies $\mathbb{P}^{\epsilon}$-a.s. 
       \begin{equation*}
       \begin{aligned}
    \partial_t r^{\epsilon}&=\epsilon^{-1}Cr^{\epsilon}+ A(\epsilon^{-\alpha+\H}w^{\epsilon}+r^{\epsilon})\\&+b(u^{\epsilon}+\epsilon^{-\alpha+\H}w^{\epsilon}+r^{\epsilon},\epsilon^{-\alpha+\H}w^{\epsilon}+r^{\epsilon})
    \end{aligned}
       \end{equation*}
        in an analytically weak sense,
        \item $u^{\epsilon}\in L^{\infty}([0,T],H)\cap L^2([0,T],H^1)$ $\mathbb{P}^{\epsilon}$-a.s., it is progressively measurable and $\mathbb{P}^{\epsilon}$-a.s it holds, in a weak sense, 
        \begin{equation*}
            u^\epsilon\st=\int_s^t \Delta u^\epsilon_\theta+b(u_\theta^\epsilon+r^\epsilon_\theta,u^\epsilon_\theta)\,d\theta+\sum_{k=1}^N \mathbb{A}^{k,\epsilon}\st u_s^\epsilon+u^{\epsilon,\natural}\st,
        \end{equation*}
        where $u^{\natural}\st\in C^{p/(N+1)\var}_{2,loc}(H^{-3})$ $\mathbb{P}^{\epsilon}$-a.s.
    \end{enumerate}
    \end{definition}
    At last, we report the definition of the limiting equations, \eqref{eq_Hgreat12} and \eqref{eq_Hsmall12}.
    \begin{definition}\label{def_rough_lim_sol_Hgreat12}
        Let $\frac1\H\in[N,N+1)\cap(1,3)$, $p\in (\frac{1}{\H},N+1)$ and assume $C,Q$ satisfy the condition of \autoref{sec_hp_noise}.\\
     A tuple $((\Omega, \mathcal{F}, (\mathcal{F}_t)_{t\geq 0},\mathbb{P}),\W^{\H}_t,u_t)$ is a probabilistically weak solution to equation \eqref{eq_Hgreat12} if 
    \begin{enumerate}
        \item $(\Omega, \mathcal{F}, (\mathcal{F}_t)_{t\geq 0},\mathbb{P})$ is a stochastic basis with complete and right-continuous filtration,
        \item $\mathbb{W}^{\H,1}_t$ is a fractional Wiener process of Hurst parameter $\H$ and space covariance $(-C)^{-1}Q(-C)^{-1}$ and it is adapted to $\mathcal{F}_t$,
        \item $\W^\H=(\mathbb{W}^{\H,k})_{k=1}^N$ is the canonical Gaussian $p$-rough path lift of $\mathbb{W}^{\H,1}$ on $H^\xi$, $\mathbb{P}$-a.s.,
        \item $u\in C_w([0,T],H)\cap L^2([0,T],H^1)$ $\mathbb{P}$-a.s., it is progressively measurable and $\mathbb{P}$-a.s it holds, in a weak sense, 
        \begin{equation*}
            u\st=\int_s^t \Delta u_r+b(u_r,u_r)\,dr+\sum_{k=1}^N \mathbb{A}^{k}\st u_s+u^{\natural}\st,
        \end{equation*}
        where $u^{\natural}\st\in C^{p/(N+1)\var}_{2,loc}(H^{-3})$ $\mathbb{P}$-a.s., and the operators $(\mathbb{A}^k)_{k=1}^N$ are defined by
    \begin{equation}\label{def_limitURD}
        \mathbb{A}^k\st\phi
        =b^{(k)}(\mathbb{W}^{\H,k}\st,\phi),
        \qquad k=1,\dots,N.
    \end{equation}
    \end{enumerate}
    \end{definition}
    \begin{definition}\label{def_rough_lim_sol_Hsmall12}
    A tuple $((\Omega, \mathcal{F}, \mathbb{P}),u_t)$ is a random weak solution to equation \eqref{eq_Hsmall12} if: 
    \begin{enumerate}
        \item $(\Omega, \mathcal{F},\mathbb{P})$ is a complete probability space,
        
        \item $u $ is a jointly measurable $H$-valued process satisfying $u\in C_w([0,T],H)\cap L^2([0,T],H^1)$ $\mathbb{P}$-a.s., and the following identity in a weak sense
        \begin{equation*}
            u\st=\int_s^t \Delta u_\theta+b(u_\theta+\overline{r},u_\theta)\,d\theta,
        \end{equation*}
        where $\overline{r}\in H$.
    \end{enumerate}
    \end{definition}

\begin{remark}
    In \cite{li2026navierstokesfractionaltransportnoise}, it is shown that, for $\H<1/2$, the rough integral can be properly built once the equation is tested against sufficiently regular test functions. This suggests a slightly different notion of solution, where the rough integral is not interpreted in the sense of Davie and, in particular, through the URD formalism. Nevertheless, the authors show that the two notions of solution are equivalent. We have chosen to stick with Davie's formulation to treat simultaneously $\H<1/2$ and $\H>1/2$, without the need to specify whether we rely on Young or rough integration.
\end{remark}
\subsubsection{\autoref{centered_noise} solutions' formulation}\label{subsubsec_formulation2}
According to the notation presented in \autoref{subsec_results}, we decompose the small scale as
\begin{align*}
    v^\epsilon
    &=\epsilon^{\H-1}w^\epsilon
      +\epsilon^{2\H-1}z^\epsilon+R^\epsilon,
\end{align*}
where we recall that the full system is given by
\begin{align}
    \begin{cases}
        \partial_t u^{\epsilon}=\nu A u^{\epsilon}+b(u^{\epsilon}+\epsilon^{\H-1}w^{\epsilon}+\epsilon^{2\H-1}z^{\epsilon}+R^{\epsilon},u^{\epsilon}),\\
        dw^{\epsilon}=\epsilon^{-1}Cw^{\epsilon} dt+ \epsilon^{-\H}dW^\H_t,\\
        \partial_t z^\epsilon
        =\epsilon^{-1}Cz^\epsilon
         +\epsilon^{-1}b(w^\epsilon,w^\epsilon),
        \\
        \begin{aligned}
        \partial_t R^\epsilon
        &=\epsilon^{-1}CR^\epsilon
          +\nu A\left(\epsilon^{\H-1}w^\epsilon
                     +\epsilon^{2\H-1}z^\epsilon+R^\epsilon\right)\\
        &\quad+b\left(u^\epsilon+\epsilon^{\H-1}w^\epsilon
                     +\epsilon^{2\H-1}z^\epsilon+R^\epsilon,
                     R^\epsilon\right)\\
        &\quad+\epsilon^{\H-1}b(u^\epsilon+R^\epsilon,w^\epsilon)
              +\epsilon^{2\H-1}b(u^\epsilon+R^\epsilon,z^\epsilon)\\
        &\quad+\epsilon^{3\H-2}
              \left(b(w^\epsilon,z^\epsilon)
                   +b(z^\epsilon,w^\epsilon)\right)
              +\epsilon^{4\H-2}b(z^\epsilon,z^\epsilon),
        \end{aligned}\\
        u^{\epsilon}(0) =u_0^{\epsilon},\quad w^{\epsilon}(0) =0,\quad z^\epsilon(0)=0,\quad R^\epsilon(0)=v_0^\epsilon.
    \end{cases}
\end{align}
Therefore, following \autoref{subsubsec_formulation1}, we consider $\hat{y}^\epsilon_t=\int_0^t\epsilon^{\H-1}w^\epsilon_s+\epsilon^{2\H-1}z^\epsilon_sds$, denote the corresponding rough path by $\hat{\mathbf{Y}}^\epsilon$ and interpret 
\begin{equation*}
\int_s^tb(\epsilon^{\H-1}w^\epsilon_r+\epsilon^{2\H-1}z^\epsilon_r,u^\epsilon_r)dr=\int_s^t b(d\hat{y}^\epsilon_r,u^\epsilon_r)
\end{equation*}
in the sense of Davie's formulation. Therefore,
\begin{equation*}
    \int_s^t b(d\hat{y}^\epsilon_r,u^\epsilon_r)=\sum_{k=1}^2\hat{\A}^{k,\epsilon}\st u^\epsilon_s+u^{\epsilon,\natural}\st
\end{equation*}
where, as in the previous subsection, $u^{\natural}\in C^{p/3-var}_{2,loc}(H^{-3}),$ while $\hat{\mathbf{A}}^\epsilon=(\hat{\mathbb{A}}^{1,\epsilon},\hat{\mathbb{A}}^{2,\epsilon})$ given by
\begin{align}
    &\hat{\mathbb{A}}^{1,\epsilon}\st\phi=b(\hat{\mathbb{Y}}^{1,\epsilon}\st,\phi),\quad \hat{\mathbb{A}}^{2,\epsilon}\st\phi=b^{(2)}\left(\hat{\mathbb{Y}}^{2,\epsilon}\st,\phi\right),
\end{align}
is a $p$ unbounded rough driver w.r.t. the scale $E_n=H^n.$ We end this section by introducing a definition of a solution to \eqref{system_second_corrector}. The appropriate concept of solution for the limiting equation in this framework is given in \autoref{def_rough_lim_sol_Hgreat12}.

\begin{definition}\label{def_rough_sol_2}
    Let $\epsilon\in (0,1)$, $\frac1\H\in(2,3)$, $p\in (\frac{1}{\H},3)$ and assume $C,Q$ satisfy the condition of \autoref{sec_hp_noise}.\\
     A tuple $((\Omega^\epsilon, \mathcal{F}^\epsilon, (\mathcal{F}_t^\epsilon)_{t\geq 0},\mathbb{P}^{\epsilon}),W^{\epsilon}_t,u^\epsilon_t,R^\epsilon_t)$ is a probabilistically weak solution to equation \eqref{system_second_corrector} if 
    \begin{enumerate}
        \item $(\Omega^\epsilon, \mathcal{F}^\epsilon, (\mathcal{F}_t^\epsilon)_{t\geq 0},\mathbb{P}^{\epsilon})$ is a stochastic basis with complete and right-continuous filtration,
        \item $W^{\epsilon,\H}_t$ is a fractional Wiener process of Hurst parameter $\H$ and space covariance $Q$ and it is adapted to $\mathcal{F}_t^\epsilon$,
        \item $R^\epsilon\in L^2([0,T],H^{\gamma})$ $\mathbb{P}^{\epsilon}$-a.s., it is progressively measurable and satisfies $\mathbb{P}^{\epsilon}$-a.s. 
       \begin{equation*}
       \begin{aligned}
    \partial_t R^\epsilon
        &=\epsilon^{-1}CR^\epsilon
          +\nu A\left(\epsilon^{\H-1}w^\epsilon
                     +\epsilon^{2\H-1}z^\epsilon+R^\epsilon\right)\\
        &\quad +b\left(u^\epsilon+\epsilon^{\H-1}w^\epsilon
                     +\epsilon^{2\H-1}z^\epsilon+R^\epsilon,
                     R^\epsilon\right)\\
        &\quad+\epsilon^{\H-1}b(u^\epsilon+R^\epsilon,w^\epsilon)
              +\epsilon^{2\H-1}b(u^\epsilon+R^\epsilon,z^\epsilon)\\
        &\quad+\epsilon^{3\H-2}
              \left(b(w^\epsilon,z^\epsilon)
                   +b(z^\epsilon,w^\epsilon)\right)
              +\epsilon^{4\H-2}b(z^\epsilon,z^\epsilon)
    \end{aligned}
       \end{equation*}
        in an analytically weak sense,
        \item $u^{\epsilon}\in L^{\infty}([0,T],H)\cap L^2([0,T],H^1)$ $\mathbb{P}^{\epsilon}$-a.s., it is progressively measurable and $\mathbb{P}^{\epsilon}$-a.s it holds, in a weak sense, 
        \begin{equation*}
            u^\epsilon\st=\int_s^t \Delta u^\epsilon_\theta+b(u_\theta^\epsilon+R^\epsilon_\theta,u^\epsilon_\theta)\,d\theta+\sum_{k=1}^2 \hat{\mathbb{A}}^{k,\epsilon}\st u_s^\epsilon+u^{\epsilon,\natural}\st,
        \end{equation*}
        where $u^{\natural}\st\in C^{p/3\var}_{2,loc}(H^{-3})$ $\mathbb{P}^{\epsilon}$-a.s.
    \end{enumerate}
    \end{definition}

\section{Quantitative Estimates for Fractional Stochastic Convolutions
and Their Quadratic Functionals}\label{sec:stoc_conv}

In this section, we establish moment bounds and quantitative
averaging estimates for fractional stochastic convolutions and
their quadratic functionals. 

Qualitative ergodic results for infinite-dimensional fractional
Ornstein-Uhlenbeck processes are available in
\cite{maslowski2008ergodicity}. Under our assumptions on $C$ and $Q$,
the analysis below complements these results with explicit
quantitative estimates.

We begin in \autoref{rough_convergence_lemma} with the
one-dimensional case. We then consider the infinite-dimensional
stochastic convolution in \autoref{hp_stoch_conv} and its quadratic
averages in \autoref{lem_ito_stokes}, exploiting the simultaneous
diagonalization of $C$ and $Q$ to apply the scalar estimates.
Finally, we study in \autoref{sec:centered} the second corrector appearing in \eqref{system_second_corrector} under the centering condition
\eqref{eq:centering_cond}. 
\subsection{One-dimensional case}
Let $B^\H_t$ be a real fractional Brownian motion of Hurst parameter $\H\in (\frac{1}{4},1)$ and $X_t$ the unique solution of the Young differential equation
\begin{align}\label{eq:RODE}
    \begin{cases}
    d X_t&=-\lambda X_t dt+\sigma dB^\H_t\\
    X_0&=x_0
    \end{cases}
\end{align}
given by $X_t=e^{-\lambda t}x_0+\sigma\int_0^t e^{-\lambda(t-s)}dB^\H_s$ and denote by $\overline{\sigma}^2$ the quantity \begin{align}\label{limit_variance}
\overline{\sigma}^2=\frac{\sigma^2\lambda}{2}\int_0^{+\infty} e^{-\lambda s}s^{2\H}ds=\frac{\sigma^2}{2\lambda^{2\H}}\Gamma(2\H+1)
\end{align}
for $\sigma,\lambda>0,$ where we denoted by $\Gamma(\cdot)$ the gamma function. The following holds

\begin{lemma}\label{rough_convergence_lemma}
   Let $\H\in (\frac{1}{4},1)$ and let $\overline{\sigma}^2$ be as defined in \eqref{limit_variance}, then
\begin{align}\label{eq:secondMomentOU}
   \expt{|X_t|^2}&=e^{-2\lambda t}|x_0|^2\notag\\
    &+\sigma^2\H \left(\int_0^t  e^{-\lambda z}| z|^{2\H-1} dz + e^{-\lambda t}\int_0^t e^{-\lambda (t-z)} | z|^{2\H-1} dz\right),
\end{align}
\begin{align}\label{rough_exponential_convergence_mean}
    | \expt{X_t^2}-\overline{\sigma}^2| \lesssim e^{-2\lambda t}| x_0|^2+\frac{ \sigma^2 }{\lambda^{2\H}} e^{-\lambda t},
   \end{align}
   \begin{align}\label{MomentOntheSup}
       \EE\sx[\sup_{t\in [0,T]}|X_t|^2\dx]&\lesssim |x_0|^2+\frac{\sigma^2}{\lambda^{2\H}}\log\sx(e+T\sigma^{\frac{1}{\H}}(\lambda^{\frac{1-\H}{\H}}\vee1)\dx)\notag\\
       &+\sx(\int_0^ {\frac{\sigma}{\lambda^\H}}\log^{\frac12}(1+\epsilon^{-\frac{1}{\H}})d\epsilon\dx)^2\notag\\
       &+\frac{\sigma}{\lambda^\H}\sx(\int_0^ {\frac{\sigma}{\lambda^\H}}\log^{\frac12}(1+\epsilon^{-\frac{1}{\H}})d\epsilon+1\dx).
   \end{align}
   Moreover if $\H\in (\frac{1}{4},\frac{1}{2}]$ it holds
\begin{align}\label{rough_power_convergence}
    \expt{|\frac{1}{t}\int_0^t |X_s|^2 ds-\overline{\sigma}^2|^2}&\lesssim \frac{\sigma^4}{\lambda^{4\H+1}t}+\frac{(|x_0|^2+\overline{\sigma}^2)^2}{\lambda^2t^2}.
   \end{align}
\end{lemma}
\begin{proof}
Let $\overline{X}_t$ be the stationary solution of \eqref{eq:RODE}, i.e.
\begin{align*}
    \overline{X}_t=\sigma \int_{-\infty}^t e^{-\lambda (t-s)}dB^\mathcal{H}_s,
\end{align*}
and denote by $\rho(h)$ its correlation function, i.e. $\rho(h)=\mathbb{E}\left[\overline{X}_t\overline{X}_{t+h}\right].$
By definition $\rho(0)=\overline{\sigma}^2$, also \cite[equation (21)] {garnier2017correction}, integration by parts and recalling the definition of $\overline{\sigma}^2$, yield for $h\geq 0$,
\begin{align}\label{eq:formula_cov}
\rho(h)
&=\frac{\sigma^2\lambda}{4}
\int_{\mathbb R}e^{-\lambda|u|}|h+u|^{2\mathcal H}du
-\frac{\sigma^2}{2}h^{2\mathcal H}\notag\\ & =\frac{\sigma^2\lambda}{4}\left(
\frac{e^{-\lambda h}\Gamma(2\H+1)}{\lambda^{2\H+1}}+\int_{0}^{h}e^{-\lambda(h-u)}|u|^{2\mathcal H}du+\int_{h}^{+\infty }e^{-\lambda(u-h)}|u|^{2\mathcal H}du\right)\notag \\ & 
-\frac{\sigma^2}{2}h^{2\mathcal H}\notag \\ & =\frac{\overline{\sigma}^2e^{-\lambda h}}{2}-\frac{\sigma^2\H}{2}\left(
\int_0^h e^{\lambda (u-h)}u^{2\H-1}du-\int_{h}^{+\infty }e^{-\lambda (u-h)}|u|^{2\mathcal H-1}du\right).
\end{align}
Finally, equation \eqref{eq:formula_cov}, the elementary identity
\begin{align*}
\int_h^\infty e^{-\lambda u}u^{2\mathcal H-1}\,du
=\frac{\Gamma(2\mathcal H)}{\lambda^{2\mathcal H}}
-\int_0^h e^{-\lambda u}u^{2\mathcal H-1}\,du
\end{align*}
and the definition of $\overline{\sigma}^2$ imply
\begin{align}\label{final_identity_cov}
\rho(h)=\overline{\sigma}^2\cosh{(\lambda h)}-\sigma^2\H\int_{0}^h\cosh{\left(\lambda(h-u)\right)}u^{2\H-1}du.
\end{align}
We will also employ the following alternative representation of $\rho(h)$, cf. \cite[Remark 2.4]{Cheridito03}
\begin{align}\label{alternative_identity_cov}
\rho(h)=\frac{\sigma^2\Gamma(2\H+1)\sin{(\pi\H)}}{2\pi}\int_{\R} e^{ihu}\frac{|u|^{1-2\H}}{\lambda^2+u^2}du.  
\end{align}
After the above preliminary formulas we can run the computations.

By definition, $X_t=\overline{X}_t+e^{-\lambda t}x_0-e^{-\lambda t}\overline{X}_0$. Therefore, expanding the square and using \eqref{final_identity_cov} we get
\begin{align}\label{eq:explicit_square}
\mathbb{E}\left[|X_t|^2\right]&=e^{-2\lambda t}|x_0|^2-2e^{-\lambda t}\rho(t)+\overline{\sigma}^2(1+e^{-2\lambda t})\\ &=e^{-2\lambda t}|x_0|^2+2\sigma^2\mathcal H e^{-\lambda t}
\int_0^t\cosh\!\left(\lambda(t-s)\right)s^{2\mathcal H-1}\,ds\notag.
\end{align}
This is exactly \eqref{eq:secondMomentOU} expanding the definition of the hyperbolic cosine. Also from \eqref{eq:explicit_square}, \eqref{limit_variance} and $|\rho(t)|\leq \overline{\sigma}^2$ we directly obtain \eqref{rough_exponential_convergence_mean}.

Regarding \eqref{MomentOntheSup}, let $M=\sup_{t\in[0,T]}X_t$ and assume, without loss of generality, that $x_0=0$. Then,  \eqref{rough_exponential_convergence_mean}
implies that $\sup_{t\in [0,T]}\EE[X_t^2]\lesssim\frac{\sigma^2}{\lambda^{2\H}}$ and, thanks to Jensen's inequality, that 
\begin{equation*}
    \EE\sx[(X\st)^2\dx]\lesssim \lambda^2(t-s)\int_s^t\EE[X_u^2]du+\sigma^2(t-s)^{2\H}\lesssim
\sigma^2(\lambda^{2-2\H}\vee1)(t-s)^{2\H}
\end{equation*}
Therefore, Dudley's theorem \cite[Theorem 10.1] {lifshits2012lectures} implies that
\begin{align*}
 \EE\sx[M\dx]&\lesssim \int_0^{\frac{\sigma}{\lambda^\H}}\log^{\frac12}\sx(1+\frac{T\sigma^{\frac{1}{\H}}(\lambda^{\frac{1-\H}{\H}}\vee1)}{\epsilon^{\frac1\H}}\dx)d\epsilon\\
&\lesssim\frac{\sigma}{\lambda^\H}\log^{\frac12}\sx(1+T\sigma^{\frac{1}{\H}}(\lambda^{\frac{1-\H}{\H}}\vee1)\dx)+\int_0^ {\frac{\sigma}{\lambda^\H}}\log^{\frac12}(1+\epsilon^{-\frac{1}{\H}})d\epsilon.   
\end{align*}
Let $m:=\EE\sx[M\dx]$. By Borell's inequality \cite[Theorem 2.1]{adler1990introduction}, we have
\begin{align*}
        \EE\sx[M^2\dx]&\lesssim m^2+\int_0^\infty (u+m)\mathbb{P}\sx(M>u+m\dx)du\\
        &\lesssim m^2+\int_0^{\infty}(u+m)\exp\sx(-\frac{u^2\lambda^{2\H}}{2\sigma^2}\dx)du\\
        &\lesssim m^2+\frac{\sigma^2}{\lambda^{2\H}}
+\frac{\sigma}{\lambda^\H}m.
\end{align*}
Concerning \eqref{rough_power_convergence}, we start by proving it for the stationary process and then extend it to the non-stationary one. First observe that by Plancherel's identity and \eqref{alternative_identity_cov}, for $\H\in (\frac{1}{4}, 1/2]$ it holds 
\begin{align}\label{eq:l_2_cov}
    \int_{\R}\rho(h)^2 dh=\frac{\sigma^4\sin^2(\pi\H)\Gamma^2(2\H+1)}{2\pi\lambda^{4\H+1}}\int_{\R} \frac{|u|^{2-4\H}}{(1+|u|^2)^2}du.
\end{align}
 Also by Isserlis's theorem \cite{isserlis1918formula}, $\mathbb{E}[\overline{X}_s^2\overline{X}_r^2]-\overline{\sigma}^4=2\rho(s-r)^2.$
 Therefore
 \begin{align}\label{eq:convergence_rough_stationary}
     \expt{|\frac{1}{t}\int_0^t |\overline{X}_s|^2 ds-\overline{\sigma}^2|^2}&=\frac{2}{t^2}\int_0^t \int_0^t  \rho(s-r)^2 ds dr\notag\\ & \leq \frac{2}{t}\int_{\R}\rho(h)^2 dh \lesssim \frac{\sigma^4}{t\lambda^{4\H+1}},
 \end{align}
 proving \eqref{rough_power_convergence} for the stationary process. Also, by definition,
 \begin{align*}
     X_s^2-\overline{X}_s^2=e^{-2\lambda s}(x_0-\overline{X}_0)^2+2e^{-\lambda s}\overline{X}_s(x_0-\overline{X}_0).
 \end{align*}
Consequently, Gaussian hypercontractivity and H\"older's inequality imply
\begin{align*}
    \mathbb{E}\left[|X_s^2-\overline{X}_s^2|^2\right]& \lesssim e^{-2\lambda s}\left(|x_0|^4+\overline{\sigma}^4\right).
\end{align*}
The latter and Minkowski's integral inequality yield
\begin{align}\label{eq:distance_squares}
    \mathbb{E}\left[\left|\frac{1}{t}\int_0^t (X_s^2-\overline{X}_s^2)ds\right|^2\right]&\lesssim \frac{(|x_0|^2+\overline{\sigma}^2)^2}{\lambda^2 t^2}.
\end{align}
Combining \eqref{eq:convergence_rough_stationary} and \eqref{eq:distance_squares} we obtain \eqref{rough_power_convergence} by triangle inequality, completing the proof.
\end{proof}
\subsection{Infinite-dimensional case}
In this section we treat the infinite-dimensional framework. Let $w^{\epsilon}$ be the solution of 
\begin{align*}
    \begin{cases}
dw^{\epsilon}&=\epsilon^{-1}Cw^{\epsilon}dt+\epsilon^{-\H}dW^\H_t\\
        w^{\epsilon}(0)&=0.
    \end{cases}
\end{align*}
Under the conditions on $Q$ and $C$ described in \autoref{sec_hp_noise}, $w^{\epsilon}$ is well-defined and moreover the following holds:
\begin{lemma}\label{hp_stoch_conv}
    Let $\H\in (\frac{1}{4},1)$,\ for each $q\geq 2$
\begin{align*}
    \sup_{\epsilon\in (0,1],\ t\geq 0}\mathbb{E}[\norm{w^{\epsilon}_t}_{H^{\xi+1}}^q]\leq C_{q}.
\end{align*}
\end{lemma}
\begin{proof}
    Due to the assumptions, $C$ generates a strongly continuous semigroup 
    \begin{align*}
        S(t)=\sum_{i}e^{\lambda_i t}e_i\otimes e_i
    \end{align*}
    on $L^2$ with infinitesimal generator $C$. Since the $W^{\H,i}$ are independent, it holds
    \begin{align}\label{eq:quadratic_moment}
       \expt{\norm{ \sum_{i\geq 1} \int_0^t e^{\epsilon^{-1}\lambda_i(t-s)} \sigma_i e_i dW^{\H,i}_s}_{H^{\xi+1}}^2}=\sum_{i\geq 1}\sigma_i^2 \norm{e_i}_{H^{\xi+1}}^2 \mathbb{E}\left[| \int_0^t e^{\epsilon^{-1}\lambda_i(t-s)} dW^{\H,i}_s|^2\right]
    \end{align}
    Based on \eqref{eq:secondMomentOU},
    it holds
    \begin{align*}
       \mathbb{E}\left[| \int_0^t e^{\epsilon^{-1}\lambda_i(t-s)} dW^{\H,i}_s|^2\right]&=\H \left(\int_0^t  e^{\epsilon^{-1}\lambda_i z}| z|^{2\H-1} dz + \int_0^t e^{\epsilon^{-1}\lambda_i (2t-z)} | z|^{2\H-1} dz\right)\\ & \leq e^{\epsilon^{-1}\lambda_i t} \frac{\left(\epsilon^{-1}|\lambda_i |t\right)^{2\H}}{2\left(\epsilon^{-1}|\lambda_i| \right)^{2\H}}+\frac{\int_0^{+\infty} e^{-z}| z|^{2\H-1}dz}{(\epsilon^{-1}|\lambda_i|)^{2\H}}\\ & \lesssim \frac{1}{(\epsilon^{-1}|\lambda_i|)^{2\H}}.
    \end{align*}
    Since we assumed $\sup \lambda_{i}<0$ we conclude
    \begin{align*}
     \mathbb{E}\left[| \int_0^t e^{\epsilon^{-1}\lambda_i(t-s)} dW^{\H,i}_s|^2\right]& \lesssim \epsilon^{2\H}   
    \end{align*}
    and therefore for each $t\in [0,T],\ \epsilon\in (0,1)$
    \begin{align*}
        \mathbb{E}\left[\norm{w^{\epsilon}_t}_{H^{\xi+1}}^2\right]& \lesssim  \sum_{i\geq 1}\sigma_i^2 \norm{e_i}_{H^{\xi+1}}^2<+\infty
    \end{align*}
    and the claim follows due to Gaussianity.
\end{proof}
\begin{remark}
    As the proof shows it is enough to ask
    \begin{align*}
        \sum_{i\geq 1}\frac{\sigma_i^2}{| \lambda_i|^{2\H}} \norm{e_i}_{H^{\xi+1}}^2<+\infty,
    \end{align*}
    so if the operator $C$ gives some regularization we get less stringent restrictions. However since our main goal is to treat the case $C=-I$ we do not stress this and do not use any regularization coming from the semigroup generated by $C$ that ultimately only yields less stringent assumptions on the decay of the $\sigma_i$'s.
\end{remark}
Next we show, similarly to relation \eqref{rough_power_convergence}, the validity of the following
\begin{lemma}\label{lem_ito_stokes}
   Let $\H\in (\frac{1}{4},\frac{1}{2}]$ denoting by 
    \begin{align}\label{eq:OUsize1}
        \overline{r}:=\int_{L^2}(-C)^{-1}b(r,r)d\mu(r),\quad
        \begin{cases}
            dw&= Cwdt+dW^{\H}_t\\
            w(0)&=0,\ 
        \end{cases}
    \end{align}
    where $\mu$ is the centered Gaussian measure on $L^2$ with covariance operator 
    \begin{align*}
        \overline{Q}=\sum_{i\geq 1} \frac{\sigma_i^2 e_i\otimes e_i}{2|\lambda_i|^{2\H}}\int_0^{+\infty}e^{-s} s^{2\H} ds .
    \end{align*}
Then, for every $T>0,$
\begin{align*}
\mathbb{E}\left[
    \left\|
    \frac1T\int_0^T(-C)^{-1}b(w_s,w_s)\,ds-\overline r
    \right\|^2
\right]
&\lesssim
\frac{1}{T}\left(\sum_{i\geq1}\sigma_i^2\|e_i\|_{H^\xi}^2\right)^2.
\end{align*}
In particular,
\begin{align}\label{convergence_ITO_STOKES_COND}
        \frac{1}{T}\int_0^T (-C)^{-1}b(w_s,w_s)ds\rightarrow \overline{r} \quad \text{in } L^2(\Omega;H).
    \end{align}
\end{lemma}
\begin{proof}
    In order to ease the notation, let us call $X^{i}_t=\langle w_t,e_i\rangle$, due to the independence of the $W^{\H,i}$'s each $X^i$ is Gaussian and they are independent if $i\neq j$. We can therefore write the left hand side of \eqref{convergence_ITO_STOKES_COND} as 
    \begin{align*}
        \frac{1}{T}\sum_{i,j\geq 1}\int_0^T (- C)^{-1} b(e_i,e_j) X^{i}_s X^j_s ds,
    \end{align*}
    while, calling $\overline{\sigma}^2_i:=\frac{\sigma_i^2}{2|\lambda_i|^{2\H}}\int_0^{+\infty}e^{-s} s^{2\H} ds$, $\overline{r}$ satisfies
    \begin{align*}
        \overline{r}=\sum_{i}\overline{\sigma}_i^2(-C)^{-1}b(e_i,e_i).
    \end{align*}
    We also denote by
    \begin{align*}
        m_t&:=\mathbb{E}\left[(-C)^{-1}b(w_t,w_t)\right]=\sum_{i\geq 1} (- C)^{-1} b(e_i,e_i)\mathbb{E}\left[|X^{i}_t|^2\right],\\
        R^i(t,s)&:=\mathbb{E}\left[X^i_tX^i_s\right]\quad \text{for each }i\geq 1.
    \end{align*}
    It is easy to check that by assumptions on the boundedness of $(-C)^{-1}$ as operator on $H$, \eqref{eq:secondMomentOU} and the summability assumptions on the $\sigma_i$'s, we have $m_{t}\in L^2(0,T;H).$
    We therefore write
    \begin{align*}
       \frac1T\int_0^T(-C)^{-1}&b(w_s,w_s)\,ds-\overline r\\ & =\left(\frac1T\int_0^T(-C)^{-1}b(w_s,w_s)-m_s\,ds\right)+\left(\frac1T\int_0^T m_s\,ds-\overline r\right)\\ & =:I^1_T+I^2_T.
    \end{align*}
    By triangle inequality, it is enough to study the two terms separately. Equation \eqref{rough_exponential_convergence_mean} implies 
    \begin{align*}
        \|m_t-\overline{r}\|&\lesssim e^{\lambda_1 t}\sum_{i\geq 1}\sigma_i^2\|e_i\|_{H^\xi}^2 .
    \end{align*}
    Consequently, since by assumption $\lambda_1<0,$ we have
\begin{align}\label{eq:estimate_1_functional}
     \mathbb{E}\left[\|I^2_T\|^2\right] &\lesssim \left(1\wedge\frac{1}{T^2}\right)\left(\sum_{i\geq 1}\sigma_i^2 \|e_i\|_{H^\xi}^2 \right)^2.
    \end{align}
    For $I^1_T,$ recalling each $X^i$ is Gaussian and they are independent if $i\neq j$, we have
    by Isserlis's theorem
\begin{align}\label{eq:quadratic_infinite}
    &\mathbb{E}\left[\left\|\int_0^T(-C)^{-1}b(w_s,w_s)-m_s\,ds\right\|^2\right]\\ &\quad=\mathbb{E}\left[\int_0^T\int_0^T\langle(-C)^{-1}b(w_s,w_s)-m_s,(-C)^{-1}b(w_t,w_t)-m_t\rangle\,ds dt\right]\notag\\ & =\sum_{i,j\geq 1}\langle (-C)^{-1}b(e_i,e_j),(-C)^{-1}b(e_i,e_j)+(-C)^{-1}b(e_j,e_i) \rangle\int_0^T \int_0^T  R^i(t,s)R^j(t,s) \,ds dt.\notag
\end{align}
Since the $X^i$ are centered, we have by Isserlis's theorem, the elementary property that mean minimizes the mean-square error among all deterministic constants and 
\eqref{rough_power_convergence}
\begin{align}\label{eq:auxiliary_quartic}
    2\int_0^T \int_0^T \left(R^i(s,t)\right)^2ds dt&=\mathbb{E}\left[\left|\int_0^T |X^i_s|^2 ds-\mathbb{E}\left[\int_0^T|X^i_s|^2ds\right]\right|^2\right] \notag\\ & \leq \mathbb{E}\left[\left|\int_0^T |X^i_s|^2 ds-T\overline{\sigma}_i^2\right|^2\right]  \lesssim T \sigma_i^4
\end{align}
for $T\geq 1$. For $0<T\leq1$, the same inequality simply follows
from Cauchy-Schwarz and \eqref{eq:secondMomentOU}. Inserting this estimate in \eqref{eq:quadratic_infinite} yields by Cauchy-Schwarz 
    \begin{align}
    \mathbb{E}\left[\left\|\int_0^T(-C)^{-1}b(w_s,w_s)-m_s\,ds\right\|^2\right]&\lesssim T\sum_{i,j \geq 1}\sigma_i^2\sigma_j^2\|e_i\|_{H^\xi}^2\|e_j\|_{H^\xi}^2. 
    \end{align}
    Consequently
    \begin{align}\label{eq:estimate_2_functio}
        \mathbb{E}\left[\|I^1_T\|^2\right] &\lesssim \frac{1}{T}\left(\sum_{i\geq 1}\sigma_i^2 \|e_i\|_{H^\xi}^2 \right)^2.
    \end{align}
    Combining \eqref{eq:estimate_1_functional} and \eqref{eq:estimate_2_functio} the claims follow.
\end{proof}
\subsubsection{The centered case}\label{sec:centered}
As discussed in the introduction, in order to treat the special case $\H<\frac{1}{2},\ \alpha=1$ a further centering assumption on the coefficients $Q,\ C $ is required. It corresponds to forcing $\overline{r}$ as defined in \eqref{eq:OUsize1} to be identically $0$\footnote{This condition is automatically satisfied when $C$ and $Q$ are diagonal with respect to the standard divergence-free Fourier basis of sines and cosines on the torus, since the nonlinear self-interaction of each basis element vanishes.}, namely since the operator $C$ is invertible that
\begin{align}\label{eq:centering_cond}
    \sum_{i\geq 1}\frac{\sigma_i^2}{|\lambda_i|^{2\H}}b(e_i,e_i)=0.
\end{align}
In this case, the following lemma is used to control the second corrector.
\begin{lemma}\label{lemma_second_corrector}
   Assume that $\H\in \left(\frac{1}{3},\frac{1}{2}\right),\ \alpha=1$ and \eqref{eq:centering_cond} holds. Denote by $z(t)$ the solution of 
   \begin{align*}
      \begin{cases}
          \partial_t z&=C z+b(w,w)\\
          z(0)&=0,
      \end{cases} 
   \end{align*}
   where $w$ has been defined in \eqref{eq:OUsize1}. Then, for each $q\geq 1$ and $T\geq 0$, it holds
   \begin{align*}
     \sup_{t\geq 0}\mathbb{E}\left[\|z_t\|_{H^{\xi}}^q \right]&\leq C_{q},\\
     \sup_{t\geq 0}\mathbb{E}\left[\left\|\int_t^{t+T} z_udu\right\|_{H^{\xi}}^q \right]&\leq C_{q}\min\{T^q,T^{q/2}\}.
   \end{align*}
\end{lemma}
\begin{proof}
    The first claim immediately follows by the mild formulation using \eqref{hp_semigroup}, Minkowski's integral inequality and \autoref{hp_stoch_conv}. Indeed, for each $t \geq 0$,
    \begin{align*}
\mathbb{E}\left[\|z_t\|_{H^{\xi}}^q \right]&\lesssim \left(\int_0^t \left(\mathbb{E}\left[\|e^{C(t-s)}b(w_s,w_s)\|^q_{H^{\xi}}\right]\right)^{1/q}ds\right)^q \\ & \lesssim  \left(\int_0^t e^{-\mathfrak{h}(t-s)}\left(\mathbb{E}\left[\|w_s\|^{2q}_{H^{\xi+1}}\right]\right)^{1/q}ds\right)^q\leq C_q
    \end{align*}
    For the second one, we first study the forcing appearing in the equation for $z$ in more detail. Arguing similarly to \autoref{lem_ito_stokes}, it is possible to show that
    \begin{align}\label{aux_convergence}
\mathbb{E}\left[
    \left\|
    \int_0^Tb(w_s,w_s)\,ds
    \right\|_{H^\xi}^2
\right]
&\lesssim
T\left(\sum_{i\geq1}\sigma_i^2\|e_i\|_{H^{\xi+1}}^2\right)^2\quad \text{for }T\geq 1.
\end{align}
Indeed, denoting by $\overline{m}_t=\mathbb{E}[b(w_t,w_t)]$, equation \eqref{rough_exponential_convergence_mean} and assumption \eqref{eq:centering_cond} imply 
    \begin{align*}
        \|\overline{m}_t\|_{H^{\xi}}&\lesssim e^{\lambda_1 t}\sum_{i\geq 1}\sigma_i^2\|e_i\|_{H^{\xi+1}}^2 .
    \end{align*}
    Also, \eqref{eq:auxiliary_quartic} and analogous computations as there yield
    \begin{align*}
    \mathbb{E}\left[\left\|\int_0^T b(w_s,w_s)-\overline{m}_s\,ds\right\|_{H^\xi}^2\right]&\lesssim T\sum_{i,j \geq 1}\sigma_i^2\sigma_j^2\|e_i\|_{H^{\xi+1}}^2\|e_j\|_{H^{\xi+1}}^2. 
    \end{align*}
    Consequently, \eqref{aux_convergence} follows verbatim as in the proof of \autoref{lem_ito_stokes}. Since $\int_0^T b(w_s,w_s)\,ds$ belongs to the second
inhomogeneous Wiener chaos, Gaussian hypercontractivity and \eqref{aux_convergence} imply
     \begin{align}\label{aux_convergence_2}
\mathbb{E}\left[
    \left\|
    \int_0^Tb(w_s,w_s)\,ds
    \right\|_{H^\xi}^q
\right]
&\lesssim
T^{q/2}\left(\sum_{i\geq1}\sigma_i^2\|e_i\|_{H^{\xi+1}}^2\right)^q\quad \text{for }T\geq 1.
\end{align}
If $T\in [0,1]$ we can just apply Minkowski's integral inequality and \autoref{hp_stoch_conv} to obtain
\begin{align*}
    \mathbb{E}\left[
    \left\|
    \int_0^Tb(w_s,w_s)\,ds
    \right\|_{H^\xi}^q
\right]
&\lesssim
T^{q}\left(\sum_{i\geq1}\sigma_i^2\|e_i\|_{H^{\xi+1}}^2\right)^q\quad \text{for }T\in [0,1].
\end{align*}
Combining \eqref{aux_convergence_2} with the previous estimate yields
\begin{align}\label{aux_convergence_3}
     \mathbb{E}\left[
    \left\|
    \int_0^Tb(w_s,w_s)\,ds
    \right\|_{H^\xi}^q
\right]
&\lesssim
\min\{T^{q},T^{q/2}\}\left(\sum_{i\geq1}\sigma_i^2\|e_i\|_{H^{\xi+1}}^2\right)^q.
\end{align}
Next we upgrade \eqref{aux_convergence_3} to a varying initial time of the integral. Therefore, for $t\geq 0$, let 
\begin{align*}
    \tilde{w}^t_u:=\int_t^{t+u} e^{C(t+u-s)}dW^\H_s.
\end{align*}
By definition, it holds $w_{t+u}=\tilde{w}^t_u+e^{Cu}w_t$ and, by stationarity of the increments of the fractional Brownian motion, the process $\{\tilde{w}^t_{u}\}_{u\geq 0}$ has the same law as the process $\{w_{u}\}_{u\geq 0}$. Therefore, for each $t\geq 0$, \eqref{aux_convergence_3} holds with $w$ replaced by $\tilde{w}^t$, too. Moreover, a computation analogous to \eqref{eq:quadratic_moment}, yields
\begin{align*}
    \mathbb{E}\left[\|e^{Cu}w_t\|_{H^{\xi+1}}^2\right]&\lesssim e^{2\lambda_1 u}\sum_{i\geq1}\sigma_i^2\|e_i\|_{H^{\xi+1}}^2,
\end{align*}
consequently, by Gaussianity
\begin{align}\label{aux_convergence_4}
    \mathbb{E}\left[\|e^{Cu}w_t\|_{H^{\xi+1}}^q\right]&\lesssim e^{q\lambda_1 u}\left(\sum_{i\geq1}\sigma_i^2\|e_i\|_{H^{\xi+1}}^2\right)^{q/2}.
\end{align}
\autoref{hp_stoch_conv} and \eqref{aux_convergence_4}, then imply
\begin{align*}
    \int_0^T \mathbb{E}\left[\|b(w_{t+u},w_{t+u})-b(\tilde{w}^t_u,\tilde{w}^t_u)\|_{H^\xi}^q\right]^{1/q} du&\lesssim \min\{T,1\}\sum_{i\geq1}\sigma_i^2\|e_i\|_{H^{\xi+1}}^2.
\end{align*}
The latter and \eqref{aux_convergence_3} yield, by triangle inequality and Minkowski's integral inequality,
\begin{align}\label{final_aux}
      \sup_{t\geq 0}\mathbb{E}\left[
    \left\|
    \int_t^{T+t}b(w_s,w_s)\,ds
    \right\|_{H^\xi}^q
\right]
&\lesssim
\min\{T^{q},T^{q/2}\}\left(\sum_{i\geq1}\sigma_i^2\|e_i\|_{H^{\xi+1}}^2\right)^q.
\end{align}
We are now ready to prove the second claim. By Fubini's theorem and the mild formulation
\begin{align*}
    \int_t^{T+t} z_u du=\int_0^{T+t}\int_{\max\{0,t-v\}}^{t+T-v} e^{vC}b(w_u,w_u)du dv.
\end{align*}
Since $t+T-v-\max\{0,t-v\}\leq T$, \eqref{final_aux}, Minkowski's integral inequality and \autoref{hp_stoch_conv} yield the final claim. Indeed,
\begin{align*}
\mathbb{E}\left[\left\|\int_t^{t+T} z_udu\right\|_{H^{\xi}}^q\right]&\lesssim \left(\int_0^{T+t} e^{-\mathfrak{h}v}\left(\mathbb{E}\left[\left\|\int_{\max\{0,t-v\}}^{t+T-v} b(w_u,w_u)du \right\|^q_{H^\xi}\right]\right)^{1/q}dv\right)^q\\ & \lesssim \min\{T^{q},T^{q/2}\}\left(\sum_{i\geq1}\sigma_i^2\|e_i\|_{H^{\xi+1}}^2\right)^q.
\end{align*}
\end{proof}
The second claim of \autoref{lemma_second_corrector} leads to the following useful result to study the unbounded rough drivers convergence of \autoref{sec_unbounder_rough} in the centered case.
\begin{corollary}\label{cor:second_corrector}
    Under the same assumptions of \autoref{lemma_second_corrector}, for each $q\geq 2,\ \beta\in \left(\frac{1}{2},2\H\right)$ and $\theta\in (\beta,2\H)$ such that $\beta<\theta-1/q$, $h^\epsilon_{t}:=\epsilon^{2\H-1}\int_0^t z^{\epsilon}_s ds$ satisfies
    \begin{align}\label{eq:hHolder}
\mathbb{E}\left[\|h^{\epsilon}\|_{C^{\beta}([0,T];H^\xi)}^q\right]&\lesssim \epsilon^{q(2\H-\theta)}.
    \end{align}
    Moreover, it holds
    \begin{align}\label{eq:a_priori_z}
        \sup_{\epsilon\in (0,1],\ t\geq 0}\mathbb{E}\left[\|z^{\epsilon}_t\|^q_{H^\xi}\right]\leq C_q.
    \end{align}
\end{corollary}
\begin{proof}
    Self-similarity of the fractional Brownian motion gives the equality in joint law of the processes
$(w_t^\epsilon,z_t^\epsilon)_t$ and $(w_{t/\epsilon},z_{t/\epsilon})_t$. Consequently, \eqref{eq:a_priori_z} follows from \autoref{lemma_second_corrector}. \autoref{lemma_second_corrector} also implies, for each $0\leq t_1\leq t_2\leq T$,
 \begin{align*}
\mathbb{E}\left[\|h^{\epsilon}_{t_2}-h^{\epsilon}_{t_1}\|_{H^\xi}^q\right]&\lesssim \epsilon^{2q\H}\min\left\{\frac{t_2-t_1}{\epsilon},\sqrt{\frac{t_2-t_1}{\epsilon}}\right\}^q&\lesssim \epsilon^{q(2\H-\theta)}|t_2-t_1|^{\theta q}.
    \end{align*}
    The claim \eqref{eq:hHolder} follows by Kolmogorov's continuity theorem.
\end{proof}

\section{Unbounded rough drivers convergence}\label{sec_unbounder_rough}
Following \autoref{sec_rough_formulation}, we aim to study the convergence, as $\epsilon\to 0$, of the canonical rough path lifts associated with
\begin{equation}\label{def:y}
    y^{\epsilon}_t=\int_0^t\epsilon^{-\alpha+\H}w^{\epsilon}_r\,dr \,\text{ and }\, \hat{y}_t^\epsilon=\int_0^t\epsilon^{\H-1}w^\epsilon_r+\epsilon^{2\H-1}z^\epsilon_rdr
\end{equation}
We first focus on $y^\epsilon$, for which we must treat separately the cases $\H<\frac12$ and $\H>\frac12$. Afterwards, assuming $\H<\frac12$, we turn to the analysis of $\hat{y}^\epsilon_t$.

\subsection{$\mathbf{Y}^\epsilon$ convergence}
Since lower Hurst parameters entail a rougher fBm and hence a more intricate rough path structure, this analysis becomes more delicate when $\H<1/2$. We therefore begin with the simpler regime $\H>1/2$, adopting the strategy of \cite{debussche2023rough}. The more involved case will instead rely on specific properties of Gaussian rough paths (see \cite{friz2010multidimensional,friz2014course}).

\begin{proposition}\label{proproughHgreat12}
    Let $\H>1/2$, $\alpha=1$. Then for $\epsilon\to 0$, $(y^{\epsilon}_t)_t$ converges to $(-C)^{-1}W^\H_t$ in $L^q(\Omega,C^{p\var}(H^\xi))$ for any $q\in [2,\infty)$ and $p>\frac1\H$.

    Moreover, the following map 
    \begin{equation*}
        \epsilon\in(0,1) \mapsto y^{\epsilon}\in L^q(\Omega,C^{p\var}(H^\xi))
    \end{equation*}
    is H\"older-continuous and it holds that
    \begin{equation}\label{est:UniformRegularityURDH>1/2}
        \mathbb{E}\left[\sup_{\epsilon\in (0,1)}\|y^\epsilon\|_{p\var,[0,T]}^q\right]\lesssim 1.
    \end{equation}
\end{proposition}
\begin{proof}
\textit{First step: Convergence.}
    Let us highlight the following two properties:
    \begin{enumerate}
        \item By the self-similarity of fBm, the process $z_t = w^{\epsilon}_{\epsilon t}$ satisfies the equation $dz_t = C z_t\,dt + d\overline{W}^\H_t$, where $\overline{W}^\H_t$ is distributed as $W^\H$. 
        \item Since $\alpha=1$, then \begin{equation}
            y^{\epsilon}_t=(-C)^{-1}[W_t^\H-\epsilon^{\H}w^{\epsilon}_t].
        \end{equation}
    \end{enumerate}
    Therefore, by applying \autoref{hp_stoch_conv}, we obtain that
    \begin{equation}\label{est:Interp1H>1/2}
        \mathbb{E}\sx[\|y^{\epsilon}\st-(-C)^{-1}W^\H\st\|^2_{H^\xi}\dx]=\mathbb{E}\sx[\|(-C)^{-1}w^{\epsilon}\st\|_{H^{\xi}}^2 \dx]\epsilon^{2\H}\lesssim\epsilon^{2\H}.
    \end{equation}
    Now, we aim at showing 
    \begin{equation}\label{est:Interp2H>1/2}
        \mathbb{E}\sx[\|y^{\epsilon}\st-(-C)^{-1}W^\H\st\|^2_{H^\xi}\dx]\lesssim |t-s|^{2\H}.
    \end{equation}
    For $r\geq 0$ denote $\overline{z}^s_{r}:=z_{r+s}-z_s$. Then, $\overline{z}^s_{r}$ solves an Ornstein-Uhlenbeck SPDE with additional time-independent random forcing:
    \begin{equation*}
\overline{z}^s_r=\int_0^rC\,\overline{z}^s_u\,du+Cz_sr+W^\H_{s+r}-W^\H_s=\int_0^rC\,\overline{z}^s_u\,du+Cz_sr+\overline{W}^{\H,s}_r,
    \end{equation*}
    where $\overline{W}^{\H,s}_r:=W^\H_{s+r}-W^\H_s$ has the same law as $W^\H_r$. Therefore by exploiting the assumptions on $C$ and equation \eqref{eq:secondMomentOU}, we have that 
    \begin{equation}
        \begin{aligned}
            \EE\left[\|(-C)^{-1}\overline{z}^s_r\|_{H^{\xi}}^2\right]&\lesssim\sum_{k}\left|\int_0^re^{\lambda_k(r-u)}du\right|^2\lambda_k^{-2}\EE\left[\langle e_k,Cz_s\rangle^2\right]\|e_k\|_{H^{\xi}}^2\\
            &+\sum_k\EE\left[\left\|\int_0^re^{\lambda_k(r-v)}\sigma_ke_k\,d\overline{W}^{\H,s,k}_v\right\|_{H^{\xi}}^2\right]\\
            &\lesssim \sum_k\lambda_k^{-2}|1-e^{\lambda_kr}|^2\EE\left[\langle e_k,z_s\rangle^2\right]\|e_k\|_{H^{\xi}}^2+r^{2\H}\sum_k\sigma_k^2\|e_k\|_{H^\xi}^2\\
            &\lesssim r^{2\H}\left(\EE\left[\|z_s\|_{H^{\xi}}^2+\sum_k\sigma_k^2\|e_k\|_{H^\xi}^2\right]\right).
        \end{aligned}
    \end{equation}
    Then, inequality \eqref{est:Interp2H>1/2} follows from
    \begin{equation}
        \begin{aligned}
            \EE\left[\|(-C)^{-1}\epsilon^{\H}w^{\epsilon}\st\|_{H^{\xi}}^2\right]\lesssim\epsilon^{2\H}\EE\left[\|(-C)^{-1}\overline{z}_{\epsilon^{-1}(t-s)}^{\epsilon^{-1}s}\|_{H^{\xi}}^2\right]\lesssim \left(t-s\right)^{2\H}.
        \end{aligned}
    \end{equation}
By interpolating  \eqref{est:Interp1H>1/2} with  \eqref{est:Interp2H>1/2}, and using the Gaussian property, we obtain for any $\chi \in (0,1)$
\begin{equation}\label{ineq1}\mathbb{E}\sx[\|y^{\epsilon}\st-(-C)^{-1}W^\H\st\|^q_{H^\xi}\dx]^{1/q}\lesssim \epsilon^{\H\chi}(t-s)^{\H(1-\chi)},\end{equation}
that implies convergence by means of Kolmogorov's theorem \cite[Theorem 3.3]{friz2010multidimensional}.

\textit{Second step: Proof of \eqref{est:UniformRegularityURDH>1/2}.}\\
 Notice that by integration by parts and by imposing $W^{k,\H}\equiv 0$ on $(-\infty,0]$, we have that
    \begin{align}\label{e:mollifiedForm}
     y^\epsilon_t&=\sum_k\frac{\sigma_ke_k}{|\lambda_k|}(\epsilon^{-1}|\lambda_k|)\int_0^tW^{k,\H}_re^{\lambda_k\epsilon^{-1}(t-r)}dr \notag\\
    &=\sum_k \frac{\sigma_ke_k}{|\lambda_k|}(\phi^{\epsilon}_k\ast W^{k,\H})(t)=:\sum_k \frac{\sigma_ke_k}{|\lambda_k|}W^{\epsilon,k,\H}_t,
    \end{align}
    where $\phi^{\epsilon}_k(r)=\mathbf{1}_{[0,T]}(r)e^{r\lambda_k\epsilon^{-1}}|\lambda_k|\epsilon^{-1}.$\\
For $\epsilon<\delta$, it holds
\begin{equation*}
    W^{\epsilon,k,\H}_t-W^{\delta,k,\H}_t=z^{\delta,k}_t-z^{\epsilon,k}_t=:v^{\epsilon,\delta,k}_t,
\end{equation*}
where $z^{\epsilon,k}_t$ solves \begin{equation*}
dz^{\epsilon,k}_t=\lambda_k\epsilon^{-1}z_t^{\epsilon,k}dt+dW^{k,\H}_t,
\end{equation*} with $0$ as starting condition. While, $v^{\epsilon,\delta}_t$ solves \begin{equation*}
dv^{\epsilon,\delta,k}_t=\lambda_k\delta^{-1}v^{\epsilon,\delta,k}_tdt+\lambda_kz^{\epsilon,k}_t(\delta^{-1}-\epsilon^{-1})dt.
\end{equation*}
Therefore,
\begin{equation}\label{stima0}
\EE\sx[\|y^{\epsilon}\st-y^{\delta}\st\|_{H^{\xi}}^2\dx]\le \sum_k\frac{\sigma_k^2\|e_k\|_{H^{\xi}}^2}{\lambda_k^2}\EE\sx[(v^{\epsilon,\delta,k}\st)^2\dx]
\end{equation}
Now,
\begin{align}\label{smallnessEpsDelta}
      \EE\sx[(v^{\epsilon,\delta,k}\st)^2\dx]&\lesssim (\delta-\epsilon)^2\epsilon^{-4}\EE\sx[\sx(\int_s^te^{\lambda_k\delta^{-1}(t-r)}\lambda_kz_r^{\epsilon,k}dr\dx)^2\dx]\notag\\
      &+(\delta-\epsilon)^2\epsilon^{-4}\EE\sx[\sx(\int_0^s \left(e^{\lambda_k\delta^{-1}(t-r)}-e^{\lambda_k\delta^{-1}(s-r)}\right)\lambda_kz_r^{\epsilon,k}dr\dx)^2\dx]\notag\\
      &\lesssim (\delta-\epsilon)^2\epsilon^{-4}\EE\sx[\sup_{t\in [0,T]}|z_t^{\epsilon,k}|^2\dx]\delta^2\sx(1-e^{\lambda_k\delta^{-1}(t-s)}\dx)^2\notag\\
      &\lesssim (\delta-\epsilon)^2\epsilon^{-4}\lambda_k^2(t-s)^2,
    \end{align}
where, thanks to \eqref{MomentOntheSup},
\begin{equation*}
    \EE\sx[\sup_{t\in [0,T]}|z_t^{\epsilon,k}|^2\dx]\lesssim \frac{\epsilon^{2\H}}{|\lambda_k|^{2\H}}\log\sx(1+T(\epsilon^{-1}|\lambda_k|)^{\frac{1-\H}{\H}}\dx)\lesssim1.
\end{equation*}

Combining \eqref{stima0}, \eqref{smallnessEpsDelta} and Kolmogorov continuity theorem with respect to the temporal variable, we end up with
\begin{equation}\label{ineq2}\EE\left[\|y^\epsilon-y^\delta\|_{p\var,[0,T]}^q\right]^{1/q}\lesssim(\delta-\epsilon) \epsilon^{-2}.\end{equation}
At last, let $y^0=(-C)^{-1}W^\H$. Then for $\epsilon,\delta \in [0,1]$, we have that 
\begin{enumerate}
    \item if $2\epsilon<\delta$, relation \eqref{ineq1} implies 
    \begin{equation*}
        \EE\left[\|y^\epsilon-y^\delta\|_{p\var,[0,T]}^q\right]^{1/q}\lesssim\epsilon^{\chi\H}+\delta^{\chi\H}\lesssim (\delta-\epsilon)^{\chi \H};
    \end{equation*}
    \item if $\delta \in (\epsilon,2\epsilon)$ and $\epsilon\le (\delta-\epsilon)^{\nu}$, relation \eqref{ineq1} implies 
    \begin{equation*}
        \EE\left[\|y^\epsilon-y^\delta\|_{p\var,[0,T]}^q\right]^{1/q}\lesssim\epsilon^{\chi\H}\lesssim (\delta-\epsilon)^{\chi \H\nu};
    \end{equation*}
    \item if $\delta \in (\epsilon,2\epsilon)$ and $\epsilon> (\delta-\epsilon)^{\nu}$, thanks to relation \eqref{ineq2} we have
    \begin{equation*}
        \EE\left[\|y^\epsilon-y^\delta\|_{p\var,[0,T]}^q\right]^{1/q}\lesssim\epsilon^{-2}(\delta-\epsilon)\lesssim (\delta-\epsilon)^{1-2\nu}.
    \end{equation*}
\end{enumerate}
Therefore, the claim follows from properly choosing $\nu$ and applying the Kolmogorov theorem w.r.t. the $\epsilon$ parameter.
\end{proof}

\begin{proposition}\label{prop_rough_pathless12}
    Let $1/3<\H<1/2$, $\alpha = 1$, $q\in [4,\infty)$, $\gamma>\frac{1}{2\H}$ s.t. $2\H + 1/\gamma > 1$, and $3>p>2\gamma$ is such that \eqref{eq:p_summability} holds. Then, as $\epsilon \to 0$, the canonical rough path lift of $(y_t^\epsilon)_t$, denoted by $\mathbf{Y}^\epsilon$, converges in $L^q(\Omega)$ to a $p$-rough path on $H^\xi$, denoted by $\mathbf{W}^\H=(\mathbb{W}^{\H,1},\mathbb{W}^{\H,2})$.
    In particular
    \begin{align}
    \langle\mathbb{W}^{\H,1}\st,e_i\rangle&\label{eq:W^1}=\frac{\sigma_i}{|\lambda_i|}W^{\H,i}\st,\\
    \langle\mathbb{W}^{\H,2}\st,e_{i_1}\otimes e_{i_2}\rangle_{H^{\otimes 2}}&\label{eq:W^2}=\frac{\prod_{j=1}^2\sigma_{i_j}}{|\prod_{j=1}^2\lambda_{i_j}|}\int_s^t\int_s^{r_{2}}dW^{\H,i_1}_{r_1} dW^{\H, i_2}_{r_2},
    \end{align}
   where 
$
\int_s^t\int_s^{r_{2}}dW^{\H,i_1}_{r_1} dW^{\H, i_2}_{r_2}
$
agrees with the corresponding component of the canonical geometric lift of \((W^{\H,i_1}, W^{\H,i_2})\); see \cite{friz2010multidimensional}.
    Moreover, the following map 
    \begin{equation*}
        \epsilon\in(0,1) \mapsto \mathbf{Y}^{\epsilon}\in L^q\left(\Omega,C^{p\var}(H^\xi)\times C^{p/2\var}_2((H^\xi)^{\otimes 2})\right)
    \end{equation*}
    is H\"older-continuous and it holds that
    \begin{equation}\label{est:UniformRegularityURDH<1/2}
        \mathbb{E}\left[\sup_{\epsilon\in (0,1)}\sum_{n=1}^2\|\mathbb{Y}^{\epsilon,n}\|_{p/n\var,[0,T]}^q\right]\lesssim 1.
    \end{equation}
\end{proposition}
\begin{proof}
It suffices to establish the claims for arbitrarily large $q$,
since lower moments follow from H\"older's inequality.
    \textit{First step: Convergence.} Recall equation \eqref{e:mollifiedForm}, namely that
    \begin{equation}
    \begin{aligned} y^\epsilon_t
    &=\sum_k \frac{\sigma_ke_k}{|\lambda_k|}(\phi^{\epsilon}_k\ast W^{k,\H})(t)=:\sum_k \frac{\sigma_ke_k}{|\lambda_k|}W^{\epsilon,k,\H}_t,
    \end{aligned}
    \end{equation}
    where $\phi^{\epsilon}_k(r)=\mathbf{1}_{[0,T]}(r)e^{r\lambda_k\epsilon^{-1}}|\lambda_k|\epsilon^{-1}.$\\
Let us estimate separately the different levels of the rough paths. Let us define 
\begin{equation*}
\overline{\sigma}=\sum_k\frac{\sigma_k^2\|e_k\|_{H^{\xi}}^2}{\lambda_k^2},\quad \hat{\sigma}=\left(\sum_{k}\frac{|\sigma_k\|e_k\|_{H^{\xi}}|^{p/2}}{|\lambda_k|^{p/2}}\right)^2.
\end{equation*}
For the first level,
\begin{align}\label{i:compWise1st}
    \EE&\sx[\sx(\sup_{P\in\mathbf{\Pi}}\sum_{s,t\in P}\|\mathbb{Y}^{\epsilon,1}\st-\mathbb{W}^{\H,1}\st\|_{H^\xi}^p\dx)^{q/p}\dx]\notag\\&=\EE\sx[\sx(\sup_{P\in\mathbf{\Pi}}\sum_{s,t\in P}\sx(\sum_k \frac{\sigma_k^2\|e_k\|_{H^{\xi}}^2}{\lambda_k^2}|W^{\epsilon,k,\H}\st-W^{k,\H}\st|^2\dx)^{p/2}\dx)^{q/p}\dx]\notag\\
        &\le \overline{\sigma}^{\frac{q}{p}\frac{p-2}{2}}\EE\sx[\sx(\sup_{P\in\mathbf{\Pi}}\sum_{s,t\in P}\sum_k \frac{\sigma_k^2\|e_k\|_{H^{\xi}}^2}{\lambda_k^2}|W^{\epsilon,k,\H}\st-W^{k,\H}\st|^p\dx)^{q/p}\dx]\notag\\
        &\le \overline{\sigma}^{\frac{q}{p}\frac{p-2}{2}}\EE\sx[\sx(\sum_k \frac{\sigma_k^2\|e_k\|_{H^{\xi}}^2}{\lambda_k^2}\sx(\sup_{P\in\mathbf{\Pi}}\sum_{s,t\in P}|W^{\epsilon,k,\H}\st-W^{k,\H}\st|^p\dx)\dx)^{q/p}\dx]\notag\\
        &\le \overline{\sigma}^{\frac{q}{2}-1}\sum_k \frac{\sigma_k^2\|e_k\|_{H^{\xi}}^2}{\lambda_k^2}\EE\sx[\sx(\sup_{P\in\mathbf{\Pi}}\sum_{s,t\in P}|W^{\epsilon,k,\H}\st-W^{k,\H}\st|^p\dx)^{q/p}\dx],
\end{align}
where $\mathbf{\Pi}$ denotes the set of partitions of $[0,T]$, while the first and last inequalities are due to Jensen's inequality.\\
Now, we estimate $\EE\sx[\|W^{\epsilon,k,\H}-W^{k,\H}\|_{p\var,[0,T]}^q\dx]$ by means of \cite[Theorem 5]{friz2014convergence}. 
Following the notation in \cite{friz2014convergence}, let $C_0^{p\var}$ be the set of functions with finite $p$-variation starting from 0 and consider $\Lambda_\epsilon^k:C_0^{p\var}([0,T],\R)\to C_0^{1\var}([0,T],\R) $ defined by
\begin{equation*}
    \Lambda_{\epsilon}^k(x)(t)=(\phi^\epsilon_k \ast x)(t).
\end{equation*}
Clearly, $\Lambda_\epsilon$ is continuous. Indeed,
\begin{equation}\label{LambdaCont}
    |\partial_t\Lambda_{\epsilon}^k(x)(t)|\le 2|\lambda_k|\epsilon^{-1}\|x\|_0\le 2|\lambda_k|\epsilon^{-1}\|x\|_{p\var},
\end{equation}
where $\|x\|_0:=\sup_{t\in [0,T]}|x_t|.$
Moreover, introducing the control $\omega_x(s,t)=\|x\|_{p\var,[s,t]}^p$ (see \cite{friz2014course}) and splitting the integral at
$t\wedge(\epsilon^{1/2}/|\lambda_k|)$, we obtain
\begin{align}\label{LambdaUnifConv}
|\Lambda_\epsilon^k(x)(t)-x(t)|
&\leq |x_{0,t}|e^{\lambda_k\epsilon^{-1}t}
 +\int_0^t |x_{t-r,t}|\phi^\epsilon_k(r)\,dr
 \notag\\
&\leq \omega_x(0,t)^{1/p}e^{\lambda_1\epsilon^{-1}t}
 +\omega_x\left(
      \left(t-\frac{\epsilon^{1/2}}{|\lambda_1|}\right)_+,t
   \right)^{1/p}
 +2\|x\|_0e^{-\epsilon^{-1/2}}.
\end{align}
The right-hand side tends to zero uniformly in $t\in[0,T]$,
by continuity of the control and the condition $x(0)=0$.
At last, we need that \begin{align*}
    \sup_{\epsilon,\delta \in (0,1]}\|R_{(\Lambda^k_{\epsilon}(W^{k\H}),\Lambda^k_{\delta}(W^{k\H}))}\|_{1/2\H\var,[0,T]^2}\lesssim\|R_{W^{k,\H}}\|_{1/2\H\var,[0,T]^2},
\end{align*} where we recall that $R$ is defined at \eqref{eq:CovGauss}. Notice that a similar inequality follows from \cite[Proposition 5.68]{friz2010multidimensional}.
Therefore, \cite[Theorem 5]{friz2014convergence} guarantees that
\begin{align}\label{i:smallness1st}
    \EE\big[\|W^{\epsilon,k,\H}-W^{k,\H}\|_{p\var,[0,T]}^q\big]&\lesssim \sup_{t\in [0,T]}\EE\sx[|\Lambda_\epsilon^k(W^{k,\H})_t-W^{k,\H}_t|^2\dx]^{\frac{q}{2}\sx(1-\frac{1}{2\H\gamma}\dx)}\notag\\ &\lesssim\epsilon^{\frac{2\H q}{2}\sx(1-\frac{1}{2\H\gamma}\dx)},
\end{align}
 with the latter inequality derived in the same manner as for \eqref{est:Interp1H>1/2}.
Therefore, by combining inequality \eqref{i:compWise1st} and inequality \eqref{i:smallness1st}, we have that
\begin{equation}\label{stima1}
    \EE\sx[\|\mathbb{Y}^{\epsilon,1}-\mathbb{W}^{\H,1}\|_{p\var,[0,T]}^q\dx]\lesssim \epsilon^{\frac{2\H q}{2}\sx(1-\frac{1}{2\H\gamma}\dx)}\overline{\sigma}^{\frac{q}{2}}.
\end{equation}
Regarding the second level of the rough paths, we have that
\begin{align}\label{i:compWise2st}
    \EE&\sx[\sx(\sup_{P\in\mathbf{\Pi}}\sum_{s,t\in P}\|\mathbb{Y}^{\epsilon,2}\st-\mathbb{W}^{\H,2}\st\|_{(H^\xi)^{\otimes 2}}^{\frac{p}{2}}\dx)^{2q/p}\dx]\notag\\&=\EE\sx[\sx(\sup_{P\in\mathbf{\Pi}}\sum_{s,t\in P}\sx(\sum_{k,j} \frac{\sigma_k^2\sigma_j^2\|e_k\|_{H^{\xi}}^2\|e_j\|_{H^{\xi}}^2}{\lambda_k^2\lambda_j^2}|\mathbb{W}^{\epsilon,k,j,\H}\st-\mathbb{W}^{k,j,\H}\st|^2\dx)^{p/4}\dx)^{2q/p}\dx]\notag\\
        &\le \EE\sx[\sx(\sup_{P\in\mathbf{\Pi}}\sum_{s,t\in P}\sum_{k,j} \frac{|\sigma_k\sigma_j\|e_k\|_{H^{\xi}}\|e_j\|_{H^{\xi}}|^{p/2}}{|\lambda_k\lambda_j|^{p/2}}|\mathbb{W}^{\epsilon,k,j,\H}\st-\mathbb{W}^{k,j,\H}\st|^{p/2}\dx)^{2q/p}\dx]\notag\\
        &\le \EE\sx[\sx(\sum_{k,j} \frac{|\sigma_k\sigma_j\|e_k\|_{H^{\xi}}\|e_j\|_{H^{\xi}}|^{p/2}}{|\lambda_k\lambda_j|^{p/2}}\sup_{P\in\mathbf{\Pi}}\sum_{s,t\in P}|\mathbb{W}^{\epsilon,k,j,\H}\st-\mathbb{W}^{k,j,\H}\st|^{p/2}\dx)^{2q/p}\dx]\notag\\
        &\le 
        \hat{\sigma}^{\frac{2q-p}{p}}
        \sum_{k,j} \frac{|\sigma_k\sigma_j\|e_k\|_{H^{\xi}}\|e_j\|_{H^{\xi}}|^{p/2}}{|\lambda_k\lambda_j|^{p/2}}\EE\sx[\sx(\sup_{P\in\mathbf{\Pi}}\sum_{s,t\in P}|\mathbb{W}^{\epsilon,k,j,\H}\st-\mathbb{W}^{k,j,\H}\st|^{p/2}\dx)^{2q/p}\dx]
\end{align}
where  $\mathbb{W}^{\epsilon,k,j,\H}$ (resp. $\mathbb{W}^{k,j,\H}$) denotes the 2nd level of the rough path associated with $\Lambda_\epsilon^k(W^{k,\H})$ (resp. $W^{k,\H}$), for $j=k$. Instead, if $j\neq k$
\begin{align*}
\mathbb{W}^{\epsilon,k,j,\H}&=\langle\pi_2(\Lambda_\epsilon^k(W^{k,\H}),\Lambda_\epsilon^j(W^{j,\H})),\delta_{1,2}\rangle_{\R^{2\times 2}},\\ \mathbb{W}^{k,j,\H}&=\langle\pi_2(W^{k,\H},W^{j,\H}),\delta_{1,2}\rangle_{\R^{2\times 2}}
\end{align*}
where $\pi_2(\Lambda_\epsilon^k(W^{k,\H}),\Lambda_\epsilon^j(W^{j,\H}))$ (resp. $\pi_2(W^{k,\H},W^{j,\H})$) denotes the 2nd level of the rough path associated with  $(\Lambda_\epsilon^k(W^{k,\H}),\Lambda_\epsilon^j(W^{j,\H}))$ (resp. $(W^{k,\H},W^{j,\H})$), and $\delta_{1,2}=v_1\otimes v_2$, for $\{v_1,v_2\}$ canonical $\R^2$'s basis.

By trivial generalizations of inequalities \eqref{LambdaCont} and \eqref{LambdaUnifConv}, we have that $(\Lambda_\epsilon^k,\Lambda_\epsilon^j)$ is continuous and satisfies uniform convergence as $\epsilon$ vanishes. Therefore, \cite[Theorem 5]{friz2014convergence} implies that
\begin{equation}\label{i:Smallness2nd}
\begin{aligned}
    \EE\sx[\|\mathbb{W}^{\epsilon,k,j,\H}-\mathbb{W}^{k,j,\H}\|_{p/2\var, [0,T]}^q\dx]\lesssim\epsilon^{\frac{2\H q}{2}\sx(1-\frac{1}{2\H\gamma}\dx)}.
\end{aligned}
\end{equation}
Then, inequalities \eqref{i:compWise2st} and \eqref{i:Smallness2nd} imply that
\begin{equation}\label{stima2}
    \EE\sx[\|\mathbb{Y}^{\epsilon,2}-\mathbb{W}^{\H,2}\|_{p/2\var,[0,T]}^{q}\dx]\lesssim \epsilon^{\frac{2\H q}{2}\sx(1-\frac{1}{2\H\gamma}\dx)}
        \hat{\sigma}^{\frac{2q}{p}}.
\end{equation}
\textit{Second step: Proof of \eqref{est:UniformRegularityURDH<1/2}.}

Let $\epsilon<\delta<1.$ Then,
\begin{align}\label{i:Holder1st}
    \EE&\sx[\|\mathbb{Y}^{\epsilon,1}-\mathbb{Y}^{\delta,1}\|_{p\var,[0,T]}^q\dx]\notag\\
    &\lesssim  \overline{\sigma}^{\frac{q}{2}-1}\sum_k \frac{\sigma_k^2\|e_k\|_{H^{\xi}}^2}{\lambda_k^2}\EE\sx[\|W^{\epsilon,k,\H}-W^{\delta,k,\H}\|^q_{p\var,[0,T]}\dx]\notag\\
    &\lesssim   \overline{\sigma}^{\frac{q}{2}-1}\sum_k \frac{\sigma_k^2\|e_k\|_{H^{\xi}}^2}{\lambda_k^2}\sup_{t\in [0,T]}\EE\sx[|W^{\epsilon,k,\H}_t-W^{\delta,k,\H}_t|^2\dx]^{\frac{q}{2}\sx(1-\frac{1}{2\H\gamma}\dx)},
\end{align}
where the first inequality follows from replicating estimate \eqref{i:compWise1st}, while the last one is due to \cite[Corollary 15.32]{friz2010multidimensional} (specifically, \cite[Remark 15.33]{friz2010multidimensional} with $\rho=1/2\H$ and $\rho'=\gamma$). By following \eqref{smallnessEpsDelta} and by recalling that $z^{\epsilon,k}_t$ and $v^{\epsilon,\delta,k}_t$ are solutions to
\begin{align*}    dz^{\epsilon,k}_t&=\lambda_k\epsilon^{-1}z_t^{\epsilon,k}dt+dW^{k,\H}_t,\\
    dv^{\epsilon,\delta,k}_t&=\lambda_k\delta^{-1}v^{\epsilon,\delta,k}_tdt+\lambda_kz^{\epsilon,k}_t(\delta^{-1}-\epsilon^{-1})dt,
\end{align*}
we have that
\begin{align}\label{i:HolderSmall1st}
    \EE\sx[|W^{\epsilon,k,\H}_t-W^{\delta,k,\H}_t|^2\dx]&=\EE\sx[(v_t^{\epsilon,\delta,k})^2\dx]\notag\\
    &=\EE\sx[\sx(\int_0^te^{\lambda_k\delta^{-1}(t-r)}\lambda_kz_r^{\epsilon,k}(\delta^{-1}-\epsilon^{-1})dr\dx)^2\dx]\notag\\
    &\lesssim(\delta-\epsilon)^2\epsilon^{-4}\EE\sx[\sup_{t\in[0,T]}(z_t^{\epsilon,k})^2\dx]\delta^2\notag\\
    & \lesssim (\delta-\epsilon)^2\epsilon^{-4}.
\end{align}
Then, estimates \eqref{i:Holder1st}, \eqref{i:HolderSmall1st} imply that
\begin{align}\label{stima3}
   \EE\Big[\|\mathbb{Y}^{\epsilon,1}&-\mathbb{Y}^{\delta,1}\|_{p\var,[0,T]}^q\Big]\notag\\
   &\lesssim \overline{\sigma}^{\frac q2-1}\sum_k \frac{\sigma_k^2\|e_k\|_{H^{\xi}}^2} {\lambda_k^2}\sx(|\epsilon-\delta|^2\epsilon^{-4}\dx)^{\frac q2\sx(1-\frac{1}{2\H\gamma}\dx)}\notag\\
   &\lesssim  \overline{\sigma}^{\frac q2}\sx(|\epsilon-\delta|^2\epsilon^{-4}\dx)^{\frac q2\sx(1-\frac{1}{2\H\gamma}\dx)}.
\end{align}

By a similar argument, we can obtain the following estimate on the second level of the rough path
\begin{align} \label{stima4}
        \EE\Big[\|\mathbb{Y}^{\epsilon,2}&-\mathbb{Y}^{\delta,2}\|_{p/2\var,[0,T]}^{q}\Big]\notag\\
        &\lesssim 
    \sum_{k,j} \frac{|\sigma_k\sigma_j\|e_k\|_{H^{\xi}}\|e_j\|_{H^{\xi}}|^{p/2}}{|\lambda_k\lambda_j|^{p/2}} \EE\sx[\|\mathbb{W}^{\epsilon,k,j,\H}-\mathbb{W}^{\delta,k,j,\H}\|_{p/2\var}^q\dx]\notag\\
    &\lesssim 
        \sum_{k,j} \frac{|\sigma_k\sigma_j\|e_k\|_{H^{\xi}}\|e_j\|_{H^{\xi}}|^{p/2}}{|\lambda_k\lambda_j|^{p/2}}\notag\\
        &\hspace{10pt}\times\sup_{t\in [0,T]}\EE\sx[|(W^{\epsilon,k,\H}_t,W^{\epsilon,j,\H}_t)-(W^{\delta,k,\H}_t,W^{\delta,j,\H}_t)|^2\dx]^{\frac q2\sx(1-\frac{1}{2\H\gamma}\dx)}\notag
        \\
        &\lesssim (\epsilon^{-4}|\epsilon-\delta|^2)^{\frac q2\sx(1-\frac{1}{2\H\gamma}\dx)}
\end{align}
At last, let $\Y^0=\W^\H$. Then for $0\le \epsilon\le\delta\le1$, we have that 
\begin{enumerate}
    \item if $2\epsilon<\delta$, relations \eqref{stima1}, \eqref{stima2} imply that
    \begin{align*}
    \sum_{i=1}^2\EE\sx[\|\mathbb{Y}^{\epsilon,i}-\mathbb{Y}^{\delta,i}\|_{p/i\var,[0,T]}^{q}\dx]^{1/q}&\lesssim\epsilon^{\H \sx(1-\frac{1}{2\H\gamma}\dx)}+\delta^{\H \sx(1-\frac{1}{2\H\gamma}\dx)}\\
    &\lesssim |\epsilon-\delta|^{\H \sx(1-\frac{1}{2\H\gamma}\dx)};    
    \end{align*}
    \item if $\delta \in (\epsilon,2\epsilon)$ and $\epsilon\le(\delta-\epsilon)^{\nu}$, then relations \eqref{stima1}, \eqref{stima2} imply that
    \begin{align*}
    \sum_{i=1}^2\EE\sx[\|\mathbb{Y}^{\epsilon,i}-\mathbb{Y}^{\delta,i}\|_{p/i\var,[0,T]}^{q}\dx]^{1/q}&\lesssim\epsilon^{\H \sx(1-\frac{1}{2\H\gamma}\dx)}+\delta^{\H \sx(1-\frac{1}{2\H\gamma}\dx)}\\
    &\lesssim |\epsilon-\delta|^{\nu \H \sx(1-\frac{1}{2\H\gamma}\dx)};    
    \end{align*}
    \item if $\delta \in (\epsilon,2\epsilon)$ and $\epsilon>(\delta-\epsilon)^{\nu}$, then
    inequalities \eqref{stima3}, \eqref{stima4} imply that
    \begin{align*}
    \sum_{i=1}^2\EE\sx[\|\mathbb{Y}^{\epsilon,i}-\mathbb{Y}^{\delta,i}\|_{p/i\var,[0,T]}^{q}\dx]^{1/q}&\lesssim |\epsilon-\delta|^{\sx(1-\frac{1}{2\H\gamma}\dx)}\epsilon^{-2}\\
            &\lesssim|\epsilon-\delta|^{\sx(1-\frac{1}{2\H\gamma}\dx)-2\nu}.    
    \end{align*}        
\end{enumerate}
Then, choosing $\nu$ sufficiently small allows us to conclude using the Kolmogorov theorem.
\end{proof}
\begin{corollary}\label{cor:beta12}
     Let $1/3<\H<1/2$, $\alpha = 1/2+\H$, $q\in [4,\infty)$, $\gamma>\frac{1}{2\H}$ s.t. $2\H + 1/\gamma > 1$, and $3>p>2\gamma$ is such that \eqref{eq:p_summability} holds. Then, as $\epsilon \to 0$, the canonical rough path lift of $(y_t^\epsilon)_t$, denoted by $\mathbf{Y}^\epsilon$, converges in $L^q(\Omega)$ to the canonical $p$-rough path on $H^\xi$ associated with $y^0_t\equiv0.$
    Moreover, the following map 
    \begin{equation*}
        \epsilon\in(0,1) \mapsto \mathbf{Y}^{\epsilon}\in L^q\left(\Omega,C^{p\var}(H^\xi)\times C^{p/2\var}_2((H^\xi)^{\otimes 2})\right)
    \end{equation*}
    is H\"older-continuous and it holds that
    \begin{equation}\label{est:UniformRegularityURDH<1/2zeros}
        \mathbb{E}\left[\sup_{\epsilon\in (0,1)}\sum_{n=1}^2\|\mathbb{Y}^{\epsilon,n}\|_{p/n\var,[0,T]}^q\right]\lesssim 1.
        \end{equation}
\end{corollary}
\begin{proof}
    Let 
    $\overline{y}^\epsilon_t=\int_0^t\epsilon^{\H-1}w^{\epsilon}_sds.$
    Then for $\alpha=1/2+\H$, we have that
    \begin{equation*}
y_t^\epsilon=\int_0^t\epsilon^{-1/2}w^{\epsilon}_sds=\epsilon^{1/2-\H}\overline{y}^\epsilon_t.
    \end{equation*}
    Therefore, \autoref{prop_rough_pathless12} implies the claim.
\end{proof}

\subsection{$\hat{\mathbf{Y}}^\epsilon_t$ convergence}
Here we focus on $\hat{\mathbf{Y}}^\epsilon$, the canonical rough path associated with $\hat{y}^\epsilon_t.$
\begin{proposition}\label{second_prop_rough_pathless12}
    Let $1/3<\H<1/2$, $q\in [4,\infty)$,  $\alpha=1,\ \gamma>\frac{1}{2\H}$ s.t. $2\H + 1/\gamma > 1,$ $3>p>2\gamma$ is such that \eqref{eq:p_summability} holds, and $\beta \in (1-1/p,2\H)$. If \eqref{eq:centering_cond} holds, then as $\epsilon \to 0$, $\hat{\mathbf{Y}}^\epsilon$ converges in $L^q(\Omega)$ to a $p$-rough path on $H^\xi$, denoted by $\mathbf{W}^\H=(\mathbb{W}^{\H,1},\mathbb{W}^{\H,2})$, where $\mathbb{W}^1$ and $\mathbb{W}^2$ are given by \eqref{eq:W^1} and \eqref{eq:W^2}.
     Moreover, the following map 
    \begin{equation*}
        \epsilon\in(0,1) \mapsto \hat{\mathbf{Y}}^{\epsilon}\in L^q\left(\Omega,C^{p\var}(H^\xi)\times C^{p/2\var}_2((H^\xi)^{\otimes 2})\right)
    \end{equation*}
    is Hölder-continuous and it holds that
    \begin{equation}\label{est:UniformRegularity2<1/2}
        \mathbb{E}\left[\sup_{\epsilon\in (0,1)}\sum_{n=1}^2\|\hat{\mathbb{Y}}^{\epsilon,n}\|_{p/n\var,[0,T]}^q\right]\lesssim 1.
    \end{equation}
\end{proposition}
\begin{proof}
    It suffices to establish the claims for arbitrarily large $q$,
since the conclusions at lower moment orders follow from
Hölder's inequality. We first fix $\theta\in(\beta,2\H)$ and take $q$
sufficiently large that $\beta<\theta-1/q$. If necessary, $q$ is further increased below for the final application of Kolmogorov's continuity theorem. Notice that 
    \begin{align*}
\hat{y}^\epsilon_t=y^{\epsilon}_t+h^\epsilon_t,
    \end{align*}
    where $y^\epsilon_t$'s rough path has been treated in \autoref{prop_rough_pathless12}, while $h^\epsilon_t$'s H\"older regularity is estimated at \autoref{cor:second_corrector}.\\
\textit{First Step: Convergence.}

    Since the $p$-variation is a semi-norm, we have that
    \begin{equation*}\label{eq:yHatStima1}
    \begin{aligned}
    \EE\sx[\|\hat{\mathbb{Y}}^{\epsilon,1}-\mathbb{W}^{\H,1}\|_{p\var,[0,T]}^q\dx]&\lesssim \EE\sx[\|\mathbb{Y}^{\epsilon,1}-\mathbb{W}^{\H,1}\|_{p\var,[0,T]}^q\dx]+\EE\sx[\|h^\epsilon\|_{p\var,[0,T]}^q\dx]\\&\lesssim    
    \epsilon^{\H q\sx(1-\frac{1}{2\H\gamma}\dx)}\overline{\sigma}^{\frac{q}{2}}+\epsilon^{ q(2\H-\theta)},
    \end{aligned}
\end{equation*}
where the last inequality follows from \eqref{stima1}, \eqref{eq:hHolder} and the choice of $\theta,q $.

We now turn to the second level. In particular,
\begin{equation}
  \begin{aligned}  \hat{\mathbb{Y}}^{\epsilon,2}\st&=\mathbb{Y}^{\epsilon,2}\st+\int_s^th^\epsilon_{s,r}\otimes dy^\epsilon_r+\int_s^ty^\epsilon_{s,r}\otimes dh^\epsilon_r+\int_s^th^\epsilon_{s,r}\otimes dh^\epsilon_r\\
  &=:\mathbb{Y}^{\epsilon,2}\st+\mathbb{H}^{\epsilon}\st
  \end{aligned}
\end{equation}
By \eqref{stima2}, it suffices to concentrate on $\mathbb{H}^\epsilon$. More precisely, we bound the $p/2$-variation of $\mathbb{H}\st^{\epsilon,1}:=\int_s^t h_{s,r}^\epsilon\otimes dy_r^\epsilon$. The corresponding estimate for $\mathbb{H}^\epsilon$ then follows immediately.
\begin{equation}
    \begin{aligned}
        \EE&\left[\Big(\sup_{P\in \Pi}\sum_{[s,t]\in P}\|\mathbb{H}^{\epsilon,1}\st\|_{(H^\xi)^{\otimes 2}}^{\frac p2}\Big)^{\frac{2 q}{p}}\right]\lesssim\\ &\lesssim\EE\left[\left(\sup_{P\in \Pi}\sum_{[s,t]\in P}\|y^\epsilon\|_{p\var,[s,t]}^{\frac p2}|t-s|^{\beta\frac p2}
        \right)^{\frac{2q}{p}}\|h^\epsilon\|_{C^\beta([0,T],H^\xi)}^q\right]\\
        &\lesssim\EE\left[\left(\|y^\epsilon\|_{p\var,[0,T]}^p+1\right)^{\frac{2q}{p}}\|h^\epsilon\|^q_{C^\beta([0,T],H^\xi)}\right]\lesssim\epsilon^{q(2\H-\theta)},
    \end{aligned}
\end{equation}
where the first inequality is due to $\beta+\frac1p>1$ and the Sewing lemma, see \cite[Lemma 2.2]{deya2019priori}, the second one is due to $\beta p>1$ and the last one follows from \eqref{eq:hHolder}.\\
\textit{Second Step: Proof of \eqref{est:UniformRegularity2<1/2}.}
Let $0<\epsilon< \delta\le1 $. We need to estimate $h^\epsilon-h^\delta$'s time regularity. Since
\begin{align*}
 h^\epsilon_t-h^\delta_t&=\int_0^t\epsilon^{2\H-1}z^\epsilon_r-\delta^{2\H-1}z^\delta_rdr\\
&=\int_0^t(\epsilon^{2\H-1}-\delta^{2\H-1})z^\epsilon_r+\delta^{2\H-1}(z^\epsilon_r-z^\delta_r)dr,   
\end{align*}
we start working on $\mu^{\epsilon,\delta}_t:=z^\epsilon_t-z^\delta_t$. It holds that
\begin{equation}\label{est:mu1}
    \begin{aligned}
        \mu^{\epsilon,\delta}_t&=\int_0^t\epsilon^{-1}e^{\epsilon^{-1}Cr}
    b(w^\epsilon_{t-r},w^\epsilon_{t-r})
    -\delta^{-1}e^{\delta^{-1}Cr}
    b(w^\delta_{t-r},w^\delta_{t-r})dr\\
        &=\int_0^\infty e^{Cr}b(w^\epsilon_{(t-\epsilon r)_+},
      w^\epsilon_{(t-\epsilon r)_+})
    -e^{Cr}b(w^\delta_{(t-\delta r)_+},
       w^\delta_{(t-\delta r)_+})dr.
    \end{aligned}
\end{equation}
Therefore,
\begin{equation*}
  \begin{aligned}  \EE\sx[\|\mu^{\epsilon,\delta}_t\|^q_{H^{\xi}}\dx] &\lesssim \left(\int_0^\infty e^{-\mathfrak{h}r}\EE\sx[\| w^{\epsilon}_{(t-\epsilon r)_+}-w^{\delta}_{(t-\delta r)_+}\|_{H^{\xi+1}}^{2q}\dx]^{1/(2q)}dr\right)^q\\ &\lesssim \int_0^\infty e^{-\mathfrak{h}r}\EE\sx[\| w^{\epsilon}_{(t-\epsilon r)_+}-w^{\delta}_{(t-\delta r)_+}\|_{H^{\xi+1}}^{2q}\dx]^{1/2}dr\\
    &\lesssim \int_0^\infty e^{-\mathfrak{h}r}\Big(\EE\sx[\| w^{\epsilon}_{(t-\epsilon r)_+}-w^{\epsilon}_{(t-\delta r)_+}\|_{H^{\xi+1}}^{2q}\dx]^{1/2}\\
    &\qquad \qquad \quad  +\EE\sx[\| w^{\epsilon}_{(t-\delta r)_+}-w^{\delta}_{(t-\delta r)_+}\|_{H^{\xi+1}}^{2q}\dx]^{1/2}\Big)dr=:I_1+I_2
\end{aligned}
\end{equation*}
where the first inequality is due to \autoref{hp_stoch_conv} and Minkowski's, while the second follows by Jensen.

Since $w^\epsilon_t=\sum_kX^{k,\epsilon}_te_k$, where $X^{k,\epsilon}_t:=\langle w^\epsilon_t,e_k\rangle$, $\EE\sx[X^{k,\epsilon}_tX^{k',\epsilon}_t\dx]=0$ for all $k\neq k'$ and the $X^{k,\epsilon}$ are Gaussian processes; it will be enough assuming $q=1$ and working on $X^k$. Indeed,
\begin{equation}
    \begin{aligned}
        X^{k,\epsilon}_t-X^{k,\delta}_t=\frac{\delta^\H-\epsilon^\H}{\delta^\H}X^{k,\epsilon}_t+\lambda_k(\epsilon^{-1}-\delta^{-1})\int_0^te^{\lambda_k\epsilon^{-1}(t-s)}X^{k,\delta}_sds.
    \end{aligned}
\end{equation}
Therefore,  \eqref{rough_exponential_convergence_mean} implies that
\begin{equation}\label{I_1}
    \EE\sx[|X^{k,\epsilon}_t-X^{k,\delta}_t|^2\dx]\lesssim |\epsilon-\delta|^2\delta^{-2}\sigma_k^2. 
\end{equation}

Analogously, we study $I_1$ by means of $X^{k,\epsilon}_t-X^{k,\epsilon}_s$:
\begin{equation}\label{I_2}
    \begin{aligned}
        \EE\sx[|X^{k,\epsilon}_t-X^{k,\epsilon}_s|^2\dx]&\lesssim \EE\sx[|X^{k,\epsilon}_s|^2\dx](1-e^{\lambda_k(t-s)\epsilon^{-1}})^2\\
        &+\sigma_k^2\epsilon^{-2\H}\EE\left(\int_s^t e^{\lambda_k\epsilon^{-1}(t-r)}dW^{\H,k}_r\right)^2\\
        &\lesssim \sigma_k^2|t-s|^{2\H}\epsilon^{-2\H},
    \end{aligned}
\end{equation}
where the last inequality is due to standard estimates and integration by parts. So, \eqref{est:mu1}, \eqref{I_1} and \eqref{I_2} imply that

\begin{equation}\label{est:mu}
    \begin{aligned}
\EE\sx[\|\mu^{\epsilon,\delta}_t\|^q_{H^{\xi}}\dx]&\lesssim\int_0^\infty e^{-\mathfrak{h}r}(|\epsilon-\delta|^q \delta^{-q}+|(t-\epsilon r)_+-(t-\delta r)_+|^{q\H}\epsilon^{-q\H})dr\\
&\lesssim \int_0^\infty e^{-\mathfrak{h}r}(|\epsilon-\delta|^q \delta^{-q}+|\epsilon-\delta|^{q\H}\epsilon^{-q\H} r^{q\H})dr\lesssim|\epsilon-\delta|^{q\H} \epsilon^{-q}.
    \end{aligned}
\end{equation}
Therefore,
\begin{equation}\label{h1var}
    \begin{aligned}
    \EE\Big[\|h^\epsilon-h^\delta\|_{1\var, H^\xi}^q\Big]&\lesssim  \int_0^T\EE\left[\| z_r^\epsilon\|^q_{H^\xi}\right]|\epsilon-\delta|^q\epsilon^{-2q}dr\\&+\int_0^T\delta^{-q}\EE[\|\mu^{\epsilon,\delta}_r\|^q_{H^\xi}]dr\lesssim |\epsilon-\delta|^{q\H}\epsilon^{-2q},
    \end{aligned}
\end{equation}
where the last inequality is due to \autoref{lemma_second_corrector} and \eqref{est:mu}.

Now, we switch to the second order. Since $\mathbb{Y}^\epsilon-\mathbb{Y}^\delta$'s $p/2$-variation is estimated at \autoref{prop_rough_pathless12}, we can focus on $\mathbb{H}^\epsilon-\mathbb{H}^\delta.$ As in the previous step, we can focus on $\mathbb{H}^{1,\epsilon}-\mathbb{H}^{1,\delta}$. Indeed, Sewing lemma implies that
\begin{equation}
\begin{aligned}
    \EE\left[\|\mathbb{H}^{1,\epsilon}-\mathbb{H}^{1,\delta}\|_{p/2\var,(H^\xi)^{\otimes2}}^q\right]&\lesssim\EE\Big[\|h^\epsilon-h^\delta\|_{1\var}^q\|y^\epsilon\|_{p\var}^q+\|h^\delta\|_{C^\beta}^q\|y^\epsilon-y^\delta\|_{p\var}^q\Big]\\
    &\lesssim |\epsilon-\delta|^{q\H}\epsilon^{-2q}+ (|\epsilon-\delta|^2\epsilon^{-4})^{\frac q2\sx(1-\frac{1}{2\H\gamma}\dx)}
\end{aligned}
\end{equation}
where the last inequality is due to \eqref{stima3} and \eqref{h1var}.

As in \autoref{prop_rough_pathless12} and \autoref{proproughHgreat12}, the statement follows from Kolmogorov continuity theorem with respect to $\epsilon$.
\end{proof}
\section{A Priori Estimates}\label{sec_apriori}
The goal of this section is to prove the following two lemmas.
\begin{lemma}\label{lemma_a_priori}
Under the same assumptions of \autoref{main_thm}, let $\frac1\H\in[N,N+1)\cap(1,3)$ and $p\in (\frac{1}{\H},N+1)$. There exists $\kappa>0$ sufficiently small such that the following inequalities hold 
\begin{align}
    \label{EnergyEst}\sup_{t\in [0,T]}\norm{u^{\epsilon}_t}^2+2\nu\int_0^T \norm{\nabla u^{\epsilon}_s}^2 ds & \lesssim 1\quad \mathbb{P}-a.s.,\\
\label{EstOnR}\mathbb{E}\left[\log\left(1+\int_0^T \norm{r^{\epsilon}_s}_{H^{\gamma}}^2ds\right)\right]&\lesssim 1\\
\label{EstOnPvar}\mathbb{E}\left[\log\left(1+\norm{u^{\epsilon}}^\kappa_{p-var,[0,T],H^{-\xi+1}}\right)\right]&\lesssim 1    
\end{align}
uniformly in $\epsilon\in (0,1)$.
    
\end{lemma}
\begin{lemma}\label{lemma_a_priori_centered}
Under the same assumptions of \autoref{centered_noise}, let $p\in (\frac{1}{\H},3)$. There exists $\kappa>0$ sufficiently small such that \eqref{EnergyEst} and the following inequalities hold 
\begin{align}
\label{EstOnR_cent}\mathbb{E}\left[\log\left(1+\int_0^T \norm{R^{\epsilon}_s}_{H^{\gamma}}^2ds\right)\right]&\lesssim 1\\
\label{EstOnPvar_cent}\mathbb{E}\left[\log\left(1+\norm{u^{\epsilon}}^\kappa_{p-var,[0,T],H^{-\xi+1}}\right)\right]&\lesssim 1    
\end{align}
uniformly in $\epsilon\in (0,1)$. Moreover $R^{\epsilon}\rightarrow 0$ in probability in $L^1(0,T;H)$ as $\epsilon\rightarrow 0,$ more specifically
\begin{align}\label{eq:vanishing_R}
    \mathbb{E}\left[\log\left(1+\int_0^T \norm{R^{\epsilon}_s} ds\right)\right]&\lesssim \epsilon^{3\H-1}.
\end{align}
    
\end{lemma}

We split this section in two parts. In the first one we do compactness in space, therefore we focus on \eqref{EnergyEst}, \eqref{EstOnR}, \eqref{EstOnR_cent} and \eqref{eq:vanishing_R}. Instead, in the second part we obtain compactness in time, namely we prove \eqref{EstOnPvar} and \eqref{EstOnPvar_cent}. While the first part does not use the rough path formulation \autoref{sec_rough_formulation}, this will be crucial in the second part of the argument. In order to ease the notation we set $\nu=1$. 

\subsection{Space Compactness}\label{sec_space_compactness}
For $\epsilon>0$ the integrals appearing in equation \eqref{e:equation1} are standard Lebesgue integral and moreover the solutions have been obtained by Galerkin approximation. Therefore it follows immediately that
\begin{align}\label{energy_estimate}
    \sup_{t\in [0,T]}\norm{u^{\epsilon}_t}^2+2\int_0^T \norm{\nabla u^{\epsilon}_s}^2 ds\leq 2\norm{u^{\epsilon}_0}^2.
\end{align}
Now we want to estimate $r^{\epsilon}$ or $R^\epsilon$. This analysis depends on the different cases under consideration.
\subsubsection*{Case $\H<\frac{1}{2},\ \alpha=\frac{1}{2}+\H$.} Our solutions are obtained by Galerkin approximation therefore, by \eqref{energy_estimate}, H\"older's and Young's  inequalities they satisfy the following relation
\begin{align*}
    \epsilon \frac{d\norm{r^{\epsilon}}^2}{dt}+2\epsilon \norm{\nabla r^{\epsilon}}^2+2\norm{(-C)^{1/2}r^{\epsilon}}^2& \leq 2\epsilon^{1/2}|\langle\nabla r^\epsilon,\nabla w^{\epsilon}\rangle|\\ &+2\epsilon^{1/2}\langle b(u^\epsilon+\epsilon^{-1/2}w^{\epsilon}+r^{\epsilon},w^{\epsilon}),r^{\epsilon}\rangle\\ & \leq \epsilon \norm{\nabla r^{\epsilon}}^2+\norm{\nabla w^{\epsilon}}^2+\delta \norm{r^{\epsilon}}^2\\
    &+C_{\delta}\epsilon \norm{u^{\epsilon}_0}^2\norm{\nabla w^{\epsilon}}_{L^{\infty}}^2 +C_{\delta}\norm{ w^\epsilon}_{L^4}^2\norm{\nabla w^\epsilon}_{L^4}^2\\
    &+\epsilon C_{\delta}\norm{r^{\epsilon}}^2\norm{\nabla w^{\epsilon}}_{L^{\infty}}^2.
\end{align*}
   Hence, by selecting $\delta$ sufficiently small and using the coercivity of $C$, we deduce that
    \begin{align}\label{inequality_gronwall_1}
    \epsilon \frac{d\norm{r^{\epsilon}}^2}{dt}+\norm{r^{\epsilon}}_{H^{\gamma}}^2 & \lesssim \norm{w^{\epsilon}}_{H^{\xi}}^2+\norm{w^{\epsilon}}_{L^4}^2\norm{\nabla w^{\epsilon}}_{L^4}^2+\epsilon \norm{r^{\epsilon}}^2\norm{w^{\epsilon}}_{H^{\xi}}^2.     
    \end{align}
    Therefore Gr\"onwall's lemma implies 
    \begin{align*}
        \epsilon \sup_{t\in [0,T]}\norm{r^{\epsilon}}^2& \leq \left(\epsilon \norm{r_0^{\epsilon}}^2+C_1\int_0^T \norm{w^{\epsilon}_s}_{H^{\xi}}^2+\norm{w^{\epsilon}_s}_{L^{4}}^2\norm{\nabla w^{\epsilon}_s}_{L^{4}}^2 ds\right)e^{C_2\int_0^T \norm{w^{\epsilon}_s}_{H^{\xi}}^2 ds}.
    \end{align*}
    Plugging this information in \eqref{inequality_gronwall_1} we also obtain
    \begin{align*}
        \int_0^T \norm{r^{\epsilon}_s}_{H^{\gamma}}^2ds& \leq \left(\epsilon \norm{r_0^{\epsilon}}^2+C_1\int_0^T \norm{w^{\epsilon}_s}_{H^{\xi}}^2+\norm{w^{\epsilon}_s}_{L^{4}}^2\norm{\nabla w^{\epsilon}_s}_{L^{4}}^2 ds\right)\\
        &\times\left(1+e^{C_2 \int_0^T \norm{w^{\epsilon}_s}_{H^{\xi}}^2 ds}\int_0^T\norm{w^{\epsilon}_s}_{H^{\xi}}^2 ds \right).
    \end{align*}
    In particular, it holds
    \begin{align*}
        \mathbb{E}\left[\log\left(1+\int_0^T \norm{r^{\epsilon}_s}_{H^{\gamma}}^2ds\right)\right]& \leq \mathbb{E}\left[\log\left(1+\epsilon \norm{r_0^{\epsilon}}^2+C_1\int_0^T \norm{w^{\epsilon}_s}_{H^{\xi}}^2+\norm{w^{\epsilon}_s}_{L^{4}}^2\norm{\nabla w^{\epsilon}_s}_{L^{4}}^2 ds\right)\right]\\ & +\mathbb{E}\left[\log\left(1+\frac{e^{C_2 \int_0^T \norm{w^{\epsilon}_s}_{H^{\xi}}^2 ds}\int_0^T\norm{w^{\epsilon}_s}_{H^{\xi}}^2 ds}{2}\right)\right]+\log2\\ & \lesssim 1+\epsilon \norm{r_0^{\epsilon}}^2+\sup_{t\in [0,T]}\mathbb{E}\left[\norm{w^{\epsilon}_t}_{H^{\xi}}^4\right]\\
        & +\mathbb{E}\left[\log\left(1+\frac{e^{C_2 \int_0^T \norm{w^{\epsilon}_s}_{H^{\xi}}^2 ds}}{2}\right)\one_{\left\{e^{C_2 \int_0^T \norm{w^{\epsilon}_s}_{H^{\xi}}^2 ds}\leq 1\right\}}\right]\\ & +\mathbb{E}\left[\log\left(1+\frac{e^{C_2 \int_0^T \norm{w^{\epsilon}_s}_{H^{\xi}}^2 ds}}{2}\right)\one_{\left\{e^{C_2 \int_0^T \norm{w^{\epsilon}_s}_{H^{\xi}}^2 ds}> 1\right\}}\right]\\ & \lesssim 1+\epsilon \norm{r_0^{\epsilon}}^2+\sup_{t\in [0,T]}\mathbb{E}\left[\norm{w^{\epsilon}_t}_{H^{\xi}}^4\right]
    \end{align*}
which is bounded, uniformly in $\epsilon$ by \autoref{hp_stoch_conv}.
In conclusion we showed
\begin{align}\label{compactness_reminder}
   \sup_{\epsilon\in (0,1)}\mathbb{E}\left[\log\left(1+\int_0^T \norm{r^{\epsilon}_s}_{H^{\gamma}}^2ds\right)\right]&\lesssim 1. 
\end{align}
\subsubsection*{Case $\H>\frac{1}{2}$}
Our solutions are obtained by Galerkin approximation therefore, by \eqref{energy_estimate}, H\"older's and Young's  inequalities they satisfy the following relation
\begin{align*}
    \epsilon \frac{d\norm{r^{\epsilon}}^2}{dt}+2\epsilon \norm{\nabla r^{\epsilon}}^2+2\norm{(-C)^{1/2}r^{\epsilon}}^2& \leq -2\epsilon^{\H}\langle \nabla w^{\epsilon},\nabla r^{\epsilon}\rangle\\
    &+2\epsilon^{\H}\langle b(u^\epsilon+\epsilon^{-1+\H}w^{\epsilon}+r^{\epsilon},w^{\epsilon}),r^{\epsilon}\rangle\\ & \leq \epsilon \norm{\nabla r^{\epsilon}}^2+\delta \norm{r^{\epsilon}}^2+C_{\delta}\epsilon^{2\H} \norm{u^{\epsilon}_0}^2\norm{\nabla w^{\epsilon}}_{L^{\infty}}^2 \\
    &+\epsilon^{2\H-1}\norm{\nabla w^{\epsilon}}^2+C_{\delta}\epsilon^{4\H-2}\norm{w^\epsilon}_{L^4}^2\norm{\nabla w^{\epsilon}}_{L^{\infty}}^2\\
    &+\epsilon C_{\delta}\norm{r^{\epsilon}}^2\norm{\nabla w^{\epsilon}}_{L^{\infty}}^2.
\end{align*}
    Therefore, up to choosing $\delta$ sufficiently small, due to the coercivity of $C$ we obtain
    \begin{align*}
    \epsilon \frac{d\norm{r^{\epsilon}}^2}{dt}+\norm{r^{\epsilon}}_{H^{\gamma}}^2 & \lesssim \norm{w^{\epsilon}}_{H^{\xi}}^2+\norm{w^{\epsilon}}_{H^{\xi}}^4+\norm{w^{\epsilon}}_{L^4}^4+\epsilon \norm{r^{\epsilon}}^2\norm{w^{\epsilon}}_{H^{\xi}}^2.     
    \end{align*}
    Now the argument follows verbatim as in the previous case and we omit details.
\subsubsection*{Case $\H<\frac{1}{2},\ \alpha=1$ and \eqref{eq:centering_cond}}
Once more, our solutions are obtained by Galerkin approximation therefore, by \eqref{energy_estimate}, H\"older's and Young's  inequalities they satisfy the following relation
\begin{align*}
    \epsilon\frac{d\norm{R^\epsilon}^2}{dt}+2\epsilon\norm{\nabla R^\epsilon}^2+2\norm{(-C)^{1/2} R^\epsilon}^2 
        &\leq 
          2 \langle\epsilon^{\H}\Delta w^\epsilon
                     +\epsilon^{2\H}\Delta z^\epsilon, R^\epsilon\rangle \\
        &\quad+2\epsilon^{\H}\langle b(u^\epsilon+R^\epsilon,w^\epsilon),R^\epsilon\rangle\\ & \quad
              +2\epsilon^{2\H}\langle b(u^\epsilon+R^\epsilon,z^\epsilon), R^{\epsilon}\rangle\\
        &\quad+\epsilon^{3\H-1}
              2\langle b(w^\epsilon,z^\epsilon)
                   +b(z^\epsilon,w^\epsilon),R^{\epsilon}\rangle\\ & \quad 
              +2\epsilon^{4\H-1}\langle b(z^\epsilon,z^\epsilon),R^{\epsilon}\rangle\\ & \leq
              (\delta+\epsilon^\H\|w^\epsilon\|_{H^\xi}+\epsilon^{2\H}\|z^\epsilon\|_{H^\xi})\|R^\epsilon\|^2\\ & \quad +C_\delta (\epsilon^{4\H}\|z^\epsilon\|_{H^\xi}^2+\epsilon^{2\H}\|w^\epsilon\|_{H^\xi}^2)\\ &\quad  +C_\delta\epsilon^{6\H-2}\|z^\epsilon\|_{H^\xi}^2\left(\|w^\epsilon\|_{H^\xi}^2+\epsilon^{2\H} \|z^\epsilon\|_{H^\xi}^2\right).
\end{align*}
Consequently, the coercivity of $C$ and the elementary bound $\epsilon^\H\|w^\epsilon\|_{H^\xi}+\epsilon^{2\H}\|z^\epsilon\|_{H^\xi}\leq \delta+C_\delta (\epsilon^{3\H} \|w^\epsilon\|_{H^\xi}^3+\epsilon^{6\H} \|z^\epsilon\|_{H^\xi}^3)$ imply, for $\delta$ sufficiently small,
    \begin{align*}
    \epsilon \frac{d\norm{R^{\epsilon}}^2}{dt}+\norm{R^{\epsilon}}_{H^{\gamma}}^2 & \lesssim \epsilon \norm{R^{\epsilon}}^2\left(\epsilon^{3\H-1}\norm{w^{\epsilon}}_{H^{\xi}}^{3}+\epsilon^{6\H-1}\norm{z^{\epsilon}}_{H^{\xi}}^{3}\right)\\ & +\epsilon^{4\H}\|z^\epsilon\|_{H^\xi}^2+\epsilon^{2\H}\|w^\epsilon\|_{H^\xi}^2+\epsilon^{6\H-2}\|z^\epsilon\|_{H^\xi}^2\left(\|w^\epsilon\|_{H^\xi}^2+\epsilon^{2\H} \|z^\epsilon\|_{H^\xi}^2\right).     
    \end{align*}
    Now, arguing as in the previous cases, we obtain \eqref{EstOnR_cent} by \autoref{hp_stoch_conv} and \eqref{eq:a_priori_z}. Note that here we crucially used $\H>1/3.$ The inequality above also implies by Gr\"onwall's inequality
    \begin{align*}
        \|R^\epsilon_t\|^2 &\lesssim e^{-t/\epsilon+C_1\epsilon^{3\H-1}\int_0^T \norm{w^{\epsilon}_s}_{H^{\xi}}^{3}+\epsilon^{3\H}\norm{z^{\epsilon}_s}_{H^{\xi}}^{3}ds}\|R^\epsilon_0\|^2\\ & +\epsilon^{2\H-1} e^{C_1\epsilon^{3\H-1}\int_0^T \norm{w^{\epsilon}_s}_{H^{\xi}}^{3}+\epsilon^{3\H}\norm{z^{\epsilon}_s}_{H^{\xi}}^{3}ds} \int_0^t e^{-(t-s)/\epsilon}\left(\epsilon^{2\H}\|z^\epsilon_s\|_{H^\xi}^2+\|w^\epsilon_s\|_{H^\xi}^2\right) ds\\ & +\epsilon^{6\H-3}e^{C_1\epsilon^{3\H-1}\int_0^T \norm{w^{\epsilon}_s}_{H^{\xi}}^{3}+\epsilon^{3\H}\norm{z^{\epsilon}_s}_{H^{\xi}}^{3}ds} \int_0^t e^{-(t-s)/\epsilon}\|z^\epsilon_s\|_{H^\xi}^2\left(\|w^\epsilon_s\|_{H^\xi}^2+\epsilon^{2\H} \|z^\epsilon_s\|_{H^\xi}^2\right) ds
    \end{align*}
    Taking the square root in the inequality above and integrating in time we obtain
    \begin{align*}
        \int_0^T \|R^\epsilon_t\|dt&\lesssim  e^{C_1\epsilon^{3\H-1}\int_0^T \norm{w^{\epsilon}_s}_{H^{\xi}}^{3}+\epsilon^{3\H}\norm{z^{\epsilon}_s}_{H^{\xi}}^{3}ds}\left(\epsilon\|R^\epsilon_0\|+\int_0^T \left(\frac{1}{\epsilon}\int_0^t e^{-(t-s)/\epsilon}F^\epsilon_s ds\right)^{1/2} dt\right),
    \end{align*}
where we denoted by
\begin{align*}
F^\epsilon_t=\epsilon^{4\H}\|z^\epsilon_t\|_{H^\xi}^2+\epsilon^{2\H}\|w^\epsilon_t\|_{H^\xi}^2+\epsilon^{6\H-2}\|z^\epsilon_t\|_{H^\xi}^2\left(\|w^\epsilon_t\|_{H^\xi}^2+\epsilon^{2\H} \|z^\epsilon_t\|_{H^\xi}^2\right).
\end{align*}
Moreover, by H\"older's inequality, Fubini's theorem, \autoref{hp_stoch_conv} and \eqref{eq:a_priori_z}
\begin{align*}
   \left( \mathbb{E}\left[\int_0^T \left(\int_0^t e^{-(t-s)/\epsilon}F^\epsilon_s ds\right)^{1/2} dt\right]\right)^2&\lesssim \int_0^T \int_0^t e^{-(t-s)/\epsilon}\sup_{s\in [0,T]}\mathbb{E}\left[F^\epsilon_s\right] ds dt\\ & \lesssim \epsilon^{6\H-1}.
\end{align*}
Consequently, the elementary estimate $\log(1+e^a b)\leq a+b$, \autoref{hp_stoch_conv} and \eqref{eq:a_priori_z} imply 
\begin{align*}
    \mathbb{E}\left[\log\left(1+\int_0^T \|R^\epsilon_t\|dt\right)\right]&\lesssim \epsilon^{3\H-1}.
\end{align*}
This is exactly \eqref{eq:vanishing_R}.
    \subsection{Time Compactness}

Up to some minor technical modifications, this part of the argument closely follows \cite[Section 8.1–8.2]{debussche2023rough}. Therefore, we only sketch the main steps of the proof and refer the reader to \cite[Section 8.1–8.2]{debussche2023rough} for all omitted details. \\
Recall that $\frac1\H\in[N,N+1)\cap(1,3)$ and $p\in (\frac{1}{\H},N+1)$. The ultimate objective of this section is to prove \eqref{EstOnPvar} and \eqref{EstOnPvar_cent}. Therefore, we need to analyse the behaviour of
\begin{equation*}
    \omega_{u^{\epsilon}}(s,t):=\norm{ u^{\epsilon}}^p_{p\var,[s,t],H^{-\xi+1}}.
\end{equation*}

We begin by proving \eqref{EstOnPvar}; the centered case can be established by an identical argument. Accordingly, we will indicate at the end which minor and straightforward modifications are required to derive \eqref{EstOnPvar_cent}.

In light of the a priori bounds from the previous subsection together with \autoref{proproughHgreat12} and \autoref{cor:beta12}, we do not need to separate the cases $\alpha=1$ and $\alpha=\frac{1}{2}+\H.$ \\
We begin by examining the increments of $u^{\epsilon,\natural}$. We would identify a uniform control in $\epsilon$ over
\begin{equation*}
    \omega_{\epsilon,\natural}(s,t)=|u^{\epsilon,\natural}|_{p/(N+1)\var,[s,t],H^{-3}}^{p/(N+1)}.
\end{equation*}

Let us recall some estimates on the drift terms 
\begin{equation*}
  \mu^{\epsilon,1}=\int_s^t \Delta u^\epsilon_\theta+b(u_\theta^\epsilon,u^\epsilon_\theta)\,d\theta,\quad\mu^{\epsilon,2}=\int_s^tb(r^\epsilon_\theta,u^\epsilon_\theta)\,d\theta.  
\end{equation*}
Indeed,
\begin{align}\label{estimate_first_variation}
\norm{\delta\mu^{\epsilon,1}_{st}}_{H^{-1}}& \lesssim \int_s^t 1+\norm{\nabla u^{\epsilon}_\tau}^2 d\tau,\notag\\ \quad \norm{\delta\mu^{\epsilon,1}_{st}}_{H^{-\xi+1}} &\lesssim|t-s|\left(1+\norm{u^{\epsilon}_0}\right)^2.
\end{align}
    Moreover for each $\phi \in L^\infty_x$ it holds
    \begin{align}\label{estimate_second_variation}
        \| \delta\mu^{\epsilon,2}_{st}(\phi)\|\lesssim \norm{\phi}_{L^{\infty}}\left(\int_s^t \norm{r^{\epsilon}_\tau}^2d\tau\right)^{1/2}\left(\int_s^t \norm{\nabla u^{\epsilon}_\tau}^2d\tau\right)^{1/2}.
    \end{align}
    Therefore, denoting by $\omega_{\mu^{\epsilon}}(s,t)$ the control associated to the $C^{1-var}([0,T];H^{-\xi+2})$ norm of $\mu^{\epsilon}=\mu^{\epsilon,1}+\mu^{\epsilon,2}$ we have
    \begin{align*}
        \omega_{\mu^{\epsilon}}(s,t)=\int_s^t 1+\norm{\nabla u^{\epsilon}_\tau}^2 d\tau+\left(\int_s^t \norm{r^{\epsilon}_\tau}^2d\tau\right)^{1/2}\left(\int_s^t \norm{\nabla u^{\epsilon}_\tau}^2d\tau\right)^{1/2}.
    \end{align*}
    In particular, \eqref{EnergyEst} and \eqref{EstOnR} imply that for each $\alpha\leq 2$ and $C>0$ we have  
\begin{align}\label{stimalog1}
    \sup_{\epsilon\in (0,1)}\mathbb{E}\left[\log\left(1+C\omega_{\mu^{\epsilon}}(0,T)^{\alpha}\right)\right] \leq \sup_{\epsilon\in (0,1)}\mathbb{E}\left[\log\left(C_1\left(1+\left(\int_0^T\| r^{\epsilon}_\tau\|^2 d\tau\right)^{\alpha/2}\right)\right)\right]\\ \leq \log(C_1)+\sup_{\epsilon\in (0,1)}\mathbb{E}\left[\log\left(1+\int_0^T\| r^{\epsilon}_\tau\|^2 d\tau\right)\right]<+\infty ,\label{estimate_log_increments}
\end{align}
where $C_1$ is a positive constant depending on $\alpha,\,T$ and $\sup_{\epsilon\in (0,1)}\|u^{\epsilon}_0\|<+\infty.$
Before we proceed to estimate $\omega_{\epsilon,\natural}(s,t)$, we first recall that the condition $\xi>\frac{d}{2}+2$ (see \cite{debussche2023rough}) ensures that
\begin{align}\label{stimaExtraURD}
    \|\mathbb{A}^{1,\epsilon}\st\phi\|_{L^\infty}&\lesssim \|\phi\|_{3}\|\mathbb{Y}^{\epsilon,1}\st\|_{H^\xi}\le\|\phi\|_{H^3}\,\omega_\epsilon(s,t)^{1/p},\\\|\mathbb{A}^{2,\epsilon}\st\phi\|_{L^\infty}&\lesssim \|\phi\|_{H^\xi}\|\mathbb{Y}^{\epsilon,2}\|_{H^\xi\otimes H^\xi}\le\|\phi\|_{H^\xi }\,\omega_\epsilon(s,t)^{2/p},
\end{align}
where $\omega_\epsilon(s,t)=\|\mathbb{Y}^{\epsilon,1}\|_{p\var,[s,t]}^{p}+\uno_{\{N=2\}}\|\mathbb{Y}^{\epsilon,2}\|_{p/2\var,[s,t]}^{p/2}$.

Now, let $(J^{\eta})_{\eta\in(0,1]}$ be the smoothing operators introduced at \cite{hofmanova2019navier}. Then for $s<\theta<t$, we have that for $\phi \in H^3$
\begin{align*}
    |\delta u^{\epsilon,\natural}_{s,\theta,t}(\phi)|&\le |\sx[\sx(\A^{1,\epsilon}_{\theta,t} \sx(u_{s,\theta}^{\epsilon}-\uno_{\{N=2\}}\A^{1,\epsilon}_{s,\theta}u^{\epsilon}_s\dx)\dx)+\uno_{\{N=2\}}\A^{2,\epsilon}_{\theta,t}u^{\epsilon}_{s,\theta}\dx](J^{\eta}\phi)|\\
        &+|\sx[\sx(\A^{1,\epsilon}_{\theta,t} \sx(u_{s,\theta}^{\epsilon}-\uno_{\{N=2\}}\A^{1,\epsilon}_{s,\theta}u^{\epsilon}_s\dx)\dx)+\uno_{\{N=2\}}\A^{2,\epsilon}_{\theta,t}u^{\epsilon}_{s,\theta}\dx](\phi-J^{\eta}\phi)|\\
        &\lesssim \omega_\epsilon^{1/p}(s,t)\eta\sx(\eta+\uno_{\{N=2\}}\omega_\epsilon^{1/p}(s,t)\dx)\|\phi\|_3+|\mu^{\epsilon}_{s,\theta}(\A^{1,\epsilon,*}_{\theta,t}J^\eta\phi)|\\
        &+|\sx[\A^{1,\epsilon}_{s,\theta}u_s^\epsilon-\uno_{\{N=2\}}\A^{1,\epsilon}_{s,\theta}u^{\epsilon}_s+\uno_{\{N=2\}}\A^{2,\epsilon}_{s,\theta}u^{\epsilon}_{s}
        +u^{\epsilon,\natural}_{s,\theta}\dx](\A^{1,\epsilon,*}_{\theta,t}J^\eta\phi)|\\
        &+\uno_{\{N=2\}}|\sx[\mu^{\epsilon}_{s,\theta}+\A^{1,\epsilon}_{s,\theta}u_s^\epsilon+\A^{2,\epsilon}_{s,\theta}u^{\epsilon}_{s}
        +u^{\epsilon,\natural}_{s,\theta}\dx](\A^{2,\epsilon,*}_{\theta,t}J^\eta\phi)|.
\end{align*}
Therefore, 
\begin{align*}
        \frac{|\delta u^{\epsilon,\natural}_{s,\theta,t}(\phi)|}{\|\phi\|_{H^3}}&\lesssim \omega_\epsilon(s,t)^{1/p}\eta^2+\omega_{\epsilon,\natural}^{\frac{N+1}p}\omega_\epsilon^{1/p}(s,t)\eta^{-1}+\omega_\epsilon^{1/p}\omega_{\mu^\epsilon}(s,t)+\uno_{\{N=1\}}\omega_{\epsilon}^{\frac{N+1}{p}}(s,t)\\
    &+\uno_{\{N=2\}}\big[\omega_\epsilon(s,t)^{2/p}\eta+\omega_\epsilon(s,t)^{3/p}+\omega_\epsilon^{2/p}\omega_{\mu^\epsilon}(s,t)+\omega_{\mu^\epsilon}\omega_{\epsilon}^{2/p}(s,t)\eta^{3-\xi_0}\\
    &+\omega_\epsilon(s,t)^{4/p}\eta^{-1}+\omega_{\epsilon,\natural}^{\frac{N+1}{p}}\omega_\epsilon^{2/p}(s,t)\eta^{-2}\big],
\end{align*}
where $\xi_0\in (d/2+2,4\wedge \xi)$, while we have used \eqref{estimate_first_variation}, \eqref{estimate_second_variation} and \eqref{stimaExtraURD}. Let $\eta=\omega_\epsilon(s,t)^{1/p}\lambda$ and  consider $\lambda$ s.t. $c(\lambda^{-1}+\lambda^{-2})\le1/2$, where $c$ denotes the implicit constant of the previous inequality. Then for all $(s,t)$ s.t. $\omega_\epsilon^{1/p}(s,t)\le\frac{1}{2\lambda}$, \cite[Corollary B.2]{hofmanova2019navier} implies that
\begin{align}\label{stimaRemainder}
    \omega_{\epsilon,\natural}(s,t)&\lesssim \omega_\epsilon(s,t)+\omega_\epsilon^{\frac{1}{N+1}}\omega_{\mu^\epsilon}^{\frac{p}{N+1}}(s,t)\notag\\
    &+\uno_{\{N=2\}}\Big[\omega_\epsilon^{\frac{2}{N+1}}\omega_{\mu^\epsilon}^{\frac{p}{N+1}}
    +\omega_{\mu^\epsilon}^{\frac{p}{N+1}}\omega_\epsilon^{\frac{5-\xi_0}{N+1}}
    \Big](s,t)\notag\\
    &\lesssim \omega_{\epsilon}(s,t)+\omega_{\mu^\epsilon}(s,t)^{p/N},
\end{align}
where the last inequality follows from assuming, without loss of generality, that $\lambda>1$.
Now, we can estimate $\omega_{u^\epsilon}(s,t)$. Notice that, differently from \cite{debussche2023rough}, we study the $p$-variation norm in $H^{-\xi+1}$ instead of $H^{-\xi+2}$. This difference allows us to treat $N=1$ and $N=2$ simultaneously. Let $\phi \in H^{\xi-1}$, then
\begin{align*}
       \frac{|u^\epsilon\st(\phi)|}{\|\phi\|_{H^{\xi-1}}}&\le \frac1{\|\phi\|_{H^{\xi-1}}}\sx[|u^\epsilon\st(J^\eta\phi)|+|u^\epsilon\st(\phi -J^\eta\phi)|\dx]\\
    &\lesssim \frac1{\|\phi\|_{H^{\xi-1}}}\sx[\|\phi\|_{H^{\xi-1}}\eta^{\xi-1}+|\sx(\mu^{\epsilon}\st+\A^{1,\epsilon}\st u^\epsilon_s+\uno_{\{N=2\}}\A^{2,\epsilon}\st u^\epsilon_s+u^{\epsilon,\natural}\st\dx)(J^\eta \phi)|\dx]\\
    &\lesssim \eta^{\xi-1}+\omega_{\mu^{\epsilon}}(s,t)+\omega_\epsilon(s,t)^{1/p}+\uno_{\{N=2\}}\omega_\epsilon(s,t)^{2/p}+\omega_{\epsilon,\natural}(s,t)^{\frac{N+1}{p}}\eta^{(\xi-4)\wedge 0}.
\end{align*}
Let $\eta=(\omega_\epsilon(s,t)^{1/p}+\omega_{\epsilon,\natural}(s,t)^{1/p})\lambda$. Then for all $(s,t)$ such that $\omega_{\epsilon}(s,t)^{1/p}+\omega_{\mu^\epsilon}(s,t)^{1/N}\le \frac{1}{4\lambda}$, it holds $\eta<1$ thanks to \eqref{stimaRemainder}.
Therefore, 
\begin{align*}
        \|u^{\epsilon}\st\|_{p\var,[s,t]}^p&\lesssim\omega_{\mu^{\epsilon}}(s,t)^p+\omega_\epsilon(s,t)+\omega_{\epsilon,\natural}(s,t)\\
    &\lesssim \omega_\epsilon(s,t)+\omega_{\mu^\epsilon}(s,t)^{p/N}
\end{align*}
for all $(s,t)$ such that $\omega_{\epsilon}(s,t)^{1/p}+\omega_{\mu^\epsilon}(s,t)^{1/N}\le \frac{1}{4\lambda}$, where the last inequality is due to \eqref{stimaRemainder}. Then, \cite[Proposition 2.5]{roveri2024well} implies that
\begin{equation}
    \|u^{\epsilon}\|_{p\var,[0,T]}^p\lesssim (\omega_{\epsilon}(0,T)+\omega_{\mu^\epsilon}(0,T)^{p/N})^p.
\end{equation}
Therefore, thanks to \eqref{est:UniformRegularityURDH>1/2}, \eqref{est:UniformRegularityURDH<1/2},\eqref{est:UniformRegularityURDH<1/2zeros} and \eqref{stimalog1}, we can estimate for $\kappa>0$ sufficiently small and uniformly in $\epsilon$

\begin{align*}
&\mathbb{E}\left[\log\left(1+\norm{u^{\epsilon}}^\kappa_{p-var,[0,T],H^{-\xi+1}}\right)\right]\\  & \leq \mathbb{E}\left[\log\left(1+\omega_\epsilon(0,T)^\kappa\right)\right]+\mathbb{E}\left[\log\left(1+\omega_{\mu^\epsilon}(0,T)^{\frac{\kappa p}{N}}\right)\right]< \infty,
\end{align*}
yielding \eqref{EstOnPvar}.

At last, \eqref{EstOnPvar_cent} follows by just replacing $\mathbf{A}^\epsilon$ with $\hat{\mathbf{A}}^\epsilon$, considering \[\mu^{\epsilon,2}=\int_s^tb(R^\epsilon_\theta, u^\epsilon_\theta)\,d\theta\]
and 
\[\omega_{\mu^{\epsilon}}(s,t)=\int_s^t 1+\norm{\nabla u^{\epsilon}_\tau}^2 d\tau+\left(\int_s^t \norm{R^{\epsilon}_\tau}^2d\tau\right)^{1/2}\left(\int_s^t \norm{\nabla u^{\epsilon}_\tau}^2d\tau\right)^{1/2}.\]
Then, the argument follows verbatim exploiting \autoref{second_prop_rough_pathless12} and \eqref{EstOnR_cent} in place of \autoref{proproughHgreat12}, \autoref{cor:beta12} and \eqref{EstOnR}.

\section{Compactness and passage to the limit}\label{sec:end_proof}
This last section will be focused on the proof of \autoref{main_thm} and \autoref{centered_noise}. 
In particular, we will first focus on the two convergence results in the case $\H<1/2$. Later on, we will conclude the proof of \autoref{main_thm} by taking into account the regime $\H>1/2$. 
\subsection*{Case $\frac{1}{3}<\H<\frac{1}{2}$}
We treat the cases $\alpha=\frac{1}{2}+\H$ and $\alpha=1$ by the same compactness
argument. We present this step using the notation of the former
regime. In the latter, the replacements
\begin{align*}
    r^\epsilon \to R^\epsilon,
    \qquad
    \mathbf{Y}^\epsilon \to \hat{\mathbf{Y}}^\epsilon
\end{align*}
are understood throughout this step, with the corresponding replacements
for their copies and limits. The a priori estimates and the unbounded rough drivers convergence are provided by
\autoref{lemma_a_priori} and \autoref{cor:beta12} in the case $\alpha=\frac{1}{2}+\H$, and by \autoref{lemma_a_priori_centered} and
\autoref{second_prop_rough_pathless12} in the centered regime.
Combined with the compactness criterion
\cite[Lemma A.2]{hofmanova2019navier}, they imply that the family
\begin{align*}
    (u^{\epsilon},r^{\epsilon}, W^\H,\mathbb{Y}^{\epsilon,1},\mathbb{Y}^{\epsilon,2})
\end{align*}
is tight in the space
\begin{align*}
    \mathcal{X}&=\mathcal{X}_{u,r,W}\times\mathcal{X}_{RP},\\
    \mathcal{X}_{u,r,W}&=(C_w([0,T];H)\cap L^2(0,T;H))\times L^2((0,T)\times \T^d)_w\times C([0,T];H),\\
    \mathcal{X}_{RP}&= C^{p\var}([0,T];H^{\xi})\times C^{p/2\var}_2([0,T];(H^{\xi})^{\otimes 2}),
\end{align*}
for $p\in (\frac1\H,3)$. Therefore, by applying the Jakubowski–Skorokhod representation theorem (see, for instance, \cite[Section 2.7]{breit2018stochastically}), we can extract a subsequence of $\epsilon_k \to 0$, which we still denote by $\epsilon$ for simplicity, and construct an auxiliary probability space $\left(\overline{\Omega},\overline{\mathcal{F}},\overline{\mathbb{P}}\right)$ carrying $\mathcal{X}$-valued random variables $\sx(\overline{u}^{\epsilon},\overline{r}^{\epsilon},\overline{W}^{\H,\epsilon},\overline{\mathbb{Y}}^{\epsilon,1},\overline{\mathbb{Y}}^{\epsilon,2}\dx)$ and $\sx(\overline{u},\overline{r},\overline{W}^{\H},\overline{\mathbb{Y}}^{1},\overline{\mathbb{Y}}^{2}\dx)$ such that
\begin{align*}
  \sx(\overline{u}^{\epsilon},\overline{r}^{\epsilon},\overline{W}^{\H,\epsilon},\overline{\mathbb{Y}}^{\epsilon,1},\overline{\mathbb{Y}}^{\epsilon,2}\dx) &\stackrel{\mathcal{X}}{\longrightarrow} \sx(\overline{u},\overline{r},\overline{W}^{\H},\overline{\mathbb{Y}}^{1},\overline{\mathbb{Y}}^{2}\dx) \quad \overline{\mathbb{P}}\text{-a.s.}
\end{align*}
On this new probability space, the copied processes satisfy the
fourth item of \autoref{def_rough_sol}, or
\autoref{def_rough_sol_2} in the centered regime, with remainder $\overline{u}^{\epsilon,\natural}$ defined by
\begin{equation*}
    \overline{u}^{\epsilon,\natural}\st:=\overline{u}^\epsilon\st-\int_s^t \Delta \overline{u}^\epsilon_\theta+b(\overline{u}_\theta^\epsilon+\overline{r}^\epsilon_\theta,\overline{u}^\epsilon_\theta)\,d\theta-\sum_{k=1}^N \overline{\mathbb{A}}^{k,\epsilon}\st \overline{u}_s^{\epsilon},
\end{equation*}
where $\overline{A}^{k,\epsilon}\st\phi=b^{(k)}\sx(\overline{\mathbb{Y}}^{k,\epsilon}\st,\phi\dx).$
For a detailed argument identifying the rough path $\sx(\overline{\mathbb{Y}}^{\epsilon,1},\overline{\mathbb{Y}}^{\epsilon,2}\dx)$ after the change of probability space we refer to Step 2 in the proof of \cite[Proposition 15]{flandoli2022global}.

The remaining terms in the rough formulation pass to the limit
by standard arguments, using in particular the convergence of
the driving rough paths established above. Moreover, in the centered regime $\alpha=1$, \autoref{lemma_a_priori_centered}
already identifies the limit of the remainder:
\begin{align*}
    \overline{R}=0
    \qquad \overline{\mathbb{P}}\text{-a.s.}
\end{align*}
To finish the convergence proof, we must determine $\overline{r}=\lim_{\epsilon\to 0}r^{\epsilon}$ in case of $\alpha=\frac{1}{2}+\H$ and verify that in both cases $\lim_{\epsilon\to 0}\overline{u}^{\epsilon,\natural}=:\overline{u}^{\natural}\in C^{p/3\var}_{2,loc}(H^{-3})$. Indeed, all the other terms appearing in the rough formulation, given in \autoref{def_rough_sol}, will converge by standard arguments. In particular, \autoref{cor:beta12} and \autoref{second_prop_rough_pathless12} imply the convergence of the rough paths and, therefore, of the terms  $\overline{\A}^{k,\epsilon}\st \overline{u}^\epsilon_s.$

Before identifying $\overline{r}$, let us stress that due to the convergence of $\overline{r}^{\epsilon}$, for $\overline{\mathbb{P}}-a.s.$ $\omega\in \overline{\Omega}$ there exists a constant $N=N(\omega)$ such that
\begin{align}\label{pathwise_estimate}
    \int_0^T \norm{\overline{r}^{\epsilon}_{\tau}}^2 d\tau \leq N(\omega).
\end{align}
\begin{lemma}\label{identification_lemma}
 Let $\alpha=\frac{1}{2}+\H$. Then, it holds
    \begin{align*}
        \overline{r}=\int (-C)^{-1}b(r,r)d\mu(r),
    \end{align*}
    where $\mu$ is the centered Gaussian measure on $L^2$ with covariance operator 
    \begin{align*}
        \overline{Q}=\sum_{i\geq 1} \frac{\sigma_i^2 e_i\otimes e_i}{2|\lambda_i|^{2\H}}\int_0^{+\infty}e^{-s} s^{2\H} ds. 
    \end{align*}
   
\end{lemma}
\begin{proof}
    Let us call by $\overline{w}^{\epsilon}$ the solution of 
    \begin{align*}
d\overline{w}^{\epsilon}=\epsilon^{-1}C\overline{w}^{\epsilon}\,dt+\epsilon^{- \H}d\overline{W}^{\H,\epsilon}_t
    \end{align*}
    and just for the matter of notation 
    \begin{align*}
        \tilde{r}=\int (-C)^{-1}b(r,r)d\mu(r).
    \end{align*}
    Due to \autoref{def_rough_sol}, for each $\psi\in C^{1}_c(0,T)$ it holds 
    \begin{align*}
        \int_0^T \psi_t \overline{r}^{\epsilon}_t dt&=\int_0^T \psi_t(-C)^{-1}b(\overline{w}^{\epsilon}_t,\overline{w}^{\epsilon}_t)dt\\ & +\epsilon^{1/2}\int_0^T\psi_t (-C)^{-1}\left(A\overline{w}^{\epsilon}_t+b(\overline{u}^{\epsilon}_t,\overline{w}^{\epsilon}_t)+b(\overline{w}^{\epsilon}_t,\overline{r}^{\epsilon}_t)+b(\overline{r}^{\epsilon}_t,\overline{w}^{\epsilon}_t)\right)dt\\ & +\epsilon\int_0^T\psi_t(-C)^{-1}\left(A\overline{r}^{\epsilon}_t+b(\overline{u}^{\epsilon}_t,\overline{r}^{\epsilon}_t)+b(\overline{r}^{\epsilon}_t,\overline{r}^{\epsilon}_t)\right)dt+\epsilon\int_0^T \partial_t \psi_t (-C)^{-1}\overline{r}^{\epsilon}_tdt\\ & =I_1^{\epsilon}+I_2^{\epsilon}+I_3^{\epsilon}+I_4^{\epsilon}.
    \end{align*}
    We will show that, up to subsequences, $I_1^{\epsilon}\rightarrow \int_0^T \psi \tilde{r} dt\quad \overline{\mathbb{P}}-a.s$ in $H^{-\xi}$ and all the other terms approach $0$ in $H^{-\xi}$, $\overline{\mathbb{P}}-a.s.$ This is enough to conclude by uniqueness of the limit. Due to \eqref{pathwise_estimate} and \eqref{EstOnR} we get immediately
    \begin{align*}
        \norm{I^{\epsilon}_3}_{H^{-\xi}}& \lesssim \epsilon\int_0^T \norm{\overline{r}^{\epsilon}_t}+\norm{\overline{r}^{\epsilon}_t}^2 dt\rightarrow 0\quad \overline{\mathbb{P}}-a.s.\\
        \norm{I^{\epsilon}_4}_{H^{-\xi}}& \lesssim\epsilon \int_0^T \norm{\overline{r}^{\epsilon}_t}dt \rightarrow 0\quad \overline{\mathbb{P}}-a.s.
    \end{align*}
    Concerning $I^{\epsilon}_2$, we split it in \begin{align*}
        I_{2,1}^{\epsilon}&=\epsilon^{1/2}\int_0^T\psi_t (-C)^{-1}\left(A\overline{w}^{\epsilon}_t+b(\overline{u}^{\epsilon}_t,\overline{w}^{\epsilon}_t)\right)dt,\\ I_{2,2}^{\epsilon}&=\epsilon^{1/2}\int_0^T\psi_t (-C)^{-1}\left(b(\overline{w}^{\epsilon}_t,\overline{r}^{\epsilon}_t)+b(\overline{r}^{\epsilon}_t,\overline{w}^{\epsilon}_t)\right)dt.
    \end{align*}
    Concerning the first one, due to \eqref{energy_estimate} and \autoref{hp_stoch_conv}
    we have
    \begin{align*}
        \mathbb{E}\left[\norm{I_{2,1}^{\epsilon}}_{H^{-\xi}}\right]& \lesssim \epsilon^{1/2}\sup_{t\in [0,T]}\mathbb{E}\left[\norm{\overline{w}^{\epsilon}_t}\right]\rightarrow 0.
    \end{align*}
    Concerning the second one, there exists $C>0$ such that for each  $\delta,\epsilon>0$ 
    \begin{align*}
        \overline{\mathbb{P}}\left(\norm{I^{\epsilon}_{2,2}}_{H^{-\xi}}>\delta\right)& \leq 2\overline{\mathbb{P}}\left(\int_0^T \norm{\overline{w}^{\epsilon}_t}\norm{\overline{r}^{\epsilon}_t}dt\gtrsim\frac{\delta}{\epsilon^{1/2}}\right) \\ & \leq \frac{2}{\log\left(1+\frac{C\delta}{\epsilon^{1/2}}\right)}\mathbb{E}\left[\log\left(1+\int_0^T\norm{\overline{w}^{\epsilon}_t}\norm{\overline{r}^{\epsilon}_t}dt\right)\right]\\ & \lesssim \frac{1+\sup_{t\in [0,T]}\mathbb{E}\left[\norm{\overline{w}^{\epsilon}_t}^2\right]+\mathbb{E}\left[\log\left(1+\int_0^T \norm{\overline{r}^{\epsilon}_t}^2\,dt\right)\right]}{\log\left(1+\frac{C\delta}{\epsilon^{1/2}}\right)}.
    \end{align*}
    Therefore $I^{\epsilon}_{2,2}\rightarrow 0$ in probability in $H^{-\xi}.$ Lastly let us observe that $\overline{w}^{\epsilon}_{\epsilon t}=\tilde{w}^{\epsilon}_t$ for $t\in [0,\epsilon^{-1}T]$ where $\tilde{w}^{\epsilon}_t$ solves
    \begin{align*}
        d\tilde{w}^{\epsilon}=C\tilde{w}^{\epsilon}_tdt+d\overline{W}^{\H,\epsilon}_t.
    \end{align*}
    Therefore it holds
    \begin{align*}
        I^{\epsilon}_1=-\epsilon\int_0^T \partial_t \psi \int_0^{\epsilon^{-1}t}(-C)^{-1}b(\tilde{w}^{\epsilon}_s,\tilde{w}^{\epsilon}_s)dsdt.
    \end{align*}
   Due to \autoref{lem_ito_stokes}, for each $t\in [0,T]$ it holds
    \begin{align*}
    \epsilon \int_0^{\epsilon^{-1}t}(-C)^{-1}b(\tilde{w}^{\epsilon}_s,\tilde{w}^{\epsilon}_s)ds\rightarrow t \tilde{r} \quad \text{in }L^2(\Omega;H).   
    \end{align*}
    Since in addition by Minkowski's inequality
    \begin{align*}
       \norm{\epsilon \int_0^{\epsilon^{-1}t}(-C)^{-1}b(\tilde{w}^{\epsilon}_s,\tilde{w}^{\epsilon}_s)ds}_{L^2(\Omega;H)} & \lesssim \int_0^T \mathbb{E}\left[\norm{\overline{w}^{\epsilon}_s}_{H^{\xi}}^4\right]^{1/2}ds\lesssim 1
    \end{align*} 
    uniformly in $\epsilon$ due to \autoref{hp_stoch_conv}, by dominated convergence theorem
    $I_1^{\epsilon}\rightarrow \int_0^T \psi_t \tilde{r}_t dt$ in $L^2(\Omega;H)$ and the claim follows.
\end{proof}

Since we have been able to pass to the limit in all the terms except the remainder, also $\overline{u}^{\epsilon,\natural}$ converges $\overline{\mathbb{P}}-a.s$ to some object $\overline{u}^{\natural}\st$ such that, for $\alpha=\frac{1}{2}+\H$, 
\begin{align}\label{eqUtile}
    \overline{u}_{\st}=\int_s^t \Delta \overline{u}_\tau+b(\overline{u}_{\tau}+\overline{r}_\tau,\overline{u}_\tau)\,d\tau+\overline{u}^{\natural}_{st} \quad \mathbb{\overline{P}}-a.s.,
\end{align}
while for $\alpha=1$,
\begin{align}
    \overline{u}_{\st}=\int_s^t \Delta \overline{u}_\tau+b(\overline{u}_{\tau},\overline{u}_\tau)\,d\tau+\sum_{k=1}^2 \overline{\mathbb{A}}^{k}\st \overline{u}_s+\overline{u}^{\natural}_{st} \quad \mathbb{\overline{P}}-a.s..
\end{align} 
Let us conclude by proving that $\overline{u}^{\natural}_{st}\in C^{p/3\var}_{2,loc}(H^{-3}).$ Since the rough paths converge and, in particular, \eqref{est:UniformRegularity2<1/2} and \eqref{est:UniformRegularityURDH<1/2zeros} hold, then we can identify a uniform in $\epsilon$ random covering $\{J_k\}_k$ of $[0,T]$ by introducing the localization \begin{align*}
    \sup_{\epsilon\in (0,1)}\omega_\epsilon^{1/p}(s,t)\le\frac{1}{2\lambda}.
\end{align*}
Therefore thanks to 
\eqref{stimaRemainder}, we have that on any $J_k$, for all $s,t \in J_k$ and $\phi \in H^3$, the following inequalities hold
  \begin{align}  &|\overline{u}^{\natural}\st(\phi)|=\lim_{\epsilon\to 0}|\overline{u}^{\epsilon,\natural}\st(\phi)|\le\|\phi\|_{H^3}\liminf_{\epsilon\to 0}{\|\overline{u}^{\epsilon,\natural}\|_{p/3\var,[s,t]}},\\
  &\|\overline{u}^{\epsilon,\natural}\|_{p/3\var,[s,t]}\lesssim\sx(\omega_{\epsilon}(s,t)+\omega_{\mu^\epsilon}(s,t)^{p/2}\dx)^{3/p}\lesssim  M(\omega),
  \end{align}
where $M(\omega)$ is a random constant that depends on \eqref{pathwise_estimate}.
Then for any $\mathbf{P}$ partition of $J_k$, we have that
\begin{equation}
    \sum_{(s,t)\in P}\|\overline{u}^{\natural}\st\|_{H^{-3}}^{p/3}\lesssim \liminf_{\epsilon\to0}\sum_{(s,t)\in P}{\|\overline{u}^{\epsilon,\natural}\st\|_{H^{-3}}^{p/3}}\lesssim \liminf_{\epsilon\to0} \|\overline{u}^{\epsilon,\natural}\|_{p/3\var,J_k}^{p/3}< \infty.
\end{equation}
Therefore, we have shown that $\overline{u}^{\natural}_{st}\in C^{p/3\var}_{2,loc}(H^{-3}).$ At last, we stress that in case of $\alpha=\frac12+\H$ we have that $\overline{u}^{\natural}\equiv0$ since it is highly regular in time and it is given by the increments of a $1$-index function of time, see \eqref{eqUtile}.
\subsection*{Case $\H>\frac{1}{2}$} 
The convergence in this final case is obtained by closely mirroring the proof from the preceding case.
The a priori estimates provided by \autoref{lemma_a_priori} combined with \autoref{proproughHgreat12} and the compactness criterion \cite[Lemma A.2]{hofmanova2019navier} imply that, in the case of \autoref{main_thm}, the family 
\begin{align*}
    (u^{\epsilon},r^{\epsilon}, W^\H,\mathbb{Y}^{\epsilon,1})
\end{align*}
is tight in the space
\begin{align*}
    \mathcal{X}&=\mathcal{X}_{u,r,W}\times\mathcal{X}_{RP},\\
    \mathcal{X}_{u,r,W}&=(C_w([0,T];H)\cap L^2(0,T;H))\times L^2((0,T)\times \T^d)_w\times C([0,T];H),\\
    \mathcal{X}_{RP}&= C^{p\var}([0,T];H^{\xi}).
\end{align*}
 Therefore,  by Jakubowski–Skorokhod’s representation theorem, see for example \cite[Section 2.7]{breit2018stochastically}, there exists a subsequence of $\epsilon_k\rightarrow 0$ which for the matter of notation we continue to denote by $\epsilon$ and an auxiliary probability space $\left(\overline{\Omega},\overline{\mathcal{F}},\overline{\mathbb{P}}\right)$ with $\mathcal{X}$ valued random variables $\sx(\overline{u}^{\epsilon},\overline{r}^{\epsilon}, \overline{W}^{\H,\epsilon},\overline{\mathbb{Y}}^{\epsilon,1}\dx)$ and $\sx(\overline{u},\overline{r}, \overline{W}^{\H},\overline{\mathbb{Y}}^{1}\dx)$ such that 
\begin{align*}
  \sx(\overline{u}^{\epsilon},\overline{r}^{\epsilon}, \overline{W}^{\H,\epsilon},\overline{\mathbb{Y}}^{\epsilon,1}\dx)\stackrel{\mathcal{X}}{\rightarrow}\sx(\overline{u},\overline{r}, \overline{W}^{\H},\overline{\mathbb{Y}}^{1}\dx)\quad \overline{\mathbb{P}}-a.s..
\end{align*}
The convergence of all the terms as well as showing that the remainder $\overline{u}^{\natural}\in C^{p/2\var}_{2,loc}(H^{-3})$ can be justified similarly to the previous case. Therefore, we need to identify $\overline{r}$ to conclude the proof of \autoref{main_thm}.
\begin{lemma}\label{identification_lemma_smoothcase}
    It holds
    \begin{align*}
        \overline{r}=0.
    \end{align*}
\end{lemma}
\begin{proof}
    Let us call by $\overline{w}^{\epsilon}$ the solution of 
    \begin{align*}
d\overline{w}^{\epsilon}=\epsilon^{-1}C\overline{w}^{\epsilon}\,dt+\epsilon^{- \H}d\overline{W}^{\H,\epsilon}_t.
    \end{align*}
    For each $\psi\in C^{1}_c(0,T)$ it holds due to \autoref{def_rough_sol}
    \begin{align*}
        \int_0^T \psi_t \overline{r}^{\epsilon}_t dt&=\epsilon^{2\H-1}\int_0^T \psi_t(-C)^{-1}b(\overline{w}^{\epsilon}_t,\overline{w}^{\epsilon}_t)dt\\ & +\epsilon^{\H}\int_0^T\psi_t (-C)^{-1}\left(A\overline{w}^{\epsilon}_t+b(\overline{u}^{\epsilon}_t,\overline{w}^{\epsilon}_t)+b(\overline{w}^{\epsilon}_t,\overline{r}^{\epsilon}_t)+b(\overline{r}^{\epsilon}_t,\overline{w}^{\epsilon}_t)\right)dt\\ & +\epsilon\int_0^T\psi(-C)^{-1}\left(A\overline{r}^{\epsilon}_t+b(\overline{u}^{\epsilon}_t,\overline{r}^{\epsilon}_t)+b(\overline{r}^{\epsilon}_t,\overline{r}^{\epsilon}_t)\right)dt+\epsilon\int_0^T \partial_t \psi_t (-C)^{-1}\overline{r}^{\epsilon}_t dt\\ & =J_1^{\epsilon}+J_2^{\epsilon}+J_3^{\epsilon}+J_4^{\epsilon}.
    \end{align*}
    We will show that, up to subsequences, all the terms approach $0$ in $H^{-\xi}$ $\overline{\mathbb{P}}-a.s.$ . This is enough to conclude by uniqueness of the limit. The treatment of $J_2^\epsilon,\ J_3^\epsilon,\ J_4^\epsilon$ follows verbatim as in the first case of \autoref{identification_lemma} and we omit details. To treat $J_1^\epsilon$ we argue in a simpler way compared to the first case of \autoref{identification_lemma}. Indeed since $\H>\frac{1}{2}$, it holds
    \begin{align*}
        \mathbb{E}\left[\|J_1^{\epsilon}\|_{H^{-\xi}}\right] \lesssim \epsilon^{2\H-1}\int_0^T\mathbb{E}\left[\|\overline{w}^{\epsilon}_t\|^2 \right]dt\rightarrow 0
    \end{align*}
    by \autoref{hp_stoch_conv}. This completes the proof.
\end{proof}
\begin{remark}
    For the sake of simplicity, we have reported estimates showing that $\overline{u}^{\natural}\in C^{p/2\var}_{2,loc}(H^{-3})$. Nevertheless, we stress that by refining \eqref{stimaExtraURD} we could have shown that 
    $\overline{u}^{\natural}\in C^{p/2\var}_{2,loc}(H^{-2})$ for $\H\in (1/2,2/3]$ and $\overline{u}^{\natural}\in C^{2p/(2p+1)\var}_{2,loc}(H^{-2})$ for $\H>2/3$. Therefore for all $\H\in (1/2,1)$, the integral term arising from the transport noise in the limiting equation can be constructed as a Young integral in $H^{-2}$ instead of $H^{-3}.$ 
\end{remark}

\section*{Declarations}
\subsection*{Acknowledgements}
FT acknowledges the warm hospitality of Universität Bielefeld, where this work started. Both authors thank Martina Hofmanová for valuable discussions at an early stage of the project. 
\subsection*{Funding} EL has received funding from the European Research Council (ERC) under the European Union’s Horizon 2020 research and innovation programme (grant agreement No. 949981), and from the Deutsche Forschungsgemeinschaft (DFG, German Research Foundation) under Germany´s Excellence Strategy – The Berlin Mathematics Research Center MATH+ (EXC-2046/2, project ID: 390685689).
FT is supported by the Istituto Nazionale di Alta Matematica (IN$\delta$AM) through the GNAMPA 2025 project “Modelli stocastici in Fluidodinamica e Turbolenza” (CUP E5324001950001), and by the project “Noise in fluid dynamics and related models" funded by the MUR Progetti di Ricerca
di Rilevante Interesse Nazionale (PRIN) Bando 2022 - grant 20222YRYSP.
\subsection*{Author Contributions} All authors contributed to the study conception and design. All authors read and approved the final manuscript.
\subsection*{Conflict of interest} On behalf of all authors, the corresponding author states that there is no conflict of interest.
\subsection*{Data Availability} Data sharing is not applicable to this article as no data sets were generated or analyzed during the current study.
\appendix 

\appendix
\section{Multiscale Model}\label{sec:appendix_multiscale}
Here we briefly recall, and adapt to our fractional framework, the arguments of \cite[Section 2]{flandoli20212d}.

We consider systems with three time scales. This statement should be understood as follows. There is a small time scale $\mathcal{T}_s$ at which one observes variations and fluctuations of the velocity fields. There is an intermediate time scale $\mathcal{T}_m$ at which these fluctuations appear random, though not yet as white noise; rather, they behave as random processes with a typical time of variation of order $\mathcal{T}_s$, which is small compared to $\mathcal{T}_m$. Finally, there is a large time scale $\mathcal{T}_l$ at which the cumulative effect of the small-scale fluctuations becomes non-negligible. We assume that these time scales satisfy the relation
\begin{align}\label{scaling_condition}
\frac{\mathcal{T}_s}{\mathcal{T}_m}=\frac{\mathcal{T}_m}{\mathcal{T}_l}.
\end{align}

As discussed in \cite{flandoli20212d}, an ideal example of this situation is provided by atmospheric flows over a large region, restricted to the lower atmospheric layer interacting with ground irregularities such as hills and mountains. We idealize such a fluid by means of the three-dimensional Euler equations with forcing, written in velocity form as
\begin{align}\label{eq:NS_1}
\begin{cases}
\partial_t u+u\cdot\nabla u+\nabla p=f,\\
\nabla\cdot u=0,
\end{cases}
\end{align}
where $f$ represents the generation of small-scale perturbations induced by ground irregularities, as discussed for instance in \cite[Chapter 5]{flandoli2023stochastic}. For long-time investigations, it would of course be necessary to include additional realistic terms, such as a small viscous dissipation $\nu\Delta u$ with $\nu\ll1$, in order to dissipate the energy injected by $f$, as in the model \eqref{e:equation1}. However, such refinements are not essential here, since our focus is solely on the scaling of the noise coefficients.

\subsection{Human Scale: Seconds} The human scale corresponds to observations made at distances measured in meters and over time intervals of seconds. Accordingly, we set $\mathcal{T}_s=1\mathrm{s}$. The velocity field $u(t,x)$ is measured in meters per second. We decompose the initial condition into large- and small-scale components according to some reasonable criterion,
\begin{align*}
    u(0)=u_l(0)+u_s(0).
\end{align*}
The small-scale component describes wind fluctuations over spatial distances of $1$–$10$ m, while the large-scale component captures structures of regional extent, ranging from $10$ to $1000$ km. Motivated by \eqref{eq:NS_1}, we split the evolution of the velocity field into the dynamics of the large and small scales as
\begin{align}\label{eq:NS_2}
    \begin{cases}
        \partial_t u_l+(u_s+u_l)\cdot\nabla u_l+\nabla p_l&=0,\\
        \partial_t u_s+(u_s+u_l)\cdot\nabla u_s+\nabla p_s&=f,\\
        \nabla\cdot u_s=\nabla\cdot u_l=0,
    \end{cases}
\end{align}
reflecting the assumption that the forcing $f$ primarily generates small-scale structures. We stress that \eqref{eq:NS_2} is a modeling assumption and cannot be rigorously derived from \eqref{eq:NS_1}.
\subsection{Intermediate scale: Minutes} 
If the same atmospheric flow is observed by a recording device with a temporal resolution of minutes, that is, $\mathcal{T}_m=1\mathrm{min}$, the small-scale fluctuations described above appear random, while retaining the same spatial scale of $1$–$10$ m. This motivates the modeling assumption that small scales can be described by a stochastic equation with memory, see also \cite{majda2001mathematical, franzke2015stochastic}. In particular, since time is now expressed in units of $\mathcal T_m$, the small-scale relaxation time in these units is
$\tau_m:=\mathcal T_s/\mathcal T_m$. This leads to the system
\begin{align}\label{eq_NS_3}
    \begin{cases}
        \partial_t u_l+(u_s+u_l)\cdot\nabla u_l+\nabla p_l&=0,\\
        \partial_t u_s+(u_s+u_l)\cdot\nabla u_s+\nabla p_s&=-\frac{1}{\tau_m}u_s+\frac{1}{\tau_m^\H}\Dot{W}^{\H},\\
        \nabla\cdot u_s=\nabla\cdot u_l=0.
    \end{cases}
\end{align}
Neglecting the Navier–Stokes dynamics at small scales, the process
\begin{align*}
    \tilde{u}_s(t)=\frac{1}{\tau_m^\H}\int_0^t e^{-\frac{(t-r)}{\tau_m}}dW^{\H}_r
\end{align*}
returns to equilibrium over a time of order $\tau_m$, and its intensity is independent of $\tau_m$.
\subsection{Regional scale: Hours} We now consider the same system, namely the lower atmospheric layer over a large region, observed at the regional scale, for instance by a satellite. Time is measured in hours, $\mathcal{T}_l=1\mathrm{h}$, and spatial units range from $10$ to $1000$ km, in contrast to the meter-scale resolution relevant at the human and intermediate levels. These are typical scales for weather prediction.

We rescale \eqref{eq_NS_3} accordingly. Variables at the intermediate scale are denoted by $(t,x)$, while those at the regional scale are denoted by $(T,X)$. Since minutes and hours differ by a factor of $60$, we set
\begin{align*}
    \varepsilon^{-1}=60,\quad t=\varepsilon^{-1}T.
\end{align*}
Similarly, we assume that spatial variables are related by
\begin{align*}
    x=\varepsilon_x^{-1}X.
\end{align*}
If $u$ denotes the velocity at the intermediate scale and $U$ the corresponding velocity at the regional scale, then
\begin{align*}
    U(T,X)=\varepsilon_x\varepsilon^{-1}u\left(T\varepsilon^{-1},X\varepsilon^{-1}_x\right).
\end{align*}
In these new coordinates, system \eqref{eq_NS_3} becomes
\begin{align*}
     \begin{cases}
        \partial_T U_l+(U_s+U_l)\cdot\nabla_X U_l+\nabla_X P_l&=0,\\
        \partial_T U_s+(U_s+U_l)\cdot\nabla_X U_s+\nabla_X P_s&=-\frac{\varepsilon^{-1}}{\tau_m}U_s+\frac{\varepsilon_x\varepsilon^{-2}}{\tau_m^\H}\Dot{W}^{\H}\left(\frac{T}{\varepsilon},\frac{X}{\varepsilon_x}\right),\\
        \nabla\cdot U_s=\nabla\cdot U_l=0.
    \end{cases}
\end{align*}
The scaling condition \eqref{scaling_condition}, together with the scaling properties of fractional Brownian motion, implies $\tau_m=\varepsilon$ and yields
\begin{align}\label{eq:NS_4}
     \begin{cases}
        \partial_T U_l+(U_s+U_l)\cdot\nabla_X U_l+\nabla_X P_l&=0,\\
        \partial_T U_s+(U_s+U_l)\cdot\nabla_X U_s+\nabla_X P_s&=-\frac{1}{\varepsilon^2}U_s+\frac{\varepsilon_x}{\varepsilon^{1+2\H}}\Dot{\tilde{W}}^{\H}\left(T,\frac{X}{\varepsilon_x}\right),\\
        \nabla\cdot U_s=\nabla\cdot U_l=0,
    \end{cases}    
\end{align}
for a new fractional Brownian motion $\tilde{W}^{\H}$. Up to the redefinition $\epsilon=\varepsilon^{2}$, system \eqref{eq:NS_4} corresponds to \eqref{system_prelimit} with the choice $\alpha=\frac{1}{2}+\H$.
\bibliography{demo}{}
\bibliographystyle{plain}

\end{document}